\documentclass[12pt,A4]{article}
\usepackage[french,english]{babel}
\usepackage[T1]{fontenc}
\usepackage[applemac]{inputenc}
\usepackage{amssymb,url,xspace,amsmath,amsthm,eucal,yhmath}
\usepackage{mathrsfs,eufrak,array,epsfig,psfrag,graphicx,dsfont,bbm,stmaryrd}
\usepackage[all]{xy}
\usepackage{hyperref}
\usepackage{graphicx}
\usepackage{verbatim}
\usepackage {cases} 
\usepackage {}

\newtheorem{thm}{Theorem}[section]
\newtheorem{lem}[thm]{Lemma}
\newtheorem{prop}[thm]{Proposition}
\newtheorem{cor}[thm]{Corollary}

\theoremstyle{definition}

\newtheorem{defn}[thm]{Definition}

\newtheorem{rem}[thm]{Remark}

\long\def\symbolfootnote[#1]#2{\begingroup%
\def\thefootnote{\fnsymbol{footnote}}\footnote[#1]{#2}\endgroup}

\renewcommand\Pr[1]{\mathbb{P}\left[#1\right]}

\newcommand\Prmu[1]{\mathbb{P}_{\mu}\left[#1\right]}
\newcommand\Prmuj[1]{\mathbb{P}_{\mu,j}\left[#1\right]}

\newcommand\Es[1]{\mathbb{E}\left[#1\right]}
\newcommand\Esmu[1]{\mathbb{E}_{\mu}\left[#1\right]}
\newcommand\Esmuj[1]{\mathbb{E}_{\mu,j}\left[#1\right]}

\newcommand \fl[1] {\left\lfloor #1 \right\rfloor}
\newcommand \ce[1] {\left\lceil #1 \right\rceil}

\def \Card {\mathop{{\rm Card } }\nolimits}

\def \P {\mathbb{P}}

\def \Pmu {\mathbb{P}_\mu}
\def \Pmuzn {\mathbb{P}_\mu[ \, \cdot \, | \, \zt=n]}
\def \Pmuln {\mathbb{P}_\mu[ \, \cdot \, | \, \lt=n]}
\def \Pmuzgeqn {\mathbb{P}_\mu[ \, \cdot \, | \, \zt\geq n]}
\def \Pmulgeqn {\mathbb{P}_\mu[ \, \cdot \, | \, \lt\geq n]}
\def \Pmuj {\mathbb{P}_{\mu,j}}

\def \N {\mathbb N}
\def \E {\mathbb E}
\def \D {\mathbb D}

\def \C {\mathcal{C}}
\def \R {\mathbb R}

\def \D {\mathbb D}
\def \T {\mathbb T}
\def \Z {\mathbb Z}
\def \F {\mathcal{F}}
\def \S {\mathcal{S}}
\def \W {\mathcal{W}}
\def \Sb {\overline{\mathcal{S}}}

\def \I {\mathcal{I}}
\def \Vc {\mathcal{V}}
\def \V {\textbf{V}}
\newcommand{\norme}[1]{\left\Vert #1\right\Vert _ \infty}

\def \bx {\textnormal{\textbf{x}}}
\def \bf {\textbf{f}}

\def \l {\lambda}

\def \card {\textnormal{Card}}
\def \supp {\textnormal{supp}}

\def \X {X^\textnormal{exc}}
\def \H {H^\textnormal{exc}}

\def \d {\displaystyle}
\def \a {\alpha}
\def \b {\beta}

\def \e {\epsilon}

\def \Nn {\textnormal{\textbf{N}}}
\def \me {\mathbbm{e}}
\def \br {\textnormal{br}}
\def \GW {\textnormal{GW}}
\def \z {\zeta}
\def \sp {{\sigma'}^2}
\def \t {\mathfrak{t}}

\def \| { \, | \,}

\def \zt {\zeta(\tau)}
\def \lt {\lambda(\tau)}

\renewcommand\i[1]{\inf_{\left[#1\right]}}

\newcommand{\keywords}[1]{ \noindent {\footnotesize
             {\small \em Keywords.} {\sc #1}}.}

\newcommand{\ams}[2]{  \noindent {\footnotesize
             {\small \em AMS {\rm 2000} subject classifications.
             {\rm Primary {\sc #1}; secondary {\sc #2}}}.} }

\begin{document}

\vskip1cm

\begin {center}\LARGE Invariance principles for Galton-Watson trees \\
conditioned on the number of leaves \end {center}

\vspace{8mm}

\centerline{\large Igor Kortchemski
\symbolfootnote[1]{Université Paris-Sud, Orsay, France. \texttt{igor.kortchemski@normalesup.org}}}

\vspace{4mm}

\centerline{\large April 2012}



 \vspace{3mm}
\vspace{1cm}

\begin{abstract}
We are interested in the asymptotic behavior of critical
Galton-Watson trees whose offspring distribution may have infinite
variance, which are conditioned on having a large fixed number of leaves. We first find an asymptotic
estimate for the probability of a Galton-Watson tree having $n$ leaves. Second, we let $\t_n$ be a critical Galton-Watson tree whose
offspring distribution is in the domain of
attraction of a stable law, and conditioned on having exactly $n$
leaves. We show that the rescaled Lukasiewicz path and contour
function of $\t_n$ converge respectively to $\X$ and $\H$, where
$\X$ is the normalized excursion of a strictly stable spectrally
positive Lévy process and $\H$ is its associated continuous-time height
function. As an application, we investigate the distribution of the
maximum degree in a critical Galton-Watson tree conditioned on
having a large number of leaves. We also explain how these results can be generalized to the case of Galton-Watson trees which are conditioned on having a large fixed number of vertices with degree in a given set, thus extending results obtained by Aldous, Duquesne and Rizzolo.

\bigskip


 \bigskip

\ams{60J80,60F17}{05C05}
\end{abstract}

\section*{Introduction}
In this article, we are interested in the asymptotic behavior of
critical Galton-Watson trees whose offspring distribution may have
infinite variance, and which are conditioned on having a large fixed number of vertices with degree in a given set. We focus in particular on Galton-Watson trees conditioned on having a large fixed number of leaves. Aldous
\cite{A1,A2} studied the shape of large critical Galton-Watson trees
whose offspring distribution has finite variance, under the condition that the
total progeny is equal to $n$. Aldous' result has then been extended
to the infinite variance case (see e.g. \cite{Duquesne,DuquesneLG}).
In a different but related direction, the effect of
conditioning a Galton-Watson tree on having height equal to $n$ has
been studied \cite{GK,KP,LGIto}, and Broutin \& Marckert \cite{BroutinM} have investigated the asymptotic behavior of uniformly distributed trees with prescribed degree sequence. In \cite {Klam}, we introduced a new type of
conditioning involving the number of leaves of the tree in order to study a specific discrete probabilistic
model, namely dissections of a regular polygon with Boltzmann weights. The results contained in the present article are important
for understanding the asymptotic behavior of the latter model (see
\cite{CK,Klam}). The more general conditioning on having a fixed number of vertices with degree in a given set has been considered very recently by Rizzolo  \cite {Rizzolo}. The results of the present work were obtained independently of  \cite {Rizzolo} (see the end of this introduction for a discussion of the relation between the present work and  \cite {Rizzolo}).

Before stating our main results, let us introduce some notation. If $ \mu$ is a probability distribution on the nonnegative integers, $\Pmu$ will be the law of the
Galton-Watson tree with offspring distribution $ \mu$ (in short the $\GW_\mu$ tree). Let $ \zt$ be the total number of vertices of a tree $\tau$ and let $\lambda(\tau)$ be its number of leaves, that is the number of individuals of
$\tau$ without children.  Let $ \mathcal{A}$ be a non-empty subset of $\N=  \{0,1,2, \ldots\}$. If $ \tau$ is a tree, denote the number of vertices $u \in \tau$ such that the number of children of $u$ is in $\mathcal{A}$ by $ \zeta_ \mathcal{A}( \tau)$. Note that $ \zeta_ \N ( \tau)= \zeta ( \tau)$ and $ \zeta _ {  \{0\}}( \tau)= \lt$.  

We now introduce three different coding functions which determine $\tau$ (see Definition \ref{def:fonctions} for details). Let $u(0), u(1), \ldots u( \zt-1)$ denote the vertices of $ \tau$ in lexicographical order. The Lukasiewicz path $ \W(\tau)=( \W_n(\tau), 0 \leq n  \leq
\zeta(\tau))$ is defined by $ \W_0(\tau)=0$ and for
$0 \leq n \leq \zeta(\tau)-1$, $ \W_{n+1}(\tau)= \W_{n}(\tau)+k_ {n}-1$, where $k_ {n}$ is the number of children of $u(n)$. For $0 \leq i \leq  \zt-1$, define $H_i( \tau)$ as the generation of $u(i)$ and set $H _ {\zt} ( \tau)=0$.
The height function $H( \tau)=(H_t( \tau); 0 \leq t \leq \zt)$ is then defined by linear interpolation. To define the contour function $(C_t ( \tau), 0 \leq t \leq  2 \zt)$, imagine a particle that explores the tree from the left
to the right, moving at unit speed along the edges. Then, for $0 \leq t \leq 2(\zt-1)$, $C_t(\tau)$ is defined as the distance to the root of the position
of the particle at time $t$ and we set
$C_t(\tau)=0$ for $t \in [2(\zt-1), 2 \zt]$. See Fig. \ref{fig:tree1} and  \ref{fig:tree} for an example.

Let
$\theta \in (1,2]$ be a fixed parameter and let $(X_t)_{t \geq 0}$ be the spectrally
positive Lévy process with Laplace exponent $ \E[ \exp(- \lambda
X_t)]= \exp(t \lambda ^ \theta)$. Let also $p_1$ be the density of $X_1$. For $ \theta=2$, note that $X$ is a constant times standard Brownian motion. Let  $\X=(\X_t)_{0 \leq t \leq 1}$ be the normalized excursion of $X$ and $\H=(\H_t)_{0 \leq t \leq 1}$ its
associated continuous-time height function (see Section \ref{sec:lev} for precise definitions). Note that $\H$ is a random continuous function on $[0,1]$
that vanishes at $0$ and at $1$ and takes positive values on
$(0,1)$, which codes the so-called $ \theta$-stable random tree (see \cite{Duquesne}).

\bigskip

We now state our main results. Fix $ \theta \in (1,2]$. Let $\mu$ be an aperiodic probability distribution on the nonnegative integers. Assume that $ \mu$ is critical (the mean of $\mu$ is $1$) and belongs to the domain of attraction of a stable law of
index $\theta \in (1,2]$.
\begin{enumerate}
\item[(I)] Let $d \geq 1$ be the largest integer such
that there exists $b \in \N$ such that supp$(\mu)\backslash  \{0\}$ is
contained in $b+d \Z$, where supp$(\mu)$ is the support of $\mu$.
Then there exists a slowly varying function $h$ such that:
$$ \Prmu{\lambda( \tau)=n} \quad \mathop{\sim}_{n \rightarrow \infty} \quad \mu( 0)^{1/\theta} p_1(0) \frac{\gcd(b-1,d)}{h(n) n^{ 1/\theta+1}}$$
for those values of $n$ such that $ \Prmu{\lambda( \tau)=n}>0$. Here we write $a_n \sim b_n$ if $a_n/b_n \rightarrow
1$ as $n\rightarrow \infty$.
\item[(II)] For every $n\geq 1$ such that $\Prmu{ \lt =n} >0$, let $\t_n$ be a random tree distributed according to $\Pmu[\, \cdot \, | \,  \lt=n]$. Then there exists a sequence of positive real numbers $(B_n)_{n \geq 1}$ converging to $ \infty$ such that

$$\left( \frac{1}{B_{\z(\t_n)} } \W_{ \lfloor \z(\t_n) t \rfloor
}(\t_n),\frac{B_{\z(\t_n)}}{\z(\t_n)}C_{2 \z(\t_n) t}(\t_n),
\frac{B_{\z(\t_n)}}{\z(\t_n)}H_{\z(\t_n) t}(\t_n)\right)_{0 \leq t
\leq 1}$$ converges in distribution to $ ( \X, \H, \H)$ as $n \rightarrow \infty$.
\end{enumerate}

At the end of this work, we explain how to extend (I) and (II) when the condition ``$\lt=n$'' is replaced by the more general condition ``$ \zeta_  { \mathcal {A}} ( \tau)=n$'' (see Theorem \ref{thm:general}). However, we shall give detailed arguments only in the case of a fixed number of leaves. This particular case is  less technical and suffices in view of applications to the study of random dissections. 
\bigskip


We now explain the main steps and techniques used to establish (I) and (II) when  $ \mathcal {A}=  \{0\}$. Let $\nu$ be the probability measure on $ \Z$ defined by
$\nu(k)=\mu(k+1)$ for $k \geq -1$. Our starting point is a
well-known relation between the Lukasiewicz path of a $\GW_\mu$ tree and an associated random
walk. Let $(W_n; n \geq 0)$ be a random walk started at $0$ with
jump distribution $\nu$ and set $\zeta =\inf\{n \geq 0; \,
W_n=-1\}$. Then the Lukasiewicz path of a $\GW_\mu$ tree has the same law as $(W_0,W_1,\ldots,W_{\zeta})$. Consequently, the
total number of leaves of a $\GW_\mu$ tree has the same law as
$\sum_{k=1}^{\zeta} 1_{\{W_k-W_{k-1}=-1\}}$. By noticing that this
last sum involves independent identically distributed Bernoulli
variables of parameter $\mu(0)$, large deviations techniques
give:
\begin{equation}\label{eq:lg}\Pmu\left[\lt=n \textrm{ and }
\left|\zt-\frac{n}{\mu(0)}\right|>{\zt}^{3/4}\right]\leq e^{-c
\sqrt{n}}\end{equation} for some $c> 0$. This roughly says that a $ \GW_ \mu$ tree with $n$ leaves has approximately $n / \mu(0)$ vertices with high probability. Since $\GW_ \mu$ trees conditioned on their total progeny are well known, this will allow us to study $\GW_ \mu$ trees conditioned on their number of leaves.

Let us now explain how an asymptotic estimate for $\Prmu{\lt=n}$ can
be derived. Define $\Lambda(n)$ by:
$$\Lambda(n)= \Card \{ 0 \leq i \leq n-1 ; \, W_{i+1}-W_i=-1\}.$$
The crucial step consists in noticing that for $n,p \geq 1$, the
distribution of $(W_0,W_1,\ldots,W_p)$ under the conditional
probability measure $\P[\, \cdot \, | \, W_p=-1, \Lambda(p)=n]$ is
cyclically exchangeable. The so-called Cyclic Lemma and the relation
between the Lukasiewicz path of a $\GW_\mu$ tree and the random walk
$W$ easily lead to the following identity (Proposition \ref{prop:K}):
\begin{equation}\label{eq:f}\Pmu[\zt=p, \lt=n]=\frac{1}{p} \,
\P[\Lambda(p)=n, W_p=-1]=\frac{1}{p} \, \P[S_p=n] \,
\P[W'_{p-n}=n-1],\end{equation} where $S_p$ is  the sum of $p$
independent Bernoulli random variables of parameter $\mu(0)$ and $W'$
is the random walk $W$ conditioned on having nonnegative jumps. From
the concentration result (\ref{eq:lg}) and using extensively a suitable local
limit theorem, we deduce the asymptotic estimate (I).

The proof of (II) is more
elaborate. The first step consists in proving the convergence on every interval $[0,a]$ with $ a \in (0,1)$. To this end, using the large deviation bound (\ref{eq:lg}), we first prove an analog of (II) when $\t_n$
is a tree distributed according to $\Pmu[ \, \cdot \, | \lt \geq
n]$. We then
use an absolute continuity relation between the conditional probability measure $\Pmuln$ and the
conditional probability measure $\Pmulgeqn$ to get the desired convergence  on every interval $[0,a]$ with $ a \in (0,1)$. The second step is to extend this convergence to the whole interval $[0,1]$ via a tightness argument based on a time-reversal property. In the case of the Lukasiewicz path, an additional argument using the Vervaat transformation is needed.

\bigskip

As an application of these techniques, we study the
distribution of the maximum degree in a Galton-Watson tree
conditioned on having many leaves. More precisely, if $ \tau$ is a tree, let $\Delta( \tau)$ be the maximum number of children of a vertex of $\tau$. Let also $\overline{\Delta}( \X)$ be the
largest jump of the càdlàg process $ \X$. Set $D(n)= \max\{ k \geq 1; \,  \mu([k,\infty)) \geq
1/n\}$. For every $n\geq 1$ such that $\Prmu{ \lt =n} >0$, let $\t_n$ be a random tree distributed according to $\Pmu[\, \cdot \, | \,  \lt =n]$. Then, under assumptions on the asymptotic behavior of the sequence $( \mu(n)^ {1/n})_ {n \geq 1}$ in the finite variance case (see Theorem \ref{thm:max}):
 \begin {enumerate}
 \item [(i)]If the variance of $ \mu$ is infinite, then $\mu(0)^{1/\theta} \Delta(\t_n)/ B_n$ converges in distribution towards  $\overline{\Delta}( \X) $.
 \item[(ii)] If the variance of $\mu$ is finite, then $\Delta(\t_n)/ D(n)$ converges in probability towards $1$.
\end{enumerate}
The second case yields an interesting application to the maximum face degree in a large uniform dissection (see \cite {CK}). Let us mention that using generating functions
and saddle-point techniques,  similar results have been obtained by Meir and Moon \cite{MM} when $\t_n$ is distributed according to $ \Pmuzn$. Our approach can be adapted to give a probabilistic proof of their result.

\bigskip

We now discuss the connections of the present article with
earlier work.  Using different arguments,
formula (\ref{eq:f}) has been obtained in a different form by Kolchin \cite{Kolchin}. The asymptotic behavior of $\Prmu{\zeta_ \mathcal{A}( \tau)= n}$ has been studied in \cite{Minami,My,My2} when $ \Card (\mathcal {A})=1$ and the second moment of $ \mu$  is finite.  Absolute continuity arguments have often been used to derive invariance principles for random trees and forests, see e.g. \cite {CP09, Duquesne,LGM,LGIto}.

Let us now discuss the relationship between the present work and Rizzolo's recent article \cite{Rizzolo}, which deals with similar conditionings of random trees. The main result of \cite{Rizzolo} considers a random tree distributed according to $\Pmu[\, \cdot \, | \,
\zeta_ \mathcal{A}( \tau)= n ]$, where it is assumed that $ 0 \in \mathcal {A}$. In the finite variance case, \cite {Rizzolo} gives the convergence in distribution in the rooted Gromov-Hausdorff-Prokhorov sense of the (suitably rescaled) tree $ \t_n$ viewed as a (rooted) metric space for the graph distance towards the Brownian CRT. Note that the convergence of the contour functions in (II), together with Corollary  \ref {cor:Lambda}, does imply the Gromov-Hausdorff-Prokhorov convergence of trees viewed as metric spaces, but the converse is not true. Furthermore our results also apply to the infinite variance case and include the case where $ 0 \not \in \mathcal {A}$.

\bigskip
The paper is organized as follows. In Section 1, we present the
discrete framework and we define Galton-Watson trees and their codings.
We prove (\ref{eq:f}) and explain how
the local limit theorem gives information on the asymptotic behavior
of large $\GW_\mu$ trees. In Section 2, we present a law of large
numbers for the number of leaves, which leads to the concentration
formula (\ref{eq:lg}). In Section 3, we prove (I). In Section 4, we establish an invariance
principle under the conditional probability $\Pmulgeqn$. In Sections
5 and 6, we refine this result by obtaining an invariance principle
under the conditional probability $\Pmuln$, thus proving (II). As an application, we study in Section 7 the distribution of the maximum degree in a Galton-Watson
tree conditioned on having many leaves. Finally, in Section 8, we explain how the techniques used to deal with the case $ \mathcal {A}=  \{ 0 \}$ can be extended to general sets $ \mathcal {A}$.

\bigskip

\textbf{Acknowledgements.} I am deeply indebted to Jean-Fran\c{c}ois Le Gall for enlightening discussions and for many suggestions on the earlier versions of this work. I also thank Louigi Addario-Berry for a useful discussion concerning the case where $ \theta=2$ and $ \mu$ has infinite variance, and Douglas Rizzolo for remarks on this work.

 \tableofcontents


\bigskip
\bigskip

\textbf {Notation and assumptions.} Throughout this work $\theta \in (1,2]$ will be a fixed parameter. We say that a probability distribution $(\mu(j))_{j \geq 0}$ on the nonnegative integers satisfies hypothesis $(H_ \theta)$ if the following three conditions hold:
\begin{enumerate}
\item[(i)] $\mu$ is critical, meaning that $\sum_{k=0}^\infty k \mu(k)=1$, and $ \mu(1)<1$.
\item[(ii)] $\mu
$ is in the domain of attraction of a stable law of index $\theta
\in (1,2]$. This
means that either the variance of $\mu$ is finite, or $\mu([j,\infty))= j^{-\theta} L(j)$,
where $L: \R_+ \rightarrow \R_+$ is a function such that $\lim_{x
\rightarrow \infty} L(tx)/L(x)=1$ for all $t>0$ (such a function is
called slowly varying). We refer to \cite{Bingham} or \cite[chapter
3.7]{Durrett} for details.
\item[(iii)] $\mu$ is aperiodic, which means that
the additive subgroup of the integers $\mathbb{Z}$ spanned by $\{j;
\, \mu(j) \neq 0 \}$ is not a proper subgroup of $\mathbb{Z}$.
\end{enumerate}
We introduce condition (iii) to
avoid unnecessary complications, but our results can be extended to
the periodic case.

Throughout this text, $\nu$ will stand for the probability measure
defined by $\nu(k)=\mu(k+1)$ for $k \geq -1$. Note that $\nu$ has
zero mean. To simplify notation, we write $ \mu_0$ instead of $ \mu(0)$. Note that $ \mu_0 >0$ under $ ( H_ \theta)$.

\section{The discrete setting : Galton-Watson trees}
\subsection{Galton-Watson trees}

\begin{defn}Let $\N=\{0,1,\ldots\}$ be the set of all nonnegative integers, $\N^*=\{1,2,\ldots\}$ and $U$ the set of labels:
$$U=\bigcup_{n=0}^{\infty} (\N^*)^n,$$
where by convention $(\N^*)^0=\{\emptyset\}$. An element of $U$ is a
sequence $u=u_1 \cdots u_m$ of positive integers, and we set
$|u|=m$, which represents the \og generation \fg \, of $u$. If $u=u_1
\cdots u_m$ and $v=v_1 \cdots v_n$ belong to $U$, we write $uv=u_1
\cdots u_m v_1 \cdots v_n$ for the concatenation of $u$ and $v$. In
particular, note that $u \emptyset=\emptyset u = u$. Finally, a
\emph{rooted ordered tree} $\tau$ is a finite subset of $U$ such
that:
\begin{itemize}
\item[1.] $\emptyset \in \tau$,
\item[2.] if $v \in \tau$ and $v=uj$ for some $j \in \N^*$, then $u
\in \tau$,
\item[3.] for every $u \in \tau$, there exists an integer $k_u(\tau)
\geq 0$ such that, for every $j \in \N^*$, $uj \in \tau$ if and only
if $1 \leq j \leq k_u(\tau)$.
\end{itemize}
In the following, by \emph{tree} we will always mean rooted ordered
tree. We denote by the set of all trees by $\T$. We will often view each vertex of
a tree $\tau$ as an individual of a population whose $\tau$ is the
genealogical tree. The total progeny of $\tau$ will be
denoted by $\zeta(\tau)= \textrm{Card}(\tau)$.  A leaf of a tree $\tau$ is a vertex $u
\in \tau$ such that $k_u(\tau)=0$. The total number of leaves of
$\tau$ will be denoted by $\lambda(\tau)$. If $\tau$ is a tree and
$u \in \tau$, we define the shift of $\tau$ at $u$ by $T_u \tau=\{v
\in U; \, uv \in \tau\}$, which is itself a tree.
\end{defn}


\begin{defn}Let $\rho$ be a probability measure on $\N$ with mean less than or equal to $1$ and,
to avoid trivialities, such that $\rho(1)<1$. The law of the
Galton-Watson tree with offspring distribution $\rho$ is the unique
probability measure $\P_\rho$ on $\T$ such that:
\begin{itemize}
\item[1.] $\P_\rho(k_\emptyset=j)=\rho(j)$ for $j \geq 0$,
\item[2.] for every $j \geq 1$ with $\rho(j)>0$, the shifted trees
$T_1 \tau, \ldots, T_j \tau$ are independent under the conditional
probability $\P_\rho( \, \cdot \,|\, k_\emptyset=j)$ and their
conditional distribution is $\P_\rho$.
\end{itemize}
A random tree whose distribution is $\P_\rho$ will be called a Galton-Watson tree with offspring distribution $ \rho$, or in short a $\GW_\rho$ tree.
\end{defn}

 In the sequel, for an integer $j \geq 1$, $\Pmuj$
will stand for the probability measure on $\T^j$ which is the
distribution of $j$ independent $\GW_\mu$ trees. The canonical element
of $\T^j$ will be denoted by $\bf$. For $\bf=(\tau_1,\ldots,\tau_j)
\in \T^j$, set $\l(\bf) = \l(\tau_1)+\cdots+\l(\tau_j)$ and
$\z(\bf)=\z(\tau_1)+\cdots+\z(\tau_j)$ for respectively the total
number of leaves of $\bf$ and the total progeny of $\bf$.

\subsection{Coding Galton-Watson trees}

We now explain how trees can be coded by three different functions.
These codings are crucial in the understanding of large
Galton-Watson trees.

\begin{figure*}[h!]
 \begin{minipage}[c]{9cm}
   \centering
      \includegraphics[scale=0.4]{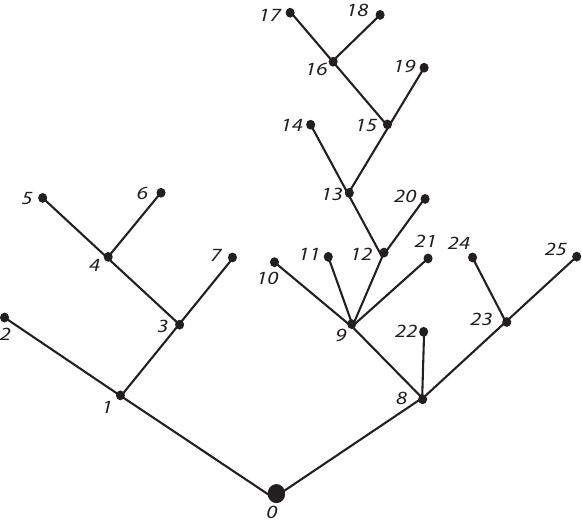}
   \end{minipage}
   \begin{minipage}[c]{9cm}
   \centering
      \includegraphics[scale=0.4]{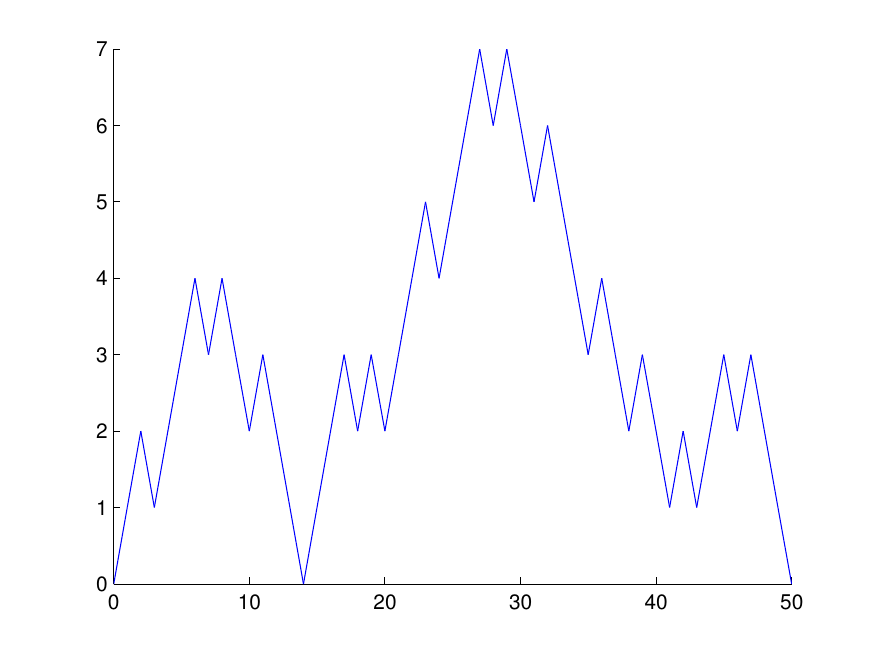}
   \end{minipage}
   \caption{\label{fig:tree1}A tree $\tau$ with its vertices indexed in
lexicographical order and its  contour function $(C_{u}(\tau);\, 0
\leq u \leq 2(\zt-1)$. Here, $\zeta(\tau)=26$.}
   \begin{minipage}[c]{9cm}
   \centering
      \includegraphics[scale=0.4]{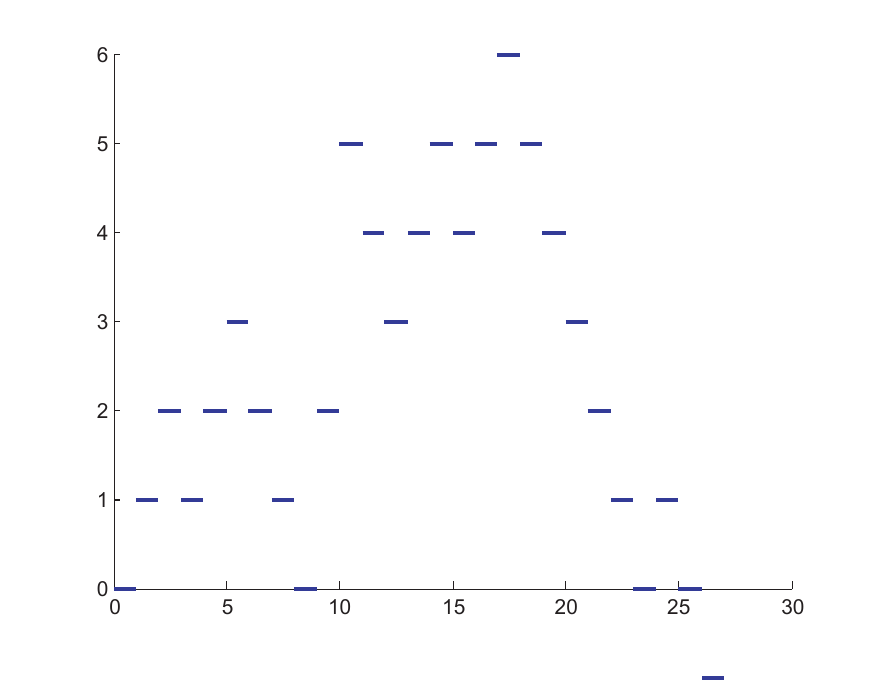}
   \end{minipage}
   \begin{minipage}[c]{9cm}
   \centering
      \includegraphics[scale=0.4]{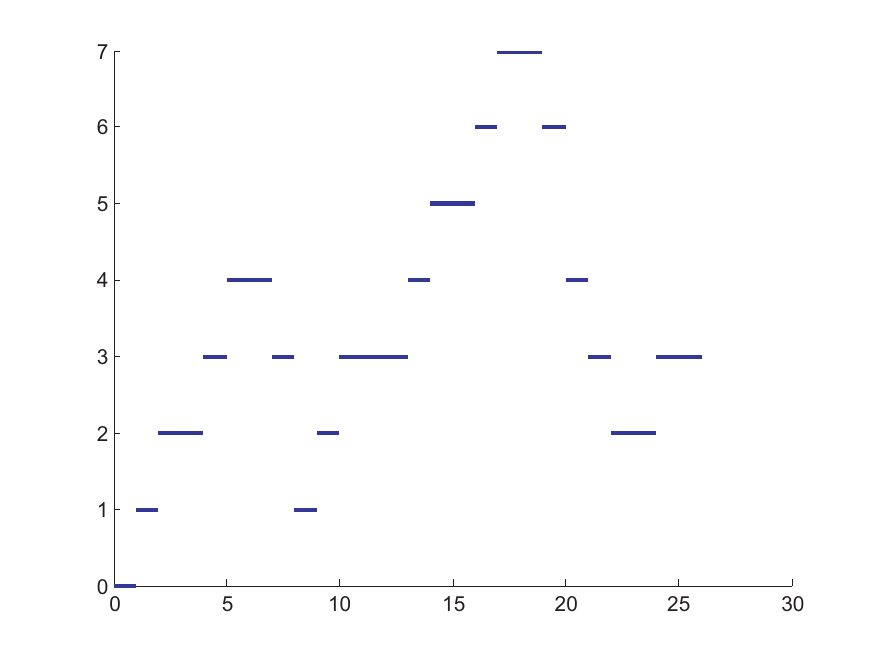}
   \end{minipage}
   \caption{\label{fig:tree} The Lukasiewicz path $\W( \tau)$ and the height function $H(\tau)$ of $\tau$.}
\end{figure*}

\begin{defn}\label{def:fonctions}
We write $u<v$ for the lexicographical order on the labels $U$ (for
example $\emptyset<1<21<22$). Consider a tree $\tau$ and order the
individuals of $\tau$ in lexicographical order:
$\emptyset=u(0)<u(1)<\cdots<u(\zeta(\tau)-1)$. The height process
$H(\tau)=(H_n(\tau), 0 \leq n < \zeta(\tau))$ is defined, for $0
\leq n < \zeta(\tau)$, by: $$H_n(\tau)=|u(n)|.$$ For technical
reasons, we set $H_{\z(\tau)}(\tau)=0$.

Consider a particle that starts from the root and visits
continuously all edges at unit speed (assuming that every edge has
unit length), going backwards as little as possible and respecting the
lexicographical order of vertices. For $0 \leq t \leq 2(\zt-1)$,
$C_t(\tau)$ is defined as the distance to the root of the position
of the particle at time $t$. For technical reasons, we set
$C_t(\tau)=0$ for $t \in [2(\zt-1), 2 \zt]$. The function $C(\tau)$
is called the contour function of the tree $\tau$. See Figure
\ref{fig:tree1} for an example, and \cite[Section 2]{Duquesne} for a
rigorous definition.

Finally, the Lukasiewicz path $ \W(\tau)=( \W_n(\tau), 0 \leq n  \leq
\zeta(\tau))$ of $\tau$ is defined by $ \W_0(\tau)=0$ and for
$0 \leq n \leq \zeta(\tau)-1$:
$$ \W_{n+1}(\tau)= \W_{n}(\tau)+k_{u(n)}(\tau)-1.$$
\end{defn}
Note that necessarily $ \W_{\zt}(\tau)=-1$.

A forest is a finite or infinite ordered sequence of trees. The Lukasiewicz path of a forest is defined as the concatenation of the Lukasiewicz paths of the trees it contains (the word \og concatenation \fg \, should be understood in the appropriate manner, see \cite[Section 2]{Duquesne} for a more precise definition). The following proposition explains the importance of the Lukasiewicz path.

\begin{prop}\label{prop:RW}Fix an integer $j \geq 1$. Let $(W_n; n \geq 0)$ be a
random walk which starts at $0$ with jump distribution $\nu(k)=\mu(k+1)$ for $k \geq -1$. Define
$\zeta_j =\inf\{n \geq 0; \, W_n=-j\}$. Then $(W_0,W_1,\ldots,W_{\zeta_j})$ is distributed as the Lukasiewicz path of a forest of $j$ independent
$\GW_\mu$ trees. In particular, the total progeny of $j$ independent
$\GW_\mu$ trees has the same law as $\z_j$.\end{prop}

\begin{proof}See \cite[Proposition 1.5]{RandomTrees}.\end{proof}

Note that the previous proposition applied with $j=1$ entails that the Lukasiewicz path of a Galton-Watson tree is distributed as the random walk $W$ stopped when it hits $-1$ for the first time. We conclude this subsection by giving a link between the height
function and the Lukasiewicz path (see \cite[Prop. 1.2]{RandomTrees}
for a proof).

\begin{prop}Let $\tau$ be a tree. Then, for every $ 0 \leq n < \zeta(\tau)$:
\begin{equation}\label{eq:Hmes}H_n(\tau)=\card \left\{ 0 \leq j < n; \, \W_j(\tau)=\inf_{j \leq k \leq n} \W_k(\tau)\right\}.\end{equation}\end{prop}

\subsection{The Cyclic Lemma}

We now state the Cyclic Lemma which is crucial in the derivation of
the joint law of $(\zt,\lt)$ under $\Pmu$. For integers $1 \leq j \leq p$, define:
$$\S_p^{(j)}= \{ (x_1,\ldots,x_p) \in \{-1,0,1,2, \ldots\}^p ; \,
\sum_{i=1}^p x_i=-j\}$$ and
$$\Sb_p^{(j)}= \{ (x_1,\ldots,x_p) \in \S_p^{(j)} ; \, \sum_{i=1}^m x_i >-j
\textrm{ for all } m \in \{0,1,\ldots,p-1\}\}.$$ For $\bx=(x_1, \ldots,x_p)
\in\S_p^{(j)}$ and $i \in \Z/p\Z$, denote by $\bx^{(i)}$ the $i$-th
cyclic shift of $\bx$ defined by $x^{(i)}_k=x_{i+k \mod p}$ for $1
\leq k \leq p$. For $\bx \in \S_p^{(j)}$, finally set:
$$\I_{\textrm{\bx}}= \left\{ i \in \Z/p\Z; \, \bx^{(i)} \in \Sb_p^{(j)}
\right\}.$$
The so-called Cyclic Lemma states that we have $\Card (\I_{\textnormal{\bx}} )= j$ for every $\textnormal{\bx} \in
\S_p^{(j)}$ (see \cite[Lemma 6.1]{Pitman} for a proof). 

Let $(W_n; n \geq 0)$  and $ \zeta_j$ be as in Proposition \ref{prop:RW}. Define
$\Lambda(k)$ by $\Lambda(k)= \Card \{ 0 \leq i \leq k-1 ; \, W_{i+1}-W_i=-1\}$.
Let finally $n,p \geq 1$ be positive integers. From the Cyclic Lemma and the fact that for all $k \in \Z/p\Z$ one has $ \Card \{ 1 \leq i \leq p ; \, x_i=-1
\}=\Card \{ 1 \leq i \leq p ; \, x^{(k)}_i=-1 \}$, it is a simple matter to deduce that:
\begin{equation}
\label{eq:kemperman}
\P[ \z_j =p, \Lambda(p) =n ] = \frac{j}{p} \P[W_p=-j, \Lambda(p)=n].
\end{equation}
See e.g. \cite[Section 6.1]{Pitman} for similar arguments. Note in particular that we have $\P[ \z_j =p] =  j \P[W_p=-j]/p $. This result allows us to derive the joint law of $(\zt,\lt)$ under $\Pmu$:

\begin{prop}\label{prop:K}Let $j$ and $n \leq p$
be positive integers. We have:
$$\Pmuj[\z(\bf)=p, \l(\bf)=n]=\frac{j}{p} \, \P[S_p=n] \, \P[W'_{p-n}=n-j].$$
where $S_p$ is the sum of $p$ independent Bernoulli random variables
of parameter $\mu_0$ and $W'$ is the random walk started from $0$
with nonnegative jumps distributed according to $ \eta(i)=\mu(i+1)/(1-\mu_0)$ for every $i\geq 0$.
\end{prop}

\begin{proof}
Using  Proposition \ref{prop:RW} and \eqref{eq:kemperman} , write $ \Pmuj[\z(\bf)=p, \l(\bf)=n]= j \P[\Lambda(p)=n, W_p=-j]/p$. To simplify notation, set $X_i =W_ {i}-W_ {i-1}$ for $i \geq 1$ and note that:
\begin{eqnarray*}
 \P[\Lambda(p)=n, W_p=-j] &=& \sum_ {1 \leq i_1 < \cdots < i_n \leq p}  \Pr{ X_i=-1, \,\forall i \in \{ i_1, \ldots, i_n\} }   \\
 && \qquad \qquad \qquad \cdot \, \Pr { \sum_ {i \not \in \{ i_1, \ldots, i_n\}} X_i=n-j; \quad X_i>-1, \, \forall i \not \in \{ i_1, \ldots, i_n\}}.
\end{eqnarray*}
The last probability is equal to $ \Pr {W'_ {p-n}=n-j} \Pr {X_i > -1, \, \forall i \not \in \{ i_1, \ldots, i_n\}}$ and it follows that:
\begin{equation}
\label{eq:LambdaUtile} \Pr { \Lambda(p)=n, W_p = -j} = \Pr {W' _ {p-n}=n-j} \Pr {S_p=n},
\end{equation}
giving the desired result.\end{proof}

\subsection{Slowly varying functions}

Slowly varying functions appear in the study of domains of
attractions of stable laws. Here we recall some properties of these
functions in view of future use.

Recall that a nonnegative measurable function $L: \R_+ \rightarrow
 \R_+$ is said to be slowly varying if, for every $t>0$, $ L(tx)/L(x)
\rightarrow 1$ as $x \rightarrow \infty$. A useful result concerning
these functions is the so-called Representation Theorem, which
states that a function $L: \R_+ \rightarrow \R_+$ is slowly varying
if and only if it can be written in the form:
$$L(x)=c(x) \exp \left( \int_1^x \frac{\e(u)}{u} du\right), \qquad x \geq 0,$$
where $c$ is a nonnegative measurable function having a finite positive
limit at infinity and $\e$ is a measurable function tending to $0$
at infinity. See e.g. \cite[Theorem 1.3.1]{Bingham} for a proof. The
following result is then an easy consequence.

\begin{prop}\label{prop:slow}Fix $\epsilon>0$ and let  $L: \R_+ \rightarrow \R_+$ be a
slowly varying function.
\begin{enumerate}
\item[(i)]We have $x^{\e} L(x) \rightarrow \infty$ and $x^{-\e} L(x) \rightarrow 0$ as $x \rightarrow \infty$.
\item[(ii)] There exists a constant $C>1$ such that $\frac{1}{C}x^{-\e}\leq{L(nx)}/{L(n)} \leq Cx^\e$ for every
integer $n$ sufficiently large and $x \geq 1$.
\end{enumerate}
\end{prop}

\subsection{The Local Limit Theorem}
\label{sec:LLT}

\begin{defn}A subset $A \subset \Z$ is said to be lattice if there exist $b \in \Z$ and $d  \geq 2$ such that
$A \subset  b+ d \Z$. The largest $d$ for which this statement holds
is called the span of $A$. A measure on $\Z$ is said to be lattice
if its support is lattice, and a random variable is said to be
lattice if its law is lattice.\end{defn}

\begin{rem}\label{rem:non-lattice}Since $\mu$ is supposed to be critical and aperiodic, using the fact that
$\mu(0)>0$, it is an exercise to check that the probability measure
$\nu$ is non-lattice.\end{rem}

Recall that $(X_t)_{t \geq 0}$ is the spectrally positive Lévy process with Laplace exponent $ \E[ \exp(- \lambda X_t)]=\exp(t \lambda ^ \theta)$ and $p_1$ is the density of $X_1$. When $ \theta=2$, we have $p_1(x)= e^ {-x^2/4}/ \sqrt {4 \pi}$. It is well known that $p_1$ is positive, continuous and bounded (see e.g. \cite[I. 4]{Zolotarev}). The following theorem will allow us to find estimates for the
probabilities appearing in Proposition \ref{prop:K}.

\begin{thm}[Local Limit Theorem]\label{thm:locallimit}
Let $(W_n)_{n \geq 0}$ be a random walk on $\Z$ started from $0$
such that its jump distribution is in the domain of attraction of a
stable law of index $\theta \in (1,2]$. Assume that $W_1$ is
non-lattice and that $\P[W_1<-1]=0$. Set $K(x)=\E[W_1^2 1_{|W_1|\leq x}]$ for $x \geq 0$. Let $ \sigma^2$ be the variance of $W_1$ and set:
$$  \begin {cases} \d 
a_n = \sigma \sqrt{n/2} & \textrm {if  $ \sigma^2 < \infty$,} \\
\d a_n = (\Gamma(1-\theta))^{1/\theta} \inf \left\{ x \geq 0; \, \Pr {W_1>x} \leq
\frac{1}{n}\right\} & \textrm {if $ \sigma^2= \infty$ and $ \theta<2$}, \\
\d a_n = \sqrt{n \, K \left( \sup  \left\{ z \geq  0; \frac{K(z)}{z^2}\geq \frac{1}{n} \right\}\right)} & \textrm {if $ \sigma^2= \infty$ and $ \theta=2$},
\end {cases}$$
with the convention $ \sup \emptyset = 0$.
\begin {enumerate}
\item[(i)] The random variable $(W_n - n \E[W_1])/a_n$ converges in distribution towards $X_1$.
\item [(ii)]We have $a_n=n^{1/\theta} L(n)$ where $L: \R_+ \rightarrow \R_+$ is slowly varying.
\item[(iii)] We have $ \d \lim_{n \rightarrow \infty} \sup_{k \in \Z}
\left| a_n \P[W_n=k]-p_1\left( \frac{k-n \E[W_1]}{a_n}\right)
\right|=0$. 
\end {enumerate}
\end{thm}

\begin{proof} First note that $\E[|W_1|]< \infty$ since $\theta>1$ (this is a
consequence of \cite[Theorem 2.6.1]{IL}).  

We start with (i). The case $ \sigma^2 < \infty$ is the classical central limit theorem. Now assume that $ \sigma ^2 = \infty$ and $ \theta < 2$. Write $G(x)= \P[|W_1|>x]$ for
$x \geq 0$  and introduce $a'_n=\inf \left\{ x \geq 0; \,
G(x) \leq {1}/{n}\right\}$, so that $a_n=
(\Gamma(1-\theta))^{1/\theta} a'_n$ for $n$ sufficiently large.  By \cite[Formula
3.7.6]{Durrett}, we have $nG(a'_n) \rightarrow 1$. By definition of
the domain of attraction of a stable law, there
exists a slowly varying function $L: \R_+ \rightarrow \R_+$ such
that $G(x)=L(x)/x^\theta$. Hence $G(a_n) \sim 1/(n
\Gamma(1-\theta)) $. Next, by \cite[Section
XVII (5.21)]{Feller} we have $K(x) \sim  x^2 G(x) {\theta}/(2-\theta)$ as $x \rightarrow \infty$. Hence:
$$\frac{nK (a_n)}{a_n^2} \quad\sim \quad\frac{n}{a_n^2} \frac{\theta}{2-\theta}
a_n^2 G(a_n) \quad\sim \quad\frac{\theta}{(2-\theta)\Gamma(1-\theta)}.$$ From
\cite[Section XVII.5, Theorem 3]{Feller}, we now get that $(W_n - n
\E[W_1])/a_n$  converges in distribution to $X_1$.
Finally, in the case $ \sigma ^2 = \infty$ and $ \theta = 2$, assertion (i) is a straightforward consequence of the proof of Theorem 2.6.2 in \cite{IL}.

We turn to the proof of (ii). By \cite[p. 46]{IL}, for every integer $k
\geq 1$, $a _ {kn} / a_n \rightarrow k^{1/\theta}$ as $n \rightarrow \infty$. Since $(a_n)$ is increasing, by a theorem of de Haan (see \cite[Theorem 1.10.7]{Bingham}), this
implies that there exists a slowly varying function $L: \R_+
\rightarrow \R_+$ such that $a_n=L(n) n^{1/\theta}$ for every
positive integer $n$.

Assertion (iii) is the classical local limit theorem (see \cite[Theorem 4.2.1]{IL}).
\end{proof}

In the case $ \sigma ^2= \infty$ and $ \theta=2$, note that $L(n) \rightarrow \infty$ as $n \rightarrow \infty$ and that $L$ can be chosen to be increasing.
\medskip

Assume that $ \mu$ satisfies $(H_ \theta)$ for a certain $ \theta \in (1,2]$. Let $(W_n)_{n \geq 0}$ be a random walk started from $0$ with jump
distribution $\nu$. Since $\mu$ is in the domain of attraction of a
stable law of index $\theta$, it follows that $\nu$ is also in this
domain of attraction. Moreover, $ \Es {W_1}=0$ and $W_1$ is not lattice by Remark
\ref{rem:non-lattice}. Let $ \sigma^2$ be the variance of $ W_1$ and define $B_n$ to be equal to the quantity $a_n$ defined in Theorem \ref{thm:locallimit}. Then, as $n \rightarrow \infty$, $W_n/B_n$ converges in distribution towards $X_1$. In what follows, $h: \R_+ \rightarrow \R_+$
will stand for a slowly varying function such that $B_n=h(n)n
^{1/\theta}$. 

\begin{lem}\label{lem:locallimit}We have:
$$\Prmu{\zeta(\tau)=n} \quad\mathop{\sim}_{n \rightarrow \infty} \quad\frac{p_1(0)}{n^{1/\theta+1} h(n)}, \qquad \Prmu{\zeta(\tau) \geq n} \quad\mathop{\sim}_{n \rightarrow \infty}\quad \frac{\theta p_{1}(0)}{n^{1/\theta} h(n)}.$$
\end{lem}

\begin{proof}This is an easy consequence
 of Theorem \ref {thm:locallimit} (iii) together with the fact that $ \Prmu { \zt =n}= \Pr {W_n=-1}/n$, as noticed before Proposition \ref {prop:K}.
\end{proof}

\begin{rem}\label{rem:lattice}In particular, $\P_\mu[\zeta(\tau)=n]>0$ for $n$
sufficiently large if $\mu$ is aperiodic. When $\mu$ is
periodic, if $d$ is the span of the support of $\mu$, one can check that for $n$ sufficiently large, one has $\Prmu{\zt=n}>0$ if and only if
$n=1 \textrm{ mod } d$.
\end{rem}

\section{A law of large numbers for the number of leaves}

In the sequel, we fix $ \theta \in (1,2]$ and consider a probability distribution $ \mu$ on $ \N$ satisfying hypothesis $ (H_ \theta)$. In this section, we show that if a $\GW_\mu$ tree has total
progeny equal to $n$, then it has approximatively $\mu_0 n$ leaves
with high probability. Intuitively, this comes from the fact that
each individual of a $\GW_\mu$ has a probability $\mu_0$ of being a
leaf. Conversely, we also establish that if a $\GW_\mu$ tree has
$n$ leaves, than it has approximatively $n/\mu_0$ vertices with high
probability.

\begin{defn}\label{def:Lambda}Consider a tree $\tau \in \T$ and let $(u(i),\, 0 \leq i \leq \zt-1)$ be the vertices of $\tau$
listed in lexicographical order and denote by $k_j$ the number of
children of $u(j)$. For $0 \leq s < \zt$ define $\Lambda_\tau (s)$
by $\Lambda_\tau(s)=\sum_{j=0}^{ \fl{s}}1_{\{k_j=0\}}$,
where $\fl{s}$ stands for the integer part of $s$. Set also $ \Lambda_\tau( \zt)= \lt$.\end{defn}

\begin{lem}\label{lem:phi}Let $(X_i)_{i \geq 1}$ be a sequence of independent identically distributed Bernoulli random variables
of parameter $\mu_0$. For $0 \leq x \leq 1$, define $\phi^*(x)=x \ln
\frac{x}{\mu_0}+(1-x) \ln \frac{1-x}{1-\mu_0}$. The following two
properties hold:
\begin{enumerate}
\item[(i)] For $a>0$ and $n \geq 1$:
$$\P\left[ \frac{1}{n}\sum_{k=1}^n X_k >\mu_0+a\right] \leq 2 e^{ -n \phi^\ast(\mu_0+a)}, \qquad \P\left[ \frac{1}{n} \sum_{k=1}^n X_k < \mu_0-a\right] \leq 2 e^{ -n \phi^\ast(\mu_0-a)}.$$
\item[(ii)] We have $ \d\phi^\ast(\mu_0+ x) =
\frac{1}{2 \mu_0(1-\mu_0)} x^2 + o(x^2)$ when $x \rightarrow 0$.
\end{enumerate}
\end{lem}

\begin{proof}For the first assertion, see \cite[Remark (c) in Theorem 2.2.3]{Dembo}.
The second one is a simple calculation left to the reader.\end{proof}

\begin{defn}Let $\e>0$. We say that a sequence of positive numbers $(x_n)$ is $oe_{\epsilon}(n)$ if there
exist positive constants $c,C >0$ such that $x_n \leq
C e^{-cn^{\epsilon}}$ for all $n$ and we write
$x_n=oe_\epsilon(n)$.\end{defn}

\begin{rem} \label{rem:oe} It is easy to see that if $x_n =oe_\e(n)$ for some $\e>0$ then the sequence $(y_n)_{n \geq 1}$ defined by
$y_n=\sum_{k=n}^{\infty} x_k$ is also $oe_\e(n)$.\end {rem}

\begin{lem}\label{lem:GdesDev} Fix $0<\eta<1$ and
$\delta >0$.
\begin{enumerate}
\item[(i)] Let $(W_n =X_1+\cdots+X_n; n \geq 0)$ be a random walk started at $0$ with jump distribution $\nu(k)=\mu(k+1)$, $k \geq
-1$ under $\P$. Then:
$$\P\left[ \sup_{\eta \leq t \leq 1} \left|\frac{1}{nt}
\sum_{j=0}^{\fl{nt}}1_{\{X_j=-1\}} -\mu_0 \right| >
\frac{\delta}{n^{1/4}} \right]=oe_{1/2}(n).$$ \item[(ii)] For those values of $n$ such that $ \Prmu{ \zt=n}>0$ we have:
$$\Prmu{ \left. \sup_{\eta\leq t \leq 1} \left| \frac{\Lambda_\tau(nt)}{nt} - \mu_0\right|\geq
\frac{\delta}{n^{1/4}} \, \right | \, \zeta(\tau)=n } =
oe_{1/2}(n).$$
\end{enumerate}
\end{lem}

\begin{proof}For the first assertion, define $Z_k=\left|\frac{1}{k}
\sum_{j=0}^{k}1_{\{X_j=-1\}} -\mu_0 \right|$ for $k \geq 1$. By
Lemma \ref{lem:phi} (ii), for $n$ sufficiently large we have
$ \phi^\ast\left(\mu_0 \pm {\delta}{n^{-1/4}}\right)>c
{n^{-1/2}},$ for some $c>0$. Since the random variables
$(1_{\{X_j=-1\}})_{j \geq 1}$ are independent Bernoulli random
variables of parameter $\mu_0$, for large $n$ and $k \geq  \fl{\eta n}$
we have by Lemma \ref {lem:phi} (i):
  $$\P [Z_k > \delta n^{- 1/4}]\leq 4\exp \left(- c \frac{k}{{n}^{1/2}} \right) \leq
  4 \exp\left(- c \frac{\eta n -1}{n^{1/2}}\right)
  \leq 4\exp\left(- \frac{c \eta}{2} n^{1/2}\right).$$ Therefore, for large enough $n$:
\begin{eqnarray*}\P\left[ \sup_{\eta \leq t \leq 1}
\left|\frac{1}{nt} \sum_{j=0}^{\fl{nt}}1_{\{X_j=-1\}}
-\mu_0 \right| > \frac{\delta}{n^{1/4}} \right]&\leq& \P\left[ \exists k \in [\eta n-1, n] \cap \N \textrm{ such that } \, Z_k >\frac{\delta}{n^{1/4}}\right]\\
&\leq&\sum_{k= \fl{\eta n}}^n \P\left [Z_k > \frac{\delta}{n^{1/4}} \right] \\
&\leq&   4 (1-\eta)n \exp \left(- \frac{c \eta}{2} n^{1/2} \right),
\end{eqnarray*}
which is $oe_{1/2}(n)$.

For the second assertion, introduce $\zeta=\inf\{n \geq 0; \,
W_n=-1\}$ and use Proposition \ref{prop:RW} which tells us that:
\begin{eqnarray*}\Pmu\left[ \left. \sup_{\eta\leq t
\leq 1} \left| \frac{\Lambda_\tau(nt)}{nt} - \mu_0\right|\geq
\frac{\delta}{n^{1/4}} \right| \zt =n \right]&=& \P \left[ \left.
 \sup_{\eta \leq t \leq 1} \left|\frac{1}{nt}
\sum_{j=0}^{\fl{nt}}1_{\{X_j=-1\}} -\mu_0 \right| >
\frac{\delta}{n^{1/4}} \,  \right| \,
\zeta=n\right]\\
&\leq& \frac{1}{ \P[ \zeta=n]}\P \left[
 \sup_{\eta \leq t \leq 1} \left|\frac{1}{nt}
\sum_{j=0}^{\fl{nt}}1_{\{X_j=-1\}} -\mu_0 \right| >
\frac{\delta}{n^{1/4}} \right].
\end{eqnarray*}
By (i), the last probability in the right-hand side is $oe_{1/2}(n)$ and by
Lemma \ref{lem:locallimit} combined with Proposition \ref{prop:slow} (ii), the quantity $\Pr { \zeta=n}= \Prmu { \zt=n}$ is bounded below by $n^ {-1/ \theta-2}$ for large $n$. The desired result follows.
\end{proof}

\begin{cor}\label{cor:GD2mieux}We have for every $ \eta \in (0,1)$ and $ \delta>0$:
$$\Prmu{ \left.\sup_{\eta\leq t \leq 1} \left| \frac{\Lambda_\tau( \zt t)}{\zt t} - \mu_0\right|\geq
\frac{\delta}{n^{1/4}} \, \right| \, \zeta(\tau) \geq n } =
oe_{1/2}(n).$$\end{cor}

\begin{proof}To simplify notation, set $ \d A_n= \left \{\sup_{\eta\leq t \leq 1} \left| \frac{\Lambda_\tau( \zt
t)}{\zt t} - \mu_0\right|\geq \frac{\delta}{n^{1/4}} \right \}$. It suffices to notice that:
$$ \Prmu{ \left. A_n \, \right|\, \zeta(\tau)
\geq n }  \leq  \sum_{k=n}
^{\infty} \frac{\Prmu{\zt =k}}{ \Prmu{\zt\geq n}} \Prmu{ \left.
A_k\, \right| \, \zeta(\tau) =k },
$$
observing that the quantities $ \Prmu {A_k | \zt=k}$ are bounded by Lemma \ref {lem:GdesDev} (ii). Details are left to the reader.
\end{proof}

We have just shown that if a $\GW_\mu$ tree has total progeny $n$, then it has approximatively $\mu_0 n$ leaves and the
deviations from this value have exponentially small probability. Part (ii) of the
following crucial lemma provides a converse to this statement by proving that if a
$\GW_\mu$ tree has $n$ leaves, then the probability that its total
progeny does not belong to $[n/\mu_0 - n^{3/4}, n/\mu_0+n^{3/4}]$
decreases exponentially fast in $n$.

\begin{lem}\label{lem:GdesDev2}We have for $1 \leq j \leq n$ and $ \delta>0$:
\begin{enumerate}
\item[(i)]
$ \d \Pmuj\left[\left|\frac{\lambda(\bf)}{n}-\mu_0\right|>\frac{ \delta}{n^{1/4}} \textrm { and }
\zeta(\bf)=n \right]=oe_{1/2}(n),$ uniformly in $j$.
\item[(ii)]  $ \d\Pmuj\left[\lambda(\bf)=n \textrm{ and }
\left|\zeta(\bf)-\frac{n}{\mu_0}\right|>{\zeta(\bf)}^{3/4}\right]=oe_{1/2}(n),$
uniformly in $j$.
\end{enumerate}
\end{lem}

\begin{proof}The proof of assertion (i) is very similar to that of Lemma \ref {lem:GdesDev}. The only difference is the fact that we are now considering a forest, but we can still use Proposition \ref {prop:RW}. We leave details to the reader.

Let us turn to the proof of the second assertion, which is a bit more
technical. First write:
\begin{eqnarray*}&&\Pmuj\left[\lambda(\bf)=n,\,
\left|\zeta(\bf)-\frac{n}{\mu_0}\right|>{\zeta(\bf)}^{3/4}\right]\\
 && \qquad\qquad\qquad =\Pmuj\left[\lambda(\bf)=n ,\,
\zeta(\bf)>\frac{n}{\mu_0}+{\zeta(\bf)}^{3/4}\right]+\Pmuj\left[\lambda(\bf)=n
,\,\zeta(\bf)<\frac{n}{\mu_0}-{\zeta(\bf)}^{3/4}\right].\end{eqnarray*}
Denote the first term on the right-hand side by $I_n$ and the second
term by $J_n$. We first deal with $I_n$ and show that
$I_n=oe_{1/2}(n)$. We observe that:
\begin{eqnarray*} I_n &\leq& \sum_{k=n}^{\infty} \Pmuj\left[\lambda(\bf)<\mu_0 k-\mu_0 k^{3/4}, \,
\zeta(\bf)=k \right].\end{eqnarray*}Assertion (i) implies
that $\Pmuj\left[\lambda(\bf)<\mu_0 k-\mu_0 k^{3/4}, \, \zeta(\bf)=k
\right]=oe_{1/2}(k)$, and this entails that $I_n=oe_{1/2}(n)$.

We complete the proof by showing that $J_n=oe_{1/2}(n)$. Write: $$J_n \leq  \sum_{k=n} ^{\fl{n/\mu_0}}\Pmuj\left[\zeta(\bf)=k ,\,
\frac{\lambda(\bf)}{k}-\mu_0 > \frac{\mu_0}{k^{1/4}}\right].$$
By  Lemma \ref{lem:phi} (ii), we have
$\phi^\ast\left(\mu_0 + {\mu_0}{k^{-1/4}}\right)>c_2
k^{-1/2}$ for some $c_2>0$ and for every $k \geq n$, provided that $n$ is sufficiently large.
Then, using Proposition
\ref{prop:K} and Lemma \ref{lem:phi} (i):
$$J_n \leq \sum_{k=n}^{\fl{n/\mu_0}}
\frac{j}{k} \P \left[ \d
 \frac{1}{k}
\sum_{p=1}^{k}1_{\{X_p=-1\}} > \mu_0+\frac{\mu_0}{k^{1/4}}\right] \leq \sum_{k=n}^{\fl{n/\mu_0}} 2 \exp(-c_2 k^{1/2})
$$
which is $oe_{1/2}(n)$.
 \end{proof}

\section{Estimate for the probability of having $n$ leaves}

In this section, we give a precise asymptotic estimate for the probability that
a $\GW_\mu$ tree has $n$ leaves. This result is of
independent interest, but will also be useful when proving an
invariance principle for $\GW_\mu$ trees conditioned on having $n$
leaves.

Recall that $\mu$ is a probability distribution on $ \N$ satisfying hypothesis $ (H _ \theta)$ with $ \theta \in (1,2]$. Recall also that $h$ is the slowly varying function that was
introduced just before Lemma \ref{lem:locallimit}.

\begin{thm}\label{thm:ltn}Let supp$( \mu)$ be the support of $ \mu$ and let $d \geq 1$ be the largest integer such that supp$(\mu)\backslash\{0\}$ is
contained in $b+d \Z$ for some $b \in \N$. Then choose $b$ minimal such that the preceding property holds.
\begin{enumerate}
\item[(i)] There exists an integer $N>0$ such that the following holds. For every $n \geq N$, $\Prmu{\lt =
n+1} > 0$ if, and only if, $n$ is a multiple of $\gcd(b-1,d)$.
\item[(ii)] We have:
$$\Prmu{\lt = n+1} \quad \mathop{\sim}_{n \rightarrow \infty} \quad\mu_0^{1/\theta} p_1(0) \frac{\gcd(b-1,d)}{h(n) n^{ 1/\theta+1}},$$
when $n$ tends to infinity in the set of multiples of $\gcd(b-1,d)$. Here we recall that $p_1$ is the continuous density of the law of $X_1$, where $(X_t)_ {t \geq 0}$ is the spectrally
positive Lévy process with Laplace exponent $ \E[ \exp(- \lambda
X_t)]= \exp(t \lambda ^ \theta)$.  
\end{enumerate}
In particular, when the second moment of $\mu$ is finite :
$$\Prmu{\lt = n+1} \quad\mathop{\sim}_{n \rightarrow \infty} \quad\sqrt{\frac{\mu_0}{2\pi
\sigma^2}}\quad \frac{\gcd(b-1,d)}{n^{3/2}},$$ when $n$ tends to infinity in the set of multiples of $ \gcd(b-1,d)$.
\end{thm}

Note that supp$(\mu)\backslash\{0\}$ is non-lattice if and only if
$d=1$. It is crucial to keep in mind that even if $\mu$ is
aperiodic, supp$(\mu)\backslash\{0\}$ can be lattice (for example if
the support of $\mu$ is $\{0,4,7\}$).

\begin{rem}In the case where $\mu$ has finite variance, Theorem \ref{thm:ltn} is
a consequence of results contained in \cite{Minami}.\end{rem}

Before giving the proof of Theorem \ref{thm:ltn}, let us mention a useful
consequence.
\begin{cor}\label{cor:Lambda}
 Fix $\delta >0$ and $\eta \in (0,1)$. We
have:
$$\Pmu\left[ \sup_{\eta\leq t \leq 1} \left| \frac{\Lambda_\tau(\zt t)}{ \zt t} - \mu_0\right|\geq
\frac{\delta}{n^{1/4}} \,| \,  \lt=n \right] = oe_{1/2}(n),$$
when $n-1$ tends to infinity in the set of multiples of $ \gcd(b-1,d)$.
\end{cor}

This bound is an immediate consequence of Corollary \ref {cor:GD2mieux} once we know that $ \Prmu { \lt=n}$ decays like a power of $n$.

\subsection{The Non-Lattice case}
\label {sec:nonlatice}

We consider a random variable $Y$ on $\N$ with distribution:
\begin{equation}\label{eq:Y}\P[Y=i]= \frac{1}{1-\mu_0} \mu(i+1)=\frac{1}{1-\mu_0} \nu(i), \qquad i
\geq 0.\end{equation}We will first establish Theorem \ref{thm:ltn} when $Y$ is
non-lattice, that is $b=0$ and $d=1$ in the notation of Theorem
\ref{thm:ltn}.

 In agreement with the notation of Proposition \ref {prop:K}, we consider the random walk $W'$ defined as $W$ conditioned on having nonnegative jumps. In particular, $W'_n$ is the sum of $n$ independent copies of the random variable $Y$, which is in the domain of attraction of a stable law of index $ \theta$. Indeed, when $ \sigma^2 = \infty$, this follows from the  characterization of the domain of attraction of stable
laws (see \cite[Theorem 2.6.1]{IL}). When $ \sigma^2 < \infty$, formula
(\ref{eq:Y}) shows that $Y$ has a finite second moment as well.

Consequently, if we write $B'_n$ for the quantity corresponding to $a_n$ in Theorem \ref{thm:locallimit} when $W$ is replaced by $W'$, we have, by Theorem \ref{thm:locallimit} (iii):
\begin{equation}\label{eq:ll}\lim_{n \rightarrow \infty} \sup_{k \in
\Z} \left| B'_n \P[W'_n=k]-p_1\left( \frac{k-n
\E[Y]}{B'_n}\right) \right|=0.\end{equation} Moreover, there exists
a slowly varying function $h': \R_+ \rightarrow \R_+$ such that
$B'_n=h'(n) n^{1/\theta}$, and $h'(n) \rightarrow \infty$ as $ n \rightarrow \infty$ when both $ \sigma^2 = \infty$ and $ \theta=2$. In the case where the second moment of
$\mu$ is finite, we have $ B'_n = \sigma' \sqrt{n/2}$ where $\sp$ is the
variance of $Y$. Note also that $\E[Y]= {\mu_0}/({1-\mu_0})$.

The following lemma establishes an important link between $h$ and
$h'$.

\begin{lem}\label{lem:hhp}If $ \sigma^2 = \infty$ we have
$ \d\lim_{n \rightarrow \infty} {B'_n}/{B_n}=\lim_{n \rightarrow \infty}
{h'(n)}/{h(n)}=(1-\mu_0)^{-1/\theta}.$
\end{lem}

\begin{proof}First assume that $ \theta < 2$.
Since $\P[Y \geq x]= \frac{1}{1-\mu_0}\P[W_1 \geq x]$ for $x \geq
0$, by Theorem \ref {thm:locallimit} (i),  we have for $n$ large
enough:
$$B'_n = \Gamma(1-\theta)^{1/\theta} \inf \left\{ x \geq 0; \, \P[Y
\geq x] \leq
\frac{1}{n}\right\}=\Gamma(1-\theta)^{1/\theta} \inf \left\{ x \geq 0; \, \P[W_1 \geq
x] \leq
\frac{1-\mu_0}{n}\right\}.
$$
Thus $ B_{\fl{{n}/(1-\mu_0)}}\leq B'_n
\leq B_{\left\lceil{n}/(1-\mu_0) \right\rceil}$, and the conclusion easily follows. The proof in the case $ \theta = 2$ is similar and is left to the reader.
\end{proof}

We will use the following refinement of the local limit theorem (see \cite[Chapter 7, P10]{Spitzer} for a proof).

\begin{thm}[Strong Local Limit Theorem]\label{thm:strongLL}Let $Z=(Z_n)_{n \geq 0}$ be a random walk on $\mathbb{Z}$ with jump distribution
$\rho$ started from $0$, where $\rho$ is a non-lattice probability
distribution on $\Z$. Assume that the second moment of $\rho$ is
finite. Denote the mean of $\rho$ by $m$ and its variance by $ \widetilde{ \sigma}^2$. Set $ \widetilde{a}_n= \widetilde{ \sigma} \sqrt {n/2}$. Then:
$$\lim_{n \rightarrow \infty} \sup_{x \in \Z} \left( 1 \vee
\frac{(x-mn)^2}{n}\right) \left| \widetilde{a}_n
\P[Z_n=x]-p_1\left( \frac{x-m
n}{ \widetilde{a}_n}\right) \right|=0.$$ \end{thm}

\begin{proof}[Proof of Theorem \ref{thm:ltn} when $Y$ is non-lattice] We first show that $h'(n) n^{1/\theta+1} \Pmu[ \lt=n]$ converges to
a positive real number.
 Fix $\e \in (0,1/2)$ and write:
$$ \Pmu[ \lt=n] = \Pmu\left[ \lt=n,
\, (1-\e) \frac{n}{\mu_0} \leq \zt \leq
(1+\e) \frac{n}{\mu_0}\right] +  \Pmu\left[ \lt=n,  \,
\left|\frac{\mu_0\zt}{n}-1\right|> \e \right].$$ By
Proposition \ref{prop:slow}, there exists $C>0$ such that $h'(n) \leq C
n$ for every positive integer $n$. Moreover, for $n$ large
enough, for every $x>0$, the property $|x \mu_0/n-1| \geq \e$ implies $|x-n/\mu_0| \geq
x^{3/4}$. Consequently:
$$h'(n)n^{1/\theta+1}  \Pmu\left[ \lt=n,  \, \left|\frac{\mu_0\zt}{n}-1\right|> \e\right] \leq
Cn^{1/\theta+2}  \Pmu\left[ \lt=n,  \,
\left|\zeta(\tau)-\frac{n}{\mu_0}\right| \geq \zt^{3/4} \right],$$
which is $oe_{1/2}(n)$ by Lemma \ref{lem:GdesDev2} (ii). It is thus
sufficient to show that: \begin{equation}\label{eq:ecrire}h'(n)
n^{1/\theta+1} \Pmu\left[ \lt=n, \, (1-\e) \frac{n}{\mu_0} \leq \zt
\leq (1+\e) \frac{n}{\mu_0}\right]\end{equation}converges to a
positive real number.

In the following, $S_p$ will denote the sum of $p$ independent
Bernoulli variables of parameter $\mu_0$. Note that $S_1$ is
non-lattice. The key idea is to
write the quantity appearing in (\ref{eq:ecrire}) as a sum, then rewrite it as an
integral and finally use the dominated convergence theorem. For $ x \in \R$, denote the smallest integer greater than or equal to $x$  by $ \ce {x}$. To simplify notation, we write $ \mathcal{O}(1)$ for a bounded sequence indexed by $n$ and $o(1)$ for a sequence indexed by $n$ which tends to $0$. Using Proposition \ref{prop:K}, we write:
\begin{eqnarray}
&&\d  h'(n) n^{1/\theta+1} \Pmu\left[ \lt=n,  \, (1-\e)
\frac{n}{\mu_0} \leq \zt \leq
(1+\e) \frac{n}{\mu_0}\right] \notag\\
&&\qquad = \quad \d h'(n) n^{1/\theta+1 }
\sum_{p= \ce{(1-\e)n/\mu_0}}^{\fl{(1+\e)n/\mu_0}} \frac{1}{p} \,
\P[S_p=n] \, \P[W'_{p-n}=n-1] \notag\\
&&\qquad = \quad \d \int_{- \frac{\e}{\mu_0} {n}+ \mathcal{O}(1)}
^{\frac{\e}{\mu_0} {n}+ \mathcal{O}(1)} dx \, \frac{h'(n)
n^{1/\theta+1}}{\fl{n/\mu_0+x}} \, \P[S_{\fl{n/\mu_0+x}}=n] \,
\P[W'_{\fl{n/\mu_0+x}-n}=n-1] \notag\\
&&\qquad = \quad \d \int_{- \frac{\e}{\mu_0} \sqrt{n}+o(1)}
^{\frac{\e}{\mu_0} \sqrt{n}+o(1)} du \, \frac{ \sqrt {n} h'(n) n^{1/\theta+1}}{\fl{n/\mu_0+u \sqrt{n}}}  \P[S_{\fl{n/\mu_0+u \sqrt{n}}}=n] \,
\P[W'_{\fl{n/\mu_0+u \sqrt{n}}-n}=n-1] \label{eq:somint}	
\end{eqnarray}
Using the case $ \theta=2$ of Theorem \ref{thm:locallimit} (iii), for fixed $u \in \R$, one sees
that:
\begin{eqnarray*} \lim_{n \rightarrow \infty }\sqrt{n} \P[S_{\fl{n/\mu_0+u
\sqrt{n}}}=n]&=&\frac{1}{\sqrt{2 \pi (1-\mu_0)}} e^{- \frac{\mu_0^2
}{2(1-\mu_0)} u^2}.\end{eqnarray*}

We now claim that there exists a bounded function $F: \R
\rightarrow \R_+$ such that: \begin{equation}
\label{eq:b}\lim_{n \rightarrow \infty }h'(n)
n^{1/\theta} \P[W'_{[n/\mu_0+u \sqrt{n}]-n}=n-1]=F(u)
\end{equation}for
every fixed $u \in \R$. We distinguish two cases. When $ \sigma^2 = \infty$, we have by (\ref{eq:ll}):
\begin{eqnarray*} && \lim_{n \rightarrow \infty }h'(n) n^{1/\theta} \P[W'_{\fl{\frac{n}{\mu_0}+u
\sqrt{n}}-n}=n-1]  \\
&& \qquad \qquad \qquad \qquad \qquad \qquad=  \left(\frac{\mu_0}{1-\mu_0}\right)^{\frac{1}{\theta}} \lim_{n \rightarrow \infty }
p_1\left(\frac{n-1-\frac{\mu_0}{1-\mu_0} \left( \fl{\frac{n}{\mu_0}+u
\sqrt{n}}-n \right) }{B'_{\fl{\frac{n}{\mu_0}+u \sqrt{n}}-n}}\right)\\
&& \qquad \qquad \qquad \qquad \qquad \qquad= \left(\frac{\mu_0}{1-\mu_0}\right)^{\frac{1}{\theta}}
p_1(0).
\end{eqnarray*}
In the case $ \theta=2$, we use the property that $h'(n) \rightarrow \infty$ as $ n \rightarrow \infty$. When  $ \sigma^2 < \infty$, (\ref{eq:ll}) gives: \begin{eqnarray*}\lim_{n
\rightarrow \infty } \sigma' \sqrt{n/2} \, \P[W'_{\fl{\frac{n}{\mu_0}+u
\sqrt{n}}-n}=n-1]  &=& \sqrt{ \frac{ \mu_0}{1-\mu_0}} \lim_{n \rightarrow
\infty }p_{1}\left(
\frac{n-1-\frac{\mu_0}{1-\mu_0} \left( \fl{\frac{n}{\mu_0}+u \sqrt{n}}-n \right)
}{ \sigma' \sqrt{\fl{\frac{n}{\mu_0}+u
\sqrt{n}}-n} /\sqrt {2}}\right)\\
& =&
\sqrt{ \frac{ \mu_0}{1-\mu_0}} \,
{p}_{1}\left(  \frac{ \sqrt {2}}{ \sigma'} \cdot \left(\frac{\mu_0}{1-\mu_0}\right)^{3/2}u\right).\end{eqnarray*}
In both cases, we have obtained our claim \eqref{eq:b}.

Next, for $n\geq 1$ and $u \in \R$, define :
$$f_n(u)= 1_{\{|u|<\frac{2\e}{\mu_0} \sqrt{n}\}} \, \sqrt{n} \P[S_{\fl{\frac{n}{\mu_0}+u
\sqrt{n}}}=n] \quad, \quad g_n(u)=1_{\{|u|<\frac{2\e}{\mu_0}
\sqrt{n}\}} \, B'_n \P[W'_{\fl{\frac{n}{\mu_0}+u \sqrt{n}}-n}=n-1].$$
The strong version of the Local Limit Theorem (Theorem
\ref{thm:strongLL}) implies that there exists $C>0$ such that $|f_n(u)| \leq C\min\left(1,
\frac{1}{u^2} \right)$ for
all $n>1$ and $u \in \R$. To bound $g_n$, write:
$$g_n(u)=1_{\{|u|<\frac{2\e}{\mu_0}
\sqrt{n}\}} \,  \frac{B'_n}{B'_{\fl{\frac{n}{\mu_0}+u \sqrt{n}}-n}}
B'_{\fl{\frac{n}{\mu_0}+u \sqrt{n}}-n} \P[W'_{\fl{\frac{n}{\mu_0}+u
\sqrt{n}}-n}=n-1].$$
Proposition \ref{prop:slow} (ii) implies that there exists $C'>0$ such
 that ${B'_n}/{B'_{\fl{{n}/{\mu_0}+u \sqrt{n}}-n}}
 \leq C'$ for every $n$ sufficiently large
and $|u|<\frac{2\e}{\mu_0} \sqrt{n}$, and then (\ref{eq:ll}) entails that there exists $C>0$ such that for all $n>1$ and $u \in
\R$ we have $|g_n(u)| \leq C$.
By the preceding bounds on $f_n$ and $g_n$, we can apply the dominated convergence theorem to the right-hand side of \eqref{eq:somint} and we get:

\begin{equation}\label{eq:int}\lim_{n \rightarrow \infty} h'(n) n^{1/\theta+1} \Pmu[ \lt=n] =
\mu_0 \int_{- \infty}^{+ \infty}du \, F(u) \frac{1}{\sqrt{2
\pi (1-\mu_0)}} e^{- \frac{\mu_0 ^2}{2(1-\mu_0)} u^2}.\end{equation}

Finally, we need to identify the value of the integral in \eqref{eq:int} and to
express $h'$ in terms of $h$. We again distinguish two cases. First suppose that $ \sigma^2 < \infty$. An explicit computation of the right-hand side of \eqref{eq:int} gives:
$$ \frac{ \sigma'}{ \sqrt {2}} n ^ {3/2}\Prmu{\lt = n} \quad \mathop{ \longrightarrow}_{n \rightarrow \infty} \quad  \sqrt{ \frac{1}{4 \pi} \cdot \frac{ \mu_0 \sp}{ \mu_0/(1- \mu_0)+ \sp(1- \mu_0)}}.$$ 
A simple calculation gives $\sp=(\sigma^2-\mu_0)/(1-\mu_0)-\left({\mu_0}/(1-\mu_0)
\right)^2$, which entails: $$\Prmu{\lt = n}  \quad
\mathop{\sim}_{n \rightarrow \infty}\quad  \sqrt{ \frac{\mu_0}{2\pi \sigma^2}} {n^{ - 3/2}}.$$

When $ \sigma^2 = \infty$, we have
$F(u)=\left({\mu_0}/(1-\mu_0)\right)^{\frac{1}{\theta}}
p_1(0)$ so that \eqref{eq:int} immediately gives that
$h'(n) n^{1/\theta+1} \Pmu[ \lt=n]$ converges towards $
\left({\mu_0}/(1-\mu_0)\right)^{\frac{1}{\theta}} p_1(0)$ as $n \rightarrow \infty$.
By Lemma \ref{lem:hhp}, we conclude that: $$ \Pmu[ \lt=n]\quad
\mathop{\sim}_{n \rightarrow \infty}\quad \mu_0^{\frac{1}{\theta}}
p_1(0) \frac{1}{h(n) n^{1/\theta+1}}.$$ Note that this formula
is still valid for $\theta=2$. This concludes the proof in the non-lattice case.
\end{proof}

\subsection{The Lattice case}

We now sketch a proof of Theorem \ref{thm:ltn} when $Y$ is lattice.

\begin{proof}[Proof of Theorem \ref{thm:ltn} when $Y$ is lattice]
For (i), by Proposition \ref{prop:K},
$\Prmu{\lt=n+1}>0$ if and only if there exists $k \geq 0$ such that
$\P[W'_k=n]>0$. As a consequence, $\Prmu{\lt=n+1}>0$ if and only if $n$ can be written as a sum of elements of $\textnormal {supp} (Y)$. Since $ \textnormal{supp} (Y) \subset b-1+d \Z$, it follows that $\Prmu{\lt=n+1}=0$ if $n$ is not divisible by $\gcd(b-1,d)$, and it is an easy number theoretical exercise to show that there exists $N>0$ such that for $n \geq N$, $\Prmu{\lt =
n+1} > 0$ if $n$ is a multiple of $\gcd(b-1,d)$.

The asymptotic estimate of (ii) is obtained exactly as in the
non-lattice case by making use of the Local Limit Theorem for
lattice random variables (see e.g. \cite[Theorem 4.2.1]{IL}). We omit the argument to avoid technicalities.
\end{proof}

\begin{rem}Let us briefly discuss the extension of the preceding results to the case where $ \mu$ is periodic. In this case, $Y$ is necessarily lattice. Indeed, the property
$ \supp(\mu) \subset d \Z$ implies $ \supp(Y)  \subset d \Z-1$. The same reasoning as above shows that
Theorem \ref{thm:ltn} remains valid in this case. However,
the span of supp$(Y)$ is not necessarily equal to the span of supp$(\mu)$. Consequently, $\Prmu{\lt=n}=0$ can hold for infinitely many $n$
(for example if the support of $\mu$ is $\{0,28,40,52\}$) or for
finitely many $n$ (for instance if the support of $\mu$ is
$\{0,3,6\}$).\end{rem}

\section{Conditioning on having at least $n$ leaves}

In this section, we show that the scaling
limit of a $\GW_\mu$ tree conditioned on having at
least $n$ leaves is the same (up to constants) as that of a $\GW_\mu$ tree
conditioned on having total progeny at least $n$. The argument goes as
follows. By the large deviation result obtained in Section 2 (which
states that if a $\GW_\mu$ tree has $n$ leaves, then the probability
that its total progeny does not belong to $[n/\mu_0 - n^{3/4},
n/\mu_0+n^{3/4}]$ decreases exponentially fast in $n$), we establish
that the probability measures $\Pmuzgeqn$ and $\Pmu[ \, \cdot \, |
\, \lt \geq \mu_0 n -n^{3/4}]$ are close to each other for large $n$.
The fact that the rescaled contour function of a $\GW_\mu$ tree under
$\Pmuzgeqn$ converges in distribution then allows us to conclude.

Henceforth, if $I$ is a closed subinterval of $\R_+$, $\C(I,\R)$
stands for the space of all continuous functions from $I$ to $\R_+$,
which is equipped with the topology of uniform convergence on every
compact subset of $I$.

Recall that $ \mu$ is a probability distribution on $ \N$ satisfying the hypothesis $ (H _ \theta)$ for some $ \theta \in (1,2]$. Recall also the definition of the sequence $(B_n)$, introduced just before Lemma
\ref{lem:locallimit}. Also recall the notation $C_t (\tau)$ for the contour function of a tree $ \tau$ introduced in Definition \ref {def:fonctions}.

\begin{thm}[Duquesne]\label{thm:cv1}There exists a random continuous
function on $[0,1]$ denoted by $\H$ such that if $\t_n$ is a tree
distributed according to $\Pmu[\, \cdot \, | \, \zt =n]$:
$$\left(\frac{B_{n}}{n} C_{2nt} (\t_n);\, 0 \leq t \leq 1 \right)
\quad\mathop{\longrightarrow}^{(d)}_{n \rightarrow \infty}
\quad(\H_t;\, 0 \leq t \leq 1),$$ where the convergence holds in the sense
of weak convergence of the laws on $\C([0,1],\R)$.
\end{thm}

\begin{proof}See \cite[Theorem 3.1]{Duquesne} or \cite{Kautre}.\end{proof}

\begin{rem} The random function $\H$, can be identified as the
normalized excursion of the height process associated to the spectrally positive stable process $X$. The notion of the height process was
introduced in \cite{LeJan} and studied in great detail in
\cite{DuquesneLG}; see Section 5.1 for a definition. \end{rem}

Using Theorem \ref{thm:cv1}, we shall prove that for every bounded nonnegative
continuous function $F$ on $\C([0,1],\mathbb{R})$  the following
convergence holds:
\begin{equation}\label{eq:cvamontrer}\E_\mu \left[ \left. F\left(\frac{B_{\zt}}{\zt} C_{2\zt t} (\tau); \, 0 \leq t \leq 1 \right) \right | \,
\lt\geq n\right]\quad\mathop{\longrightarrow}_{n \rightarrow
\infty}\quad \E[F(\H)].\end{equation}

Recall that $\Pmuj$ stands for the probability measure on $\T^j$
which is the distribution of $j$ independent $\GW_\mu$ trees.

\begin{lem}\label{lemma:passage}Fix $1 \leq j \leq n$. Let $U$ be a bounded nonnegative measurable
function on $ \mathbb {T}^j$. Then: $$|\Esmuj{U(\bf) 1_{{\zeta(\bf)}
\geq n}}-\Esmuj{U(\bf) 1_{{\l(\bf)} \geq \mu_0n-n^{3/4}}}| \leq
 \norme{U} \Pmuj\left[n-(\mu_0^{-1}+1)n^{3/4} \leq
{\zeta(\bf)} \leq n\right]+oe_{1/2}(n)$$ where the estimate $oe_{1/2}(n)$ is
uniform in $j$.
\end{lem}

\begin{proof}
First note that:
\begin{eqnarray}\Esmuj{U(\bf) 1_{{\zeta(\bf)} \geq n}} &=& \d
 {\Esmuj{U(\bf) 1_{\zeta(\bf) \geq n} 1_{{\l(\bf)} \geq \mu_0 {\zeta(\bf)} -
{\zeta(\bf)}^{3/4} }}}+
\d {\Esmuj{U(\bf) 1_{\zeta(\bf) \geq n} 1_{{\l(\bf)} < \mu_0 {\zeta(\bf)} - {\zeta(\bf)}^{3/4} }}} \notag\\
&=&  \d {\Esmuj{U(\bf) 1_{\zeta(\bf) \geq n} 1_{{\l(\bf)} \geq \mu_0
{\zeta(\bf)} - {\zeta(\bf)}^{3/4} }}} + oe_{1/2}(n), \label {eq:lempassage}\end{eqnarray}
where we have used Lemma
\ref{lem:GdesDev2} (i) in the last equality.

 Secondly, write:
\begin{eqnarray*}\Esmuj{U(\bf) 1_{{\l(\bf)} \geq \mu_0 n - n^{3/4}}}&=& \Esmuj{U(\bf) 1_{{\l(\bf)} \geq \mu_0 n - n^{3/4}} 1_{{\zeta(\bf)} \geq n}}+
\Esmuj{U(\bf) 1_{{\l(\bf)} \geq \mu_0 n - n^{3/4}} 1_{{\zeta(\bf)} <
n}}
\end{eqnarray*}
Let $C_n$ and $D_n$ be respectively the first and the second term appearing in the right-hand side. To simplify notation, set $ \a(n)=  \mu_0 n - n^{3/4}$ for $n \geq 1$. Then, by
Lemma \ref{lem:GdesDev2} (ii), we have for $n$ large enough:
\begin{eqnarray*}D_n&=&\Esmuj{U(\bf) 1_{{\l(\bf)} \geq   \a(n), \, {\zeta(\bf)} < n, \, \left| {\zeta(\bf)}-\frac{{\l(\bf)}}{\mu_0} \right| \leq {\zeta(\bf)}^{3/4}}} +  oe_{1/2}(n) \\
& \leq &  \Esmuj{U(\bf) 1_{{\l(\bf)} \geq  \a(n), \, {\zeta(\bf)} < n, \, {\zeta(\bf)} \geq \frac{{\l(\bf)}}{\mu_0} - {\zeta(\bf)}^{3/4}}} +  oe_{1/2}(n)  \\
& \leq& \norme {U} \Prmuj{n-(\mu_0^{-1}+1)n^{3/4} \leq {\zeta(\bf)} \leq n}+  oe_{1/2}(n)
\end{eqnarray*}
We next consider $C_n$.   Choose $n$ sufficiently large so that the
function $x \mapsto \a(x)$ is increasing over $[\mu_0 n
- n^{3/4}, \infty)$ and write:
\begin{eqnarray*}
C_n&=&\Esmuj{U(\bf)  1_{{\zeta(\bf)} \geq n, \, {\l(\bf)} \geq   \a(\zeta(\bf)) }} + \Esmuj{U(\bf) 1_{{\l(\bf)} \geq  \a (n) , \, {\zeta(\bf)} \geq n, \, {\l(\bf)} < \a( {\zeta(\bf)})  }} \\
&\leq& \Esmuj{U(\bf)  1_{{\zeta(\bf)} \geq n,  \, {\l(\bf)} \geq  \a(\zeta(\bf))  }} + \Esmuj{U(\bf)  1_{{\zeta(\bf)} \geq n,  \, {\l(\bf)} < \a(\zeta(\bf))  }} \\
&=& \Esmuj{U(\bf)  1_{{\zeta(\bf)} \geq n, \, {\l(\bf)} \geq  \a(
{\zeta(\bf)})  }} + oe_{1/2}(n)
\end{eqnarray*}
by Lemma \ref{lem:GdesDev2} (i).

By the preceding estimates we have, for $n$ large:
\begin{eqnarray*}
0 &\leq& \Esmuj{U(\bf) 1_{{\l(\bf)} \geq \mu_0 n - n^{3/4}}} - {\Esmuj{U(\bf) 1_{\zeta(\bf) \geq n} 1_{{\l(\bf)} \geq \mu_0
{\zeta(\bf)} - {\zeta(\bf)}^{3/4} }}} \\
&\leq&  \norme {U}
\Prmuj{n-(\mu_0^{-1}+1)n^{3/4} \leq {\zeta(\bf)} \leq n}+oe_{1/2}(n)
\end{eqnarray*}
and by combining this bound with \eqref{eq:lempassage} we get the desired estimate.\end{proof}

\begin{prop}\label{prop:passage}Let $U, (U_n)_{n \geq 1}: \mathbb{T} \rightarrow \R_+$ be uniformly bounded measurable
functions, meaning that there exists $M>0$ such that for all $n \geq
1$ and $ \tau \in \T$, $U_n(\tau) \leq M$ and $U( \tau) \leq M$.
\begin{enumerate}
\item[(i)] If $\Esmu{U(\tau) | \, \zt = n}$ converges when $n$ tends to
infinity, then $\Esmu{U(\tau) | \,\zt\geq n}$ converges to the same
limit.
\item[(ii)] If $\Esmu{U_n(\tau) | \,\zt \geq n}$ converges when $n$ tends to
infinity, then $\Esmu{U_{n}(\tau) | \, {\lt}\geq \lceil \mu_0 n -{n} ^{3/4} \rceil}$ converges to
the same limit.
\end{enumerate}
\end{prop}

\begin{proof}
Using the formula:
$$\E_\mu \left[U(\tau) \, | \, \zeta(\tau)\geq
n\right]=\frac{1}{\Pmu[\zeta(\tau) \geq n]}
\sum_{k=n}^\infty\P_\mu[\zt=k] \cdot \E_\mu\left[U(\tau) \, | \,
\zeta(\tau)=k\right]$$ it is an easy exercise to verify that the first
assertion is true.

We turn to the proof of (ii). Fix $0 < \eta < 1/4$. By Lemma \ref{lem:locallimit}, we may suppose that $n$ is sufficiently large so that $ \Prmu { \zt \geq n} \geq c_3 n^ {-1 / \theta- \eta}$ for a constant $c_3>0$. Next, setting again $ \a(n)=\mu_0n-n^{3/4}$, we have:
\begin{eqnarray*}&& \left|\Esmu{U_n(\tau) | \zt \geq n}-\Esmu{U_n(\tau) |
\lt\geq \a(n)} \right|\\
 && \quad   \leq \quad  \left|\frac{\Esmu{U_n(\tau) 1_{\zt
\geq n}}}{\Prmu{\zt \geq n}}-\frac{\Esmu{U_n(\tau) 1_{\lt\geq
\a(n)}}}{\Prmu{\zt \geq n}} \right| + \frac{\Esmu{U_n(\tau)
1_{\lt\geq \a(n)}}}{\Prmu{ \lt \geq \a(n) }} \left|
\frac{\Pmu[\lt\geq \a(n)]}{\Pmu[\zt
\geq n]}- 1 \right| \\
 && \quad   \leq \quad M
\frac{n^{1/\theta+\eta}}{c_3}\Pmu\left[n-(\mu_0^{-1}+1)n^{3/4} \leq
\zt \leq n\right]+M\left| \frac{\Pmu[\lt\geq
\a(n)]}{\Pmu[\zt \geq n]}- 1 \right|+ oe_{1/2}(n),\end{eqnarray*} where
we have used Lemma \ref{lemma:passage} in the last inequality. By Lemma \ref {lem:locallimit}, the first term of the right-hand side tends to $0$. From Theorem \ref {thm:ltn} (ii), it is easy to get that $\Pmu[\lt\geq n] \sim  \theta \mu_0  ^  {1/ \theta} p_1(0)/ (n^ {1/ \theta} h(n))$ as $n \rightarrow \infty$. By combining this estimate with Lemma \ref{lem:locallimit}, we obtain that ${\Pmu[\lt\geq
 \alpha(n)]}/{\Pmu[\zt \geq n]}$ tends to $1$ as $n \rightarrow \infty$. This completes the proof.\end{proof}

\begin{thm}\label{thm:inv1}
For $n \geq 1$, let $\t_n$ be a random tree distributed according to
$\Pmu[\, \cdot \, | \, \lt \geq n]$. Then:
\begin{equation}\label{eq:thminv1}\left(\frac{B_{{\z(\t_n)}}}{{\z(\t_n)}}
C_{2{\z(\t_n)} t} (\t_n); \, 0 \leq t \leq 1 \right)
\quad\mathop{\longrightarrow}^{(d)}_{n \rightarrow \infty} \quad
(\H_t;\, 0 \leq t \leq 1),\end{equation} where the convergence holds in
the sense of weak convergence of the laws on $\C([0,1],\R)$.
\end{thm}

\begin{proof}Let $F$ be a bounded nonnegative continuous function on
$\C([0,1],\mathbb{R})$. By Theorem \ref{thm:cv1}:
$$\E_\mu \left[ \left. F\left(\frac{B_{{\z(\tau)}}}{{\z(\tau)}} C_{2{\z(\tau)} t} (\tau); \, 0 \leq t \leq 1 \right) \right|\, \zt=n \right]\quad\mathop{\longrightarrow}_{n
\rightarrow \infty} \quad \E[F(\H)].$$ Proposition
\ref{prop:passage} (i) entails:
$$\E_\mu \left[ \left. F\left(\frac{B_{{\z(\tau)}}}{{\z(\tau)}} C_{2{\z(\tau)} t} (\tau); \, 0 \leq t \leq 1 \right) \right|\, \zt \geq n \right]\quad\mathop{\longrightarrow}_{n
\rightarrow \infty} \quad \E[F(\H)].$$Proposition \ref{prop:passage}
(ii) then implies:
$$\E_\mu \left[ \left. F\left(\frac{B_{{\z(\tau)}}}{{\z(\tau)}} C_{2{\z(\tau)} t} (\tau); \, 0 \leq t \leq 1 \right) \right|\,   {\lt}\geq \lceil \mu_0 n -{n} ^{3/4} \rceil \right]\quad\mathop{\longrightarrow}_{n \rightarrow
\infty} \quad \E[F(\H)].$$ Since $\lceil \mu_0 n -{n} ^{3/4} \rceil$ takes all positive integer values when $n$ varies, the proof is complete.
\end{proof}

\begin{rem}When the second moment of $\mu$ is finite, $\H=\sqrt{2} \mathbbm {e}$ where $ \mathbbm {e}$ denotes the normalized
excursion of linear Brownian motion. Since the scaling constants
$B_n=\sigma \sqrt{n/2}$ are known explicitly, in that case the theorem can be formulated as:
$$\left(\frac{\sigma}{2 \sqrt{\z(\t_n)}} C_{2{\z(\t_n)} t} (\t_n); \, 0 \leq t \leq 1 \right)
\quad\mathop{\longrightarrow}^{(d)}_{n \rightarrow \infty} \quad
\me.$$
\end{rem}

\section{Conditioning on having exactly $n$ leaves}

\begin{center}\emph{To avoid technical issues, we suppose that $ \supp ( \mu) \backslash  \{0\}$ is non-lattice, so that
$\Prmu{\lt=n}>0$ for $n$ large enough.}
\end{center}

Recall that we have obtained an invariance principle for $\GW_\mu$
trees under the probability distribution $\Pmu[\, \cdot \, | \, \lt
\geq n]$. Our goal is now to establish a similar result for trees under the probability distribution $\Pmu[\, \cdot \,
| \, \lt = n]$. The key idea is to use an ``absolute continuity''
property. Let us briefly sketch the main step of
the argument.

Let $k \geq 1$. If $\tau$ is a tree and if $u(0), u(1), \ldots$ are the vertices of $ \tau$ in lexicographical order, let $T_k( \tau)$ be the first index $j$ such that $  \{u(0),u(1), \ldots, u(j)\}$ contains $k$ leaves and $T_k ( \tau)= \infty$ if there is no such index. Fix $a \in (0,1)$ and recall the notation $ \W( \tau)$ for the Lukasiewicz path of a tree $ \tau$. Then there exists a positive function $D ^n _a$ on $ \Z_+$ such that, for every nonnegative function $f$ on the space of finite paths in $ \Z$:
$$\Esmu{f\left( \W_{ \cdot \wedge T_{ \fl { an }} (\tau)}(\tau)\right) \, | \, \lt=n}=\Esmu{f\left( \W_{ \cdot \wedge T_{ \fl { an }} (\tau)}(\tau)\right) D^n_a( \W_{T_ { \fl { an } }(\tau) }(\tau)) \, | \,  \lt \geq n}.$$
By combining the invariance principle for trees under $\Pmu[\, \cdot \, | \, \lt
\geq n]$ together with estimates for $D^n_a( \W_{T_ { \fl { an } }(\tau) }(\tau))$ as $ n \rightarrow \infty$, we shall deduce an  invariance principle for
trees under $\Pmu[\, \cdot \, | \, \lt = n]$.

\subsection{The normalized excursion of the Lévy process}

\label{sec:lev}

We follow the presentation of \cite{Duquesne}. The underlying
probability space will be denoted by $(\Omega, \mathcal{F}, \P)$.
Let $X$ be a process with paths in $\D(\R_+, \R)$, the space of
right-continuous with left limits (càdlàg) real-valued functions,
endowed with the Skorokhod $J_1$-topology. We refer the reader to
\cite[chap. 3]{Bill} and \cite[chap. VI]{Shir} for background
concerning the Skorokhod topology. We denote
the canonical filtration generated by $X$ and augmented with the
$\P$-negligible sets by $(\F_t)_{t \geq 0}$. In agreement with the notation in the previous sections, we assume that $X$ is a strictly stable
spectrally positive Lévy process with index $\theta \in (1,2]$ such
that for $\lambda>0$:
\begin{equation}\label{eq:X}\E[\exp(-\lambda X_t)]=\exp(t
\lambda^\theta).\end{equation} See \cite{Bertoin} for the proofs of
the general assertions of this subsection concerning Lévy processes.
In particular, for $\theta=2$ the process $X$ is $\sqrt{2}$ times
the standard Brownian motion on the line. Recall that $X$ has the
following scaling property: for $c>0$, the process $(c^{-1/\theta}
X_{ct}, t \geq 0$) has the same law as $X$. In particular, if we
denote by $p_t$ the density of $X_t$ with respect to the Lebesgue
measure, $p_t$ enjoys the following scaling property:
\begin{equation}\label{eq:scaling}p_ { \lambda s}(x)= \lambda^{-1/\theta} p_s(x \lambda^{-1/\theta})\end{equation}
for $x \in \R$ and $s, \lambda > 0$. The following
notation will be useful: for $s<t$ set
$$I^s_t = \inf_{[s,t]} X, \qquad I_t=\i{0,t} X.$$
Notice that the process $I$ is continuous since $X$ has no negative
jumps.

The process $X-I$ is a strong Markov process and $0$ is regular for
itself with respect to $X-I$. We may and will choose $-I$ as the
local time of $X-I$ at level $0$. Let $(g_i,d_i), i \in \mathcal{I}$
be the excursion intervals of $X-I$ above $0$. For every $i \in \mathcal{I}$ and $ s \geq  0$, set $\omega_s^i= X_{(g_i+s) \wedge d_i}-X_{g_i}$.
We view $ \omega^i$ as an element of the excursion space $ \mathcal {E}$, which is defined by:
$$ \mathcal{E}=  \{ \omega \in \D(\R_+, \R_+); \, \omega(0)=0 \textrm{ and } \z( \omega):=
 \sup\{s>0; \omega(s)>0\,\} \in (0, \infty)  \}.$$
From Itô's excursion theory, the point measure
$$\mathcal{N}( dt d \omega)= \sum_{i \in \mathcal{I}} \delta_{(-I_{g_i},\omega^i)}$$
is a Poisson measure with intensity $dt \Nn(d\omega)$, where $\Nn(d
\omega)$ is a $\sigma$-finite measure on $ \mathcal{E}$ which is called the It\^o excursion measure. Without risk of confusion, we will also use the
notation $X$ for the canonical process on the space $\D(\R_+, \R)$.

Let us define the normalized excursion of $X$. For every $\lambda>0$, define the re-scaling operator $S^{(\lambda)}$
on the set of excursions by:
$$S^{(\lambda)}(\omega)=\left( \lambda^{1/\theta} \omega(s/\lambda), \, s \geq 0\right).$$
Note that $ \Nn(\zeta>t) \in (0, \infty)$ for $t>0$. The scaling property of X shows that the image of $\Nn(\cdot \,| \,
\zeta>t)$ under $S^{(1/\zeta)}$ does not depend on $t>0$. This
common law, which is supported on the càdlàg paths with unit
lifetime, is called the law of the normalized excursion of $X$ and
denoted by $\Nn( \, \cdot \, | \z=1)$. We write $ \X = ( \X_s, 0 \leq s \leq  1)$ for a process distributed according to $\Nn( \, \cdot \, | \z=1)$. In particular, for
$\theta=2$ the process $\X$ is $\sqrt{2}$ times the normalized
excursion of linear Brownian motion. Informally,  $\Nn( \, \cdot \,
| \z=1)$ is the law of an excursion under the Itô measure
conditioned to have unit lifetime.

We will also use the so-called continuous-time height process $H$
associated with $X$ which was introduced in \cite{LeJan}. If
$\theta=2$, $H$ is set to be equal to $X$. If $\theta \in
(1,2)$, the process $H$ is defined for every $t \geq 0$ by:
$$H_t:=\lim_{\e \rightarrow 0} \frac{1}{\e} \int_0 ^t
1_{\{X_s<I^s_t+\e\}}ds,$$ where the limit exists in $\P$-probability and in
$\Nn$-measure on $  \{t < \zeta\}$. The definition of $H$ thus makes sense under $ \P$ or under $\Nn$. The process $H$ has a
continuous modification both under $ \P$ and under $ \Nn$ (see \cite[Chapter 1]{DuquesneLG} for
details), and from now on we consider only this modification. Using simple scaling arguments one can also define $H$ as a continuous random process under $\Nn( \, \cdot \, | \z=1)$. Let us finally mention that the limiting process $ \H$ in Theorem \ref{thm:cv1} has the distribution of $H$ under
$\Nn( \, \cdot \, | \z=1)$.

\subsection{An invariance principle}

Recall that the Lukasiewicz path $ \W(\tau)$ of a tree $\tau \in \T$
is defined up to time $\zt$. We extend it to $\Z_+$ by setting
$ \W_i(\tau)=0$ for $i \geq \zt$. Similarly, we extend the height
function $H(\tau)$ to $\Z_+$ by setting $H_i(\tau)=0$ for $i \geq
\zt$. We then extend $H(\tau)$ to $\R_+$ by linear interpolation,
$$H_t (\tau) = (1 - \{t\})H_{ \fl{t}}(\tau) + \{t\}H_{\fl{t}+1}(\tau), \qquad t \geq 0,$$
where $\{t\} = t-\fl{t}$.

Recall that $ \mu$ is a probability distribution on $ \N$ satisfying the hypothesis $ (H _ \theta)$ for some $ \theta \in (1,2]$. Recall also the notation $h$, $B_n$ introduced just before Lemma
\ref{lem:locallimit}. For technical reasons, we put $B_u = B_ {\fl{u}}$ for $u \geq 1$. It is useful to keep in mind that $B_n= \sigma
\sqrt{n/2}$ when the variance $\sigma^2$ of $\mu$ is finite. We rely on the following theorem.

\begin{thm}[Duquesne \& Le Gall]\label{thm:cvW}Let $\t_n$ be a random tree distributed according to $\Pmuzgeqn$. We have:
$$ \left( \frac{1}{B_n} \W_{ \lfloor nt \rfloor}(\t_n), \frac{B_n}{n} H_{nt}(\t_n)\right)_ {t \geq 0} \quad
\mathop{\longrightarrow}^{(d)}_{n \rightarrow \infty}\quad (
X_t,H_t)_{0 \leq t \leq 1} \textrm{ under }\Nn( \, \cdot \, | \,
\z>1). $$
\end{thm}

\begin{proof}See the concluding remark of \cite[Section 2.5]{DuquesneLG}.\end{proof}

\subsection{Absolute continuity}

Recall from the beginning of this section the definition of $ T_k ( \tau)$ for a tree $ \tau$.

\begin{prop} \label{prop:abscon}Let $n$ be a positive integer and let $k$ be an integer such that $ 1 \leq k \leq n-1$. To simplify notation, set
$ \W ^ {(k)}(\tau)=( \W_0 ( \tau), \ldots, \W_ {T_k ( \tau)} ( \tau))$. For every bounded function $f: \cup_ {i \geq 1} \Z^i \rightarrow \R_+$
 we have:
$$
\Esmu{f(W ^ {(k)}(\tau))|\, \lt =n}
=
\Esmu{ \left. f( \W ^ {(k)}(\tau)) \frac{\psi_{n-k}( \W_{ T_k ( \tau) }(\tau))/\psi_n(1)}{\psi^*_{n-k}( \W_{T_k ( \tau)}(\tau))/\psi^*_n(1)} \right|\, \lt \geq n},
$$
where
$\psi_p(j)=\Prmuj{\l(\bf)=p}$ and $\psi^*_p(j)=\Prmuj{\l(\bf) \geq p}$ for every integer $p \geq 1$.
\end {prop}

\begin {proof}
Let the random walk $W$ be as in Proposition \ref {prop:RW}. The result follows from the latter proposition and an application of the strong Markov property to the random walk $W$ at the first time it has made $k$ negative jumps. See \cite[Lemma 10]{LGM} for details of the argument in a slightly different context.
\end {proof}

We will also use the following continuous version of Proposition \ref{prop:abscon} (see \cite[Proposition 2.3]{Kautre} for a proof).

\begin{prop}\label{prop:Gamma}For $s>0$ and $x \geq 0$, set $q_s(x)=\frac{x}{s}p_s(-x)$. For every $a \in (0,1)$ and $x>0$ define:
$$\Gamma_a(x)=\frac{\theta q_{1-a}(x)}{\int_{1-a} ^ \infty ds \, q_s(x)}.$$
Then for every measurable bounded function $F : \D([0,a],\R^2) \rightarrow \R_+$:
$$\Nn \left( F( (X_t)_{0 \leq t \leq a }, (H_t)_{0 \leq t \leq a }) | \, \z=1 \right)=\Nn \left( F( (X_t)_{0 \leq t \leq a },(H_t)_{0 \leq t \leq a }) \Gamma_a(X_a) | \, \z>1 \right).$$
\end{prop}

We now control the Radon-Nikodym
density appearing in Proposition \ref {prop:abscon}. Recall that $p_s$ stands for the density of $X_s$. It is well known that $p_1$ is bounded over $\R$ and that the derivative of $q_u$ is bounded over $\R$ for every $u>0$  (see e.g. \cite[I. 4]{Zolotarev}).

\begin{lem}\label{lem:technical} Fix $\a>0$. We have:
$$(i) \lim_{n \rightarrow \infty }   \sup_{1 \leq j \leq \a B_n} \left|
\psi^*_n(j)- \int_1 ^\infty ds \, q_s\left(\frac{j}{
 B_{n/\mu_0} }\right) \right|=0, \quad (ii) \lim_{n \rightarrow \infty } \sup_{1 \leq j \leq \a B_n} \left| n \psi_n(j)-
q_1\left(\frac{j}{B_{n/\mu_0}}\right) \right|=0.$$\end{lem}

The proof of Lemma \ref{lem:technical} is technical and is postponed to Section \ref{sec:technic}.

\begin{cor} \label{cor:tejhnique}Let $r_n$ be a sequence of positive integers such that $n/r_n \rightarrow \mu_0$ as $n \rightarrow \infty$. \begin{enumerate}

\item[(i)]We have $ \d \lim_{n \rightarrow \infty }  \sup_{1 \leq j \leq \a B_n} \left|
\psi^*_{n- \fl {a \mu_0 r_n }}(j)- \int_{1-a} ^\infty ds \,
q_s\left(\frac{j}{ B_{n/\mu_0} }\right) \right|=0$.
\item[(ii)]We have $ \d \lim_{n \rightarrow \infty }  \sup_{1 \leq j \leq \a B_n} \left| n \psi_{n- \fl {a \mu_0
r_n}}(j)- q_{1-a} \left(\frac{j}{B_{n/\mu_0}}\right)
\right|=0$.
\end{enumerate}\end{cor}

\begin{proof}We shall only prove (i). The second assertion is easier and is left to the reader.
By Lemma \ref{lem:technical} (i):
$$\sup_{1 \leq j \leq \a B_n} \left|
\psi^*_ {n- \fl {a \mu_0 r_n }}(j)- \int_1 ^\infty ds \, q_s\left(\frac{j}{
 B_ {(n- \fl {a \mu_0 r_n })/\mu_0} }\right) \right|=0.$$ 
By (\ref{eq:scaling}) and the definition of $q_s(x)$: $$\int_{1-a} ^\infty ds \,
q_s\left(\frac{j}{
 B_{n/\mu_0} } \right)=\int_1 ^\infty ds
\, q_s\left(\frac{j}{
 (1-a)^{1/\theta}B_{ n/\mu_0} } \right).$$
To simplify notation, set
$a_1(n,j)=\frac{j}{(1-a)^{1/\theta} B_{n/\mu_0}}$ and
$a_2(n,j)=\frac{j}{B_{(n- \fl {a \mu_0 r_n})/\mu_0}}$.  It is thus sufficient to
verify that for $n$ sufficiently large:
\begin{equation}
\label{eq:amontrer}
\sup_{1 \leq j \leq \a B_n} \left| \int_1 ^\infty ds \,
(q_s(a_1(n,j))-q_s(a_2(n,j))) \right| \quad \mathop{ \longrightarrow} _ {n \rightarrow \infty} \quad 0.\end{equation}
From  (\ref{eq:scaling}), we have for $ x \geq 0$:
$$\int_1 ^\infty ds \, q_s(x)=  x \int_1 ^ \infty \frac{ds}{s} p_s(-x)= x \int_1^ \infty
\frac{ds}{s^ {1+1/ \theta}}  p_1(-x s ^ {-1/ \theta})=\theta \int_0 ^ {x} p_1(-u) \,du,$$
so that
$$\left| \int_1 ^\infty ds \,
(q_s(a_1(n,j))-q_s(a_2(n,j))) \right| = \theta
\left|\int^{a_2(n,j)} _ {a_1(n,j)} p_1(-u) \,du
\right| \leq  \theta M'|a_2(n,j) - a_1(n,j)|,$$ where we have used the fact that $p_1$ is bounded by a
positive real number $M'$. Thus we see that \eqref{eq:amontrer} will follow if we can verify that:
$$ \sup _ {1 \leq  j \leq \alpha B_n} |a_1(n,j)-a_2(n,j) | \quad\mathop { \longrightarrow}_ {n \rightarrow \infty} \quad 0,$$
and to this end it is enough to establish that:
$$  \left| \frac{B_n}{(1-a) ^ {1/ \theta} B_ {n/ \mu_0}} - \frac{B_n}{B_{(n- \fl{a \mu_0 r_n})/\mu_0}} \right|\quad\mathop { \longrightarrow}_ {n \rightarrow \infty} \quad 0.$$
The last convergence is however immediate from our assumption on the sequence $(r_n)$.\end{proof}

\subsection{Convergence of the scaled contour and height functions}

We now aim at proving invariance theorems under the conditional probability measure $\Pmu[\, \cdot \, | \, \lt  =
n]$.

Recall the notation $T_k( \tau)$ introduced in the beginning of this section. For $u \geq 0$, set $T_u( \tau)=T_ { \fl {u}}( \tau)$.

\begin{lem}\label{lem:CVf}Fix $a \in (0,1)$ and $\alpha < \min( a/2, (1-a)/2)$.
\begin{enumerate}
\item[(i)] We have
$ \d\lim_{n \rightarrow \infty} \Prmu{ \left.
  \sup_{ b \in (a-\a,a+\a)} \left|\frac{ T_{\mu_0 b n} (\tau)}{n}-b\right|> \frac{1}{n^{1/4}} \,  \right| \, \zt \geq n} = 0.$
\item[(ii)]  We have
$ \d \lim_{n \rightarrow \infty} \Prmu{ \left.
  \sup_{ b \in (a-\a,a+\a)} \left|\frac{ T_{b n} (\tau)}{n}-\frac{b}{\mu_0}\right|> \frac{1}{n^{1/4}}  \,  \right| \, \lt=n } = 0.$
\end{enumerate}
\end{lem}
\begin{proof}
Both assertions are easy consequences of Corollaries \ref{cor:GD2mieux} and \ref{cor:Lambda}. Details are left to the reader.
\end{proof}

\begin{lem}\label{lem:CVproba}Let $d$ be a positive integer. Fix $a \in (0,1)$ and consider a sequence $(Z^n)_ {n \geq 1}$ of càdlàg processes with values in $ \R^d$.
Let also $(K_n)_{n \geq 1}$ and $(S_n)_{n \geq 1}$ be two 
sequences of positive random variables converging in probability towards
$1$. Assume that $(Z^n)_{n \geq 1}$ converges in distribution in
$\D([0,\infty),\R^d)$ towards a  càdlàg process $Z$ such that
a.s. $Z$ is continuous at $a$. Then $(K_n Z^n_{S_n t};\, 0 \leq t \leq
a)$ converges in distribution in $\D([0,a],\R)$ towards $(X_t;\, 0
\leq t \leq a)$.
\end{lem}

\begin{proof} By the Skorokhod Representation Theorem
(see e.g. \cite[Theorem 6.7]{Bill}), we can assume that $(X^n)_{n
\geq 1}$ converges almost surely in $\D([0,\infty),\R^d)$ towards
$(X_t; t \geq 0)$ and that both $(K_n)_{n \geq 1}$ and $(S_n)_{n
\geq 1}$ converge almost surely towards $1$. The conclusion follows by standard properties of the
Skorokhod topology (see e.g. \cite[VI. Theorem 1.14]{Shir}).
\end{proof}

\begin{lem}\label{lem:absolue}For $n \geq 1$, let ${r_n}$  be the greatest positive integer such that $ \lceil \mu_0 {r_n}-r_n^{3/4} \rceil=n$. Fix $a \in (0,1)$. Let $\t_n$ be a
random tree distributed according to $\Pmu[\, \cdot \, | \, \lt  =
n]$. Then the law of
$$\left(\frac{1}{B_{{r_n}}} \W_{ \fl{T_ {a \mu_0 {{r_n}}}(\t_n) \frac{t}{a}}}(\t_n) ,
\frac{B_{{r_n}}}{{r_n}} H_{ T_ {a \mu_0 {{r_n}}}(\t_n)
\frac{t}{a}}(\t_n)\right)_{0 \leq t \leq a}$$ converges to the law of $( X_t,H_t)_{0 \leq t \leq a }$ under $\Nn( \, \cdot
\, | \, \z=1)$.
\end{lem}

\begin{proof}We start by proving that for every $\a >1$:
\begin{equation}
\label{eq:cva}\lim_{n \rightarrow \infty } \left(  \sup_{  \frac{1}{\a} B_n \leq j \leq \a B_n} \left| \frac{\psi_{n- \fl {a
\mu_0 {r_n} }}(j)/\psi_n(1)}{\psi^*_{n- \fl{a \mu_0 {r_n}}}(j)/\psi^*_n(1)}-\Gamma_a\left(\frac{j-1}{B_
{n/\mu_0}}\right)\right|\right)=0.\end{equation}
By Theorem \ref{thm:ltn}, $\psi^*_n(1)/n\psi_n(1) \rightarrow \theta$ as
$n \rightarrow \infty$. Using Corollary \ref{cor:tejhnique}, it then
suffices to verify that there exists $ \delta>0$ such that for $n$ sufficiently large:$$ \inf_ {\frac{1}{\a} B_n \leq j \leq \a B_n}  \int_ {1-a} ^ { \infty}ds q_s \left( \frac{j-1}{B_n} \right)> \delta.$$
This follows from the fact that there exists $ \delta'>0$ such that $ \int_ {1-a} ^ { \infty}ds q_s \left(x \right)> \delta'$ for every $x \in [{1}/{\a}, \a]$
Details are left to the reader.

Fix a bounded continuous function
$F : \D([0,a],\R^2) \rightarrow \R_+$. To simplify notation, for every tree $ \tau$ with $ \lambda( \tau) \geq n$, set $W^ {(n)}( \tau)=(W^ {(n)}_t( \tau))_ {0 \leq t \leq  a}$ and $H^ {(n)} ( \tau)=(H^ {(n)}_t ( \tau))_ {0 \leq t \leq  a}$, where for $ 0 \leq  t \leq  a$:
$$W^ {(n)}_t ( \tau) =\frac{1}{B_{{r_n}}} \W_{ \fl{T_ {a \mu_0 {{r_n}}}(\tau)\frac{t}{a}}}(\tau),  \qquad H^ {(n)}_t( \tau)=\frac{B_{{r_n}}}{{r_n}} H_{ T_ {a \mu_0 {{r_n}}}(\tau)\frac{t}{a}}(\tau).$$
Then set $G ^ {(n)}( \tau)=F\left(W^ {(n)}( \tau),H^ {(n)} ( \tau)  \right)$.  Note that by \eqref{eq:Hmes}, $H^ {(n)}( \tau)$ is a measurable function of $W^ {(n)}( \tau) $.
Fix $\a >1$ and put:
$$A^\a_n(\tau)=\left\{
\frac{1}{\a} B_{n/\mu_0}< \W_{ T_{a \mu_0 {{r_n}}} ( \tau)} (\tau)<
\a B_{n/\mu_0}\right\}.$$
By combining Proposition \ref{prop:abscon} and the estimate \eqref{eq:cva}, we get:
\begin{equation}
\label{eq:cvutiliser}\lim_{n \rightarrow \infty} \left|
\Es{G ^ {(n)}\left(\t_n  \right)1_{A^\a_n(\t_n)}}  - \Esmu{ \left. G ^ {(n)}\left( \tau\right) 1_{A^\a_n(\tau)}  \Gamma_a
\left(\frac{  \W_{ T_ {a \mu_0 {r_n}}(\tau)}(\tau)}{B_{n/\mu_0}}
\right) \right| \, \lt \geq n} \right|=0.
\end{equation}

We now claim that the law of $ \left(W^ {(n)}( \tau),H^ {(n)} ( \tau)  \right)$ under $ \Pmulgeqn$ converges towards the law of $(X_t,H_t)_  {0 \leq t \leq a}$ under $\Nn(\cdot \|
\z>1)$. To establish this convergence, by Proposition
\ref{prop:passage} (ii), it is sufficient to show that the law of $$ \left(\frac{1}{B_n} \W_{ \fl{T_ {a \mu_0 {n}}(\tau) \frac{t}{a}}}(\tau),\frac{B_{n}}{n} H_{ T_{a \mu_0 {n}} ( \tau)\frac{t}{a}}(\tau) \right)_{ 0 \leq t \leq  a }$$ under $ \Pmuzgeqn$ converges towards the law of $(X_t,H_t)_  {0 \leq t \leq a}$ under $\Nn(\cdot \|
\z>1)$. Indeed, Proposition \ref {prop:passage} (ii) will then imply that the same convergence holds if we replace $ \Pr { \, \cdot \, | \zt \geq n}$ by $\Pr { \, \cdot \, | {\lt}\geq \lceil \mu_0 n -{n} ^{3/4} \rceil}$ and we just have to replace $n$ by $r_n$. By Lemma \ref{lem:CVf}, under $ \Pmuzgeqn$, $ T_{a \mu_0 {n}}(\tau)/(an)$ converges in probability towards $1$, and by Theorem \ref{thm:cvW},
the law of $\left(\frac{1}{B_n} \W_{ \fl{n t}} ( \tau), \frac{B_n}{n} H_ {nt} ( \tau) \right)_ { t \geq 0}$ converges to the law of $(X_t,H_t)_  { t \geq  0}$ under $\Nn( \, \cdot
\, | \, \z>1)$. Our claim now follows from Lemma \ref{lem:CVproba}.

From the definition of ${r_n}$, we have ${r_n}/n \rightarrow 1/ \mu_0$ as $n \rightarrow \infty$, which entails $ B_ {{r_n}} / B_ {n/ \mu_0} \rightarrow 1$. Thanks to \eqref{eq:cvutiliser} and the preceding claim, we get that:
\begin{eqnarray}
\lim_{n \rightarrow \infty} \Es{ G^ {(n)}( \t_n) 1_{A^\a_n(\t_n)}}&=& \Nn(F \left(
(X_t,H_t)_{0 \leq t \leq a}  \right) \Gamma_a(X_a) 1_{\{\frac{1}{\a} < X_a < \a \}}
\| \z
> 1) \notag\\
&=& \Nn(F \left(
(X_t,H_t)_{0 \leq t \leq a}  \right)1_{\{\frac{1}{\a} < X_a < \a \}} \| \z = 1), \label{eq:cvG}
\end{eqnarray}
where we have used Proposition \ref{prop:Gamma} in the last
equality. By taking $F\equiv 1$, we obtain: $$\lim_{\a \rightarrow
\infty }\lim_{n \rightarrow \infty} \Pr{A^\a_n(\t_n)}=1.$$ 
By choosing $ \alpha>0$ sufficiently large, we easily deduce from the convergence \eqref{eq:cvG} that:
$$\lim_{n \rightarrow \infty} \Es{G^ {(n)} ( \t_n)}=\Nn(F \left(
(X_t,H_t)_{0 \leq t \leq a}  \right)  \| \z = 1).$$
This completes the proof.\end{proof}

Recall that $C(\tau)$ stands for the contour function of the tree
$\tau$, introduced in Definition \ref{def:fonctions}.

\begin{thm}\label{thm:cvG}
For every $n\geq 1$ such that $\Pmu[ \lt = n]>0$, let $\t_n$ be a
random tree distributed according to $\Pmu[\, \cdot \, | \, \lt  =
n]$. Then the following convergences hold.
\begin{enumerate}
\item[(i)]Fix $a \in (0,1)$. We have:
\begin{equation}\label{eq:H0}\left(\frac{1}{B_{\z(\t_n)}} \W_{ \fl{\z(\t_n) t}}(\t_n);\, 0 \leq t \leq a \right)
\quad \mathop{\longrightarrow}^{(d)}_{n \rightarrow \infty} \quad
\left( X_t;\, 0 \leq t \leq a \right) \textrm{ under }\Nn( \, \cdot
\, | \, \z=1).\end{equation}
\item[(ii)] We have:
\begin{equation}\label{eq:H}\left(\frac{B_{\z(\t_n)}}{\z(\t_n)}C_{2 \z(\t_n) t}(\t_n), \frac{B_{\z(\t_n)}}{\z(\t_n)}H_{\z(\t_n)
t}(\t_n)\right)_{0 \leq t \leq 1 } \quad
\mathop{\longrightarrow}^{(d)}_{n \rightarrow \infty}\quad (
H_t,H_t)_{0 \leq t \leq 1} \textrm{ under }\Nn( \, \cdot \, | \,
\z=1).\end{equation}
\end{enumerate}
\end{thm}

\begin{rem}\label {rem:changement}It is possible to replace the scaling factors $1/B_{\z(\t_n)}$ and $B_ {\z(\t_n)}/\z(\t_n)$
by respectively $\mu_0^{1/\theta}/B_n$ and $ \mu_0 ^ {1-1/ \theta} B_n/n$ without
changing the statement of the theorem. This follows indeed from the
fact that $\z(\t_n)/n$ converges in distribution towards $1/\mu_0$
under $\Pmuln$.\end{rem}

The convergence of rescaled contour functions in (ii) implies that the tree $\t_n$, viewed as
a finite metric space for the graph distance and suitably rescaled,
converges to the $\theta$-stable tree in distribution for the
Gromov-Hausdorff distance on isometry classes of compact metric
spaces (see e.g. \cite[Section 2]{RandomTrees} for details).

The convergence \eqref{eq:H0} actually holds with $a=1$. This will be proved later in Section 6.

\begin{proof}
Recall that throughout this section we limit ourselves to the case where
$\Pmu[ \lt = n]>0$ for all $n$ sufficiently large.

We start with (i). As in Lemma \ref{lem:absolue},  let ${r_n}$  be the greatest positive integer such that $ \lceil \mu_0 {r_n}-{r_n}^{3/4} \rceil=n$ and write:
$$\frac{1}{B_{\z(\t_n)}} \W_{ \fl{\z(\t_n) t}}(\t_n)= K_n \cdot \frac{1}{B_{{r_n}}} \W_{ \fl {S_n \cdot {T_ {a \mu_0 {{r_n}}}(\t_n) \frac{t}{a}}}}(\t_n),$$
where $K_n=B_{{r_n}}/B_{\z(\t_n)}$ and $S_n=a \z(\t_n)/ T_ {a \mu_0 {{r_n}}}(\t_n)$. Recall that ${r_n}/n \rightarrow 1/ \mu_0$. By Corollary \ref{cor:Lambda},
${\z(\t_n)}/n$ converges in probability to $1/\mu_0$. On the one hand, this entails that $K_n$ converges in probability towards $1$, and on the other hand, together with Lemma \ref{lem:CVf} (ii), this entails that $S_n$ converges in probability towards $1$. The convergence \eqref{eq:H0} then follows from Lemmas \ref {lem:absolue} and \ref{lem:CVproba}.

For the second assertion, we start by observing that:
\begin{equation}\label{eq:cvH}\left(\frac{B_{\z(\t_n)}}{ \z(\t_n)
} H_{ \z(\t_n) t}(\t_n);\, 0 \leq t \leq a \right)
\quad\mathop{\longrightarrow}^{(d)}\quad \left( H_t;\, 0 \leq t \leq
a \right) \textrm{ under }\Nn( \, \cdot \, | \, \z=1).\end{equation}
This convergence follows from Lemmas \ref {lem:absolue} and \ref {lem:CVproba} by the same arguments we used to establish (i). To complete the proof we use known relations between the height process and the contour process (see e.g. \cite[Remark 3.2]{Duquesne}) to show that an analog of \eqref{eq:cvH} also holds for the contour process. For $0 \leq p <
{\z(\t_n)}$ set $b_p=2p-H_p(\t_n)$ so that $b_p$ represents the time needed by the
contour process to reach the $(p+1)$-th individual of ${\z(\t_n)}$. Also set $b_{\z(\t_n)}=2
({\z(\t_n)}-1)$. Note that $C_{b_p}=H_p$ for every $ p \in  \{0,1, \ldots, \zeta( \t_n) \}$. From this
observation and the definitions of the contour function and the height function of a tree, we easily get:
\begin{equation}\label{eq:ineg}\sup_{t \in [b_p,b_{p+1}]}
|C_t(\t_n)-H_p(\t_n)| \leq
|H_{p+1}(\t_n)-H_p(\t_n)|+1.\end{equation} for $0 \leq p <
\z(\t_n)$. Then define the random function $g_n : [0, 2 {\z(\t_n)}]
\rightarrow \N$ by setting $g_n(t)=k$ if $t \in [b_k,b_{k+1})$ and
$k<{\z(\t_n)}$, and $g_n(t)={\z(\t_n)}$ if $t \in [2({\z(\t_n)}-1),2
{\z(\t_n)}]$. If $ t < 2( \z( \t_n)-1)$, $g_n(t)$ is the largest rank of an individual that has been visited before time $t$ by the contour function, if the individuals are listed $0,1, \ldots, \z( \t_n)-1$
in lexicographical order. Finally, set
$\widetilde{g}_n(t)=g_n({\z(\t_n)} t)/{\z(\t_n)}$. Fix $ \alpha \in
(0,1)$. Then, by (\ref{eq:ineg}):
$$\sup_{t \leq \frac{b_{\lfloor \alpha \z(\t_n) \rfloor}}{{\z(\t_n)}}}\left| \frac{B_{\z(\t_n)}}{{\z(\t_n)}} C_{{\z(\t_n)} t}(\t_n)- \frac{B_{\z(\t_n)}}{{\z(\t_n)}} H_{{\z(\t_n)} \widetilde{g}_n(t)}\right|
\leq \frac{B_{\z(\t_n)}}{{\z(\t_n)}}+\frac{B_{\z(\t_n)}}{{\z(\t_n)}}
\sup_{k \leq \lfloor \alpha \z(\t_n)
\rfloor}|H_{k+1}(\t_n)-H_k(\t_n)|,$$ which converges in probability
to $0$ by (\ref{eq:cvH}) and the path continuity of $H$. On the other hand,
it follows from the definition of $g_n$ that
\begin{eqnarray*}
\sup_{t \leq
\frac{b_{\lfloor \alpha \z(\t_n) \rfloor}}{{\z(\t_n)}}}\left|
\widetilde{g}_n(t)-\frac{t}{2} \right| &\leq&  \frac{1}{ \z( \t_n)} \left(\sup_ {k \leq  \fl { \a \z( \t_n)}}
\left| k- \frac{b_k}{2}\right|+1 \right) \\
&\leq& \frac{1}{2
B_{\z(\t_n)}} \sup_{k \leq \a {\z(\t_n)}}
\frac{B_{\z(\t_n)}}{{\z(\t_n)}}
H_k(\t_n)+\frac{1}{{\z(\t_n)}}\quad\mathop{\longrightarrow}^{(\P)}\quad
0
\end{eqnarray*}
by (\ref{eq:cvH}). Finally, by the definition of $b_n$ and using
(\ref{eq:cvH}) we see that $\frac{b_{\lfloor \alpha \z(\t_n)
\rfloor}}{{\z(\t_n)}}$ converges in probability towards $2\a$.  By applying the preceding observations with $ \alpha$ replaced by $ \alpha' \in ( \alpha,1)$, we
conclude that:
\begin{equation}\label{eq:HC}\frac{B_{\z(\t_n)}}{{\z(\t_n)}} \sup_{0 \leq t \leq \a} |C_{2 {\z(\t_n)} t}(\t_n) - H_{{\z(\t_n)} t}(\t_n)|
\quad \mathop{\longrightarrow}^{(\P)} \quad 0.\end{equation}
Together with (\ref{eq:cvH}), this implies:
\begin{equation}
\label{eq:cvja}\left(\frac{B_{\z(\t_n)}}{{\z(\t_n)}} C_{ 2 \z(\t_n)
t}(\t_n);\, 0 \leq t \leq a \right)  \quad
\mathop{\longrightarrow}^{(d)} \quad  \left( H_t;\, 0 \leq t \leq a
\right) \textrm{ under }\Nn( \, \cdot \, | \, \z=1).
\end{equation}

We now use a
time-reversal argument in order to show that the convergence holds
on the whole segment $[0,1]$. To this end, we adapt \cite[Remark
3.2]{Duquesne} and \cite[Section 2.4]{DuquesneLG} to our context.
See also \cite{Kautre}, where we used the same argument to give
another proof of Duquesne's Theorem \ref{thm:cv1}.  Observe that $ (C_t(\t_n); \, 0 \leq t \leq 2(\z(\t_n)-1))$ and
$(C_{2(\z(\t_n)-1)-t}(\t_n); \,  0 \leq t \leq 2(\z(\t_n)-1))$ have
the same distribution. From this convergence and the convergence \eqref{eq:cvja}, it is an easy exercise to obtain that:\begin{equation}\label{eq:cvC}\left(\frac{B_{\z(\t_n)}}{{\z(\t_n)}} C_{ 2 \z(\t_n)
t}(\t_n);\, 0 \leq t \leq 1 \right)  \quad
\mathop{\longrightarrow}^{(d)}  \quad \left( H_t;\, 0 \leq t \leq 1
\right) \textrm{ under }\Nn( \, \cdot \, | \, \z=1).\end{equation}
See the last paragraph of the proof of Theorem 6.1 in \cite{LGIto}
for additional details in a similar argument.

Finally, we verify that  (\ref{eq:H}) can be derived from \eqref{eq:cvC}.
To this end, we show that the convergence (\ref{eq:HC}) also holds for
$\alpha=1$. First note that:  \begin{equation}
\label{eq:g}
\sup_{0 \leq t \leq 2 }\left|
\widetilde{g}_n(t)-\frac{t}{2} \right| \leq \frac{1}{\z(\t_n)} \left(\frac{1}{2}
\sup_{k \leq {\z(\t_n)}} H_k(\t_n)+1\right)= \frac{1}{2
B_{\z(\t_n)}} \sup_{k \leq 2 {\z(\t_n)}}
\frac{B_{\z(\t_n)}}{\z(\t_n)} C_k(\t_n)+\frac{1}{{\z(\t_n)}}
\mathop{\longrightarrow}^{(\P)} 0\end{equation} by (\ref{eq:cvC}).
Secondly, by \eqref{eq:ineg}:
\begin{eqnarray*} \sup_{0 \leq t \leq 2}\left|
\frac{B_{\z(\t_n)}}{\z(\t_n)} C_{{\z(\t_n)} t}(\t_n)-
\frac{B_{\z(\t_n)}}{\z(\t_n)} H_{{\z(\t_n)}
\widetilde{g}_n(t)} ( \t_n)\right|
&\leq&
\frac{B_{\z(\t_n)}}{\z(\t_n)} + \frac{B_{\z(\t_n)}}{\z(\t_n)}
\sup_{k < \z(\t_n)
}|H_{k+1}(\t_n)-H_k(\t_n)| \\
& =&
\frac{B_{\z(\t_n)}}{\z(\t_n)}+\frac{B_{\z(\t_n)}}{\z(\t_n)} \sup_{k
<  \z(\t_n) }
\left|C_{ b_ {k+1}}(\t_n)-C_{ b_k}(\t_n)\right|.\end{eqnarray*}
By (\ref{eq:cvC}), in order to show that the latter quantity tends
to $0$ in probability, it is sufficient to verify that $\sup_{k <
\z(\t_n) } \z(\t_n) ^ {-1}\left|
{b_{k+1}}-{b_{k}}\right|$ converges to
$0$ in probability. But by the definition of $b_p$:
\begin{eqnarray*}\sup_{k < \z(\t_n) } \left|
\frac{b_{k+1}}{\z(\t_n)}-\frac{b_{k}}{\z(\t_n)}\right|=\sup_{k <
\z(\t_n) } \left|  \frac{2+H_k(\t_n)-H_{k+1}(\t_n)}{\z(\t_n)}\right|
&\leq&
\frac{2}{\z(\t_n)}+ 2\sup_{k < \z(\t_n) } \frac{H_k(\t_n)}{\z(\t_n)}
\end{eqnarray*}
which converges in probability to $0$ by the same argument as in \eqref{eq:g}. We have thus obtained that $$\frac{B_{\z(\t_n)}}{\z(\t_n)} \sup_{0 \leq t \leq 1} |C_{2 {\z(\t_n)} t}(\t_n) - H_{{\z(\t_n)} \widetilde{g}_n(2t)}(\t_n)|
\quad \mathop{\longrightarrow}^{(\P)} \quad 0.$$ Combining this with
(\ref{eq:cvC}), we conclude that:
$$\left(\frac{B_{\z(\t_n)}}{\z(\t_n)}C_{2 \z( \t_n) t}(\t_n), \frac{B_{\z(\t_n)}}{\z(\t_n)}H_{\z(\t_n) \widetilde{g}_n(2t)
}(\t_n)\right)_{0 \leq t \leq 1 } \quad
\mathop{\longrightarrow}^{(d)}_{n \rightarrow \infty}\quad (
H_t,H_t)_{0 \leq t \leq 1} \textrm{ under }\Nn( \, \cdot \, | \,
\z=1).$$
The convergence \eqref{eq:g} then entails:
$$\left(\frac{B_{\z(\t_n)}}{\z(\t_n)}C_{2 \z(\t_n) t}(\t_n), \frac{B_{\z(\t_n)}}{\z(\t_n)}H_{\z(\t_n) t}(\t_n)\right)_{0 \leq t \leq 1 } \quad
\mathop{\longrightarrow}^{(d)}_{n \rightarrow \infty}\quad (
H_t,H_t)_{0 \leq t \leq 1} \textrm{ under }\Nn( \, \cdot \, | \,
\z=1).$$
This completes the proof.
\end{proof}

\subsection{Proof of the technical lemma}

\label {sec:technic}

In this section, we control the Radon-Nikodym densities appearing in
Proposotion \ref{prop:abscon}. We heavily
rely on the strong version of the Local Limit Theorem (Theorem
\ref{thm:strongLL}).

Throughout this section, $(W_n)_{n \geq 0}$ will stand for the
same random walk as in Proposition \ref{prop:RW}. Recall also the notation $q_s$ introduced in Proposition \ref {prop:Gamma}.


\subsubsection{Proof of Lemma \ref{lem:technical} (i)}
We will use two lemmas to prove Lemma \ref{lem:technical} (i): the first one gives an estimate for
$\Prmuj{\z(\bf) \geq n}$ and the second one shows that
$\Prmuj{\z(\bf) \geq n}$ is close to $\Prmuj{\l(\bf) \geq
\mu_0 n-n^{3/4}}$.

\begin{lem}\label{lem:technical2}We have
$ \d \lim_{n \rightarrow \infty }  \sup_{1 \leq j \leq \a B_n} \left|
\Prmuj{\z(\bf) \geq n}- \int_1 ^\infty ds \, q_s\left(\frac{j}{
 B_{n} }\right) \right|=0.$\end{lem}

\begin{proof}Le Gall established this result in the case where the variance of $\mu$ is finite in \cite{LGIto}.
See \cite[Lemma 3.2 (ii)]{Kautre} for the proof in the general case,
which is a generalization of Le Gall's proof.
\end{proof}

\begin{lem}\label{lem:g}Fix $\a>0$. We have
$ \d \lim_{n  \rightarrow \infty}  \sup_{1 \leq j \leq \a B_n}|\Pmuj[\z(\bf) \geq n]-\Pmuj[\l(\bf) \geq
\mu_0n-n^{3/4}]|=0$.
\end{lem}

\begin{proof}
To simplify notation, set $ \gamma=\mu_0^{-1}+1$. By Lemma \ref{lemma:passage}, it is sufficient to show that:
$$\lim_{n  \rightarrow \infty}  \sup_{1 \leq j \leq \a B_n}\Pmuj[n- \gamma n^{3/4} \leq \z(\bf) \leq
n]=0.$$ From the local limit theorem (Theorem \ref{thm:locallimit}), we have, for every $j \in \Z$: $$ \left|{B}_k \P[W_k=j]-
p_1\left(\frac{j}{{B}_k}\right) \right| \leq \epsilon(k),$$ where
$\e(k)\rightarrow 0$. The function $x \mapsto |x p_1(-x)|$ is bounded over $\R$
by a real number which we will denote by $M$ (see e.g. \cite[I. 4]{Zolotarev}). Set $M_n(j)=\Pmuj[n- \gamma n^{3/4} \leq \z(\bf) \leq n]$ and $ \delta(n)= \fl{n- \gamma n^{3/4}}+1$. Fix $\e>0$ and suppose that $n$ is
sufficiently large so that $ n- \gamma n^{3/4} \leq k \leq n$ implies $|\e(k)| \leq
\e$ and $ B_k \geq  B_n/2$. Then, for $1 \leq j
\leq \a B_n$, by \eqref{eq:kemperman}:
\begin{eqnarray*}
M_n(j)&=& \sum_{k= \delta(n)}^n
\Pmuj[\z(\bf)=k] =\sum_{k=\delta(n)}^n \frac{j}{k} \Pmu[W_k=-j]
 \\
 &\leq&  \sum_{k= \delta(n)}^n  \frac{ j }{ k {B}_k}
 \left(p_1\left(-\frac{j}{{B}_k}\right)+ \e(k) \right)\\
 &\leq& \sum_{k= \delta(n) }^n  \frac{M+2 \a \e}{k},
\end{eqnarray*}
which tends to $0$ as $n \rightarrow \infty$.
\end{proof}

\begin{proof}[Proof of Lemma \ref{lem:technical} (i)] By
Lemmas \ref{lem:technical2}  and \ref{lem:g}:
\begin{equation}
\label{eq:e}
\lim_{n \rightarrow \infty } \left( \sup_{1 \leq j \leq \a B_n} \left|
\Pmuj[\l( \bf) \geq \mu_0n-n^{3/4}]- \int_1 ^\infty ds \,
q_s\left(\frac{j}{
 B_n }\right) \right|\right)=0.
 \end{equation}
Let ${r_n}$  be the greatest positive integer such that $ \lceil \mu_0 {r_n}-{r_n}^{3/4} \rceil=n$. We apply \eqref{eq:e} with $n$ replaced by ${r_n}$, and we see that the desired result will follow if we can prove that
$$\lim_{n \rightarrow \infty } \left(
\sup_{1 \leq j \leq \a B_n} \left| \int_1 ^\infty ds \,
q_s\left(\frac{j}{
 B_{ {{r_n}}} }\right)- \int_1 ^\infty
ds \, q_s\left(\frac{j}{ B_{n/\mu_0} }\right) \right|\right)=0.$$
The proof of the latter convergence is similar to that of
(\ref{eq:amontrer}) noting that:
$$\lim_{n \rightarrow \infty}\left| \frac{B_n}{
 B_{{r_n}}}-\frac{B_n}{ B_{n/\mu_0}}\right|=0.$$
This completes the proof of Lemma \ref{lem:technical} (i).
\end{proof}

\subsubsection{Proof of Lemma \ref{lem:technical} (ii)}

\label {sec:technik}

The proof of Lemma \ref{lem:technical} (ii) is very
technical, so we will sometimes only sketch arguments.

As previously, denote by $S_n$ the sum of $n$ independent Bernoulli
random variables of parameter $\mu_0$, and by $W'$ the random walk
$W$ conditioned on having nonnegative jumps. More precisely,
$\P[W'_1=i] = \mu(i+1) / (1-\mu_0)$ for $i \geq 0$.
Recall that $\E[W'_1]={\mu_0}/({1-\mu_0})$ and that $ \sp$ is the variance of $W'_1$.

Fix
$0<\e<1$. By Lemma \ref{lem:GdesDev2} (ii):

\begin{equation}
\label{eq:psi}
n \psi_n(j) =n \Pmuj \left[ \lt = n ,\,   \frac{n}{\mu_0}-\e n \leq \z(\bf)
\leq \frac{n}{\mu_0} +\e n \right] + oe_{1/2}(n),
\end{equation}
where the estimate $oe_{1/2}(n)$ is uniform in $j$. It is thus
sufficient to control the first term in the last expression.  For $ | u| \leq  \epsilon \sqrt {n}$ and $1 \leq j \leq n$ set: 
$$ r_n(u)= \fl{ {n}/{\mu_0}+u \sqrt{n}}, \qquad a_n(u)=\sqrt{n} \P[S_{r_n(u)}=n], \qquad
b_n(u,j)= B'_n \P[W'_{r_n(u)-n}=n-j],$$
and using Proposition \ref{prop:K} write:
\begin{eqnarray}
 && n \Pmuj \left[ \lt = n ,\, \frac{n}{\mu_0}-\e n \leq \z(\bf)
\leq \frac{n}{\mu_0} +\e n \right] \notag \\
&& \d \qquad \qquad\qquad \qquad \qquad = n \sum_{p= \ce{{n}/{\mu_0}-\e n}}^{\fl{{n}/{\mu_0}+\e n}}
\frac{j}{p} \,
\P[S_p=n] \, \P[W'_{p-n}=n-j]\notag\\
&& \d \qquad \qquad \qquad \qquad \qquad =  n \d \int_ {n/ \mu_0 - \epsilon n + O(1)} ^{n/ \mu_0 + \epsilon n + O(1)} dx \frac{j}{ \fl {x}} \Pr {S_ { \fl {x}} =n} \Pr {W'_ { \fl {x}-n}=n-j}
\notag \\
&& \qquad \qquad \qquad \qquad \qquad =  \frac{j}{ B'_n } \int_{- \e \sqrt{n}+o(1)} ^{\e \sqrt{n}+o(1)} du \,
\frac{n}{r_n(u)} a_n(u) b_n(u,j).\label{eq:retenir}
\end{eqnarray}

Let us introduce the following notation. Set  $c= {\mu_0}/({1-\mu_0})$ and for $u,x \in \R$:
$$F(u)=\frac{1}{\sqrt{2 \pi \mu_0 (1-\mu_0)}} e^{- \frac{ 1}{2 \mu_0
(1-\mu_0)} u^2}, \qquad 
G_0(u,x)=  c^ {1/ \theta} p_1 \left(-c^{1/\theta}x-\frac{ \sqrt {2} c^{3/2}u}{ \sigma'} 1_{\{ \sigma^2< \infty\}} \right).$$
Put ${F}_0(u)= \sqrt { \mu_0} F (\mu_0 ^ {3/2}u)$. Fix $ \alpha>0$. Set finally $\a'=\a(1+(1-\mu_0)^{1/\theta})$. By Lemma \ref{lem:hhp}, for $n$ sufficiently large, we have $ \a B_n \leq  \a' B'_n$. We suppose in the following that $n$ is sufficiently large so that the latter inequality holds.

\begin {lem} \label {lem:techniqueutile}For fixed $u \in \R$, we have:
$$a_n(u)  \quad 
 \mathop{\longrightarrow}_{n \rightarrow \infty} \quad  {F}_0(u), \qquad  \sup_ {1 \leq j \leq \alpha B_n} \left| b_n(u,j) - G_0(u, {j}/{B'_n}) \right|\quad 
 \mathop{\longrightarrow}_{n \rightarrow \infty} \quad  0.$$
\end {lem}

\begin {proof}The first convergence is an immediate consequence of Theorem \ref {thm:locallimit} (iii) after noting that $ \Es {S_1}= \mu_0$ and that the variance of $S_1$ is $ \mu_0 (1- \mu_0)$. The second convergence is more technical. To simplify notation, set: $$ q_n(u)=r_n(u)-n, \qquad  {Q}_n(u,j)= c^ {1/ \theta} p_1 \left(\frac{n-j- c q_n(u)}{B'_{q_n(u)}} \right).$$
Note that $q_n(u)=n/c+ u \sqrt {n} + \mathcal{O}(1)$. In particular,  $B'_n \sim c^ {1/ \theta} B'_ {q_n(u)} $ as $ n \rightarrow \infty$. Consequently, by \eqref{eq:ll}, 
$\left| b_n(u,j) - Q_n(u,j) \right| \rightarrow 0$ as $ n \to \infty$, uniformly in $0 \leq j \leq \alpha B_n$.
It thus remains to show that
\begin{equation}
\label{eq:cvmoche} \sup_ {1 \leq j \leq \alpha B_n} \left| Q_n(u,j) - G_0(u,j/B'_n)\right|\longrightarrow 0.
\end{equation}
To this end, introduce:$$K_n(u,j)=\left | \frac{n-j- c q_n(u)}{B'_{q_n(u)}} + c^{1/\theta} \frac{j}{
B'_n }+  \frac{ \sqrt {2} c^{3/2}u}{ \sigma'}  1_{\{ \sigma^2< \infty\}}\right|.$$Recall that the absolute value of the derivative of
$p_1$ is bounded by a constant which will be denoted by $M'$, giving $\left| Q_n(u,j) - G_0(u,j/B'_n)\right| \leq  M' K_n(u,j)$. It is thus sufficient to show that  $K_n(u,j)\rightarrow 0$ as $ n \to \infty$, uniformly in $0 \leq j \leq \alpha B_n$.

We first treat the case where $\sigma^2< \infty$,
so that $\theta=2$. In this case, $B'_n=\sigma' \sqrt{n/2}$,
where $\sp$ is the variance of $W'_1$. Simple calculations show that $K_n(u,j) \leq  A /  \sqrt {n}$ for some $A \geq 1$ depending only on $u$, so that $K_n(u,j)\rightarrow 0$ as $ n \to \infty$, uniformly in $0 \leq j \leq \alpha B_n$.

Let us now suppose that $\sigma^2= \infty$. First assume that $ \theta <2$. Choose $\eta>0$ such
that $\e':=1/\theta-\eta-1/2>0$. By Proposition \ref{prop:slow} (i), for $n$ sufficiently large, $B'_{q_n(u) } \geq  {n^{1/\theta-\eta}}$.
Moreover, we can write ${B'_n}/({c^{1/\theta} B'_{q_n(u)}})=1+\e_n(u)$
where, for fixed $u$, $\e_n(u) \rightarrow 0$ as $n \rightarrow \infty$. Putting these estimates together, we
obtain that for large $n$ and for $1 \leq j \leq  \alpha B_n$:
$$K_n(u,j)= \left| c^ {1/ \theta}\frac{j}{B'_n} \e_n(u)+ \frac{c u \sqrt {n}+ \mathcal{O}(1)}{B'_ {q_n(u)}} \right|\leq \a' c^{1/\theta} \e_n(u)+ \frac{cu}{n^{\e'}}+\mathcal{O} \left( \frac{1}{n^ {1/4}} \right),$$
which tends to $0$ as $ n \to \infty$.

We finally treat the case when $ \sigma^2= \infty$ and $ \theta=2$. Recall the definition of the slowly varying function $h'$ introduced in Section \ref {sec:nonlatice} and let $ \e_n(u)$ be as previously. By the remark following the proof of Theorem \ref{thm:locallimit}, $h'$ is increasing so that for $n$ large enough:
$$K_n(u,j)= \left| c^ {1/ \theta}\frac{j}{B'_n} \e_n(u)+ \frac{c u \sqrt {n}+ \mathcal{O}(1)}{h'( {q_n(u)}) \sqrt { q_n(u)}} \right|\leq \a' c^{1/\theta} \e_n(u)+ A\frac{u}{h'(n/(2c))}+\mathcal{O} \left( \frac{1}{n^ {1/4}} \right)$$
for some $A>0$. The latter quantity tends to $0$ as $n \rightarrow \infty$ since $h'(n) \rightarrow \infty$ as $n \rightarrow \infty$ by the remark following the proof of Theorem \ref{thm:locallimit}. This completes the proof.
\end {proof}

\begin{proof}[Proof of Lemma \ref{lem:technical} (ii)] 
From Theorem \ref {thm:strongLL} we have the bound $a_n(u) \leq  C (1 \wedge u ^ {-2})$ and by \eqref{eq:ll}, the functions $b_n$ are uniformly bounded. Since, for $ j \leq  \alpha B_n$,
$$ \frac{j}{B'_n} \left| \int_{- \e \sqrt{n}} ^{\e \sqrt{n}} du \,
\frac{n}{r_n(u)} \, a_n(u) b_n(u,j) - 
\mu_0  \int_{- \e \sqrt{n}} ^{\e \sqrt{n}}
du \,  a_n(u) b_n(u,j) \right| \leq C \alpha'  \int_ {- \e \sqrt {n}} ^ { \e \sqrt {n}} du (1 \wedge u^ {-2}) \left| \frac{n}{r_n(u)}- \mu_0 \right|,$$
it follows from the dominated convergence theorem that:
\begin{equation}
\label{eq:cvdd}\sup_ {1 \leq j \leq  \alpha B_n}\left| \frac{j}{B'_n} \int_{- \e \sqrt{n}} ^{\e \sqrt{n}} du \,
\frac{n}{r_n(u)} \, a_n(u) b_n(u,j) - 
\frac{j}{B'_n} \mu_0  \int_{- \e \sqrt{n}} ^{\e \sqrt{n}}
du \,  a_n(u) b_n(u,j) \right| \quad 
 \mathop{\longrightarrow}_{n \rightarrow \infty} \quad 0.
\end{equation}
Recall that $ q_1(x)=xp_1(-x)$. By \eqref{eq:psi}, \eqref{eq:retenir} and \eqref{eq:cvdd}, to prove Lemma \ref{lem:technical} (ii), it is sufficient to establish that:
\begin{equation}
 \label{eq:lemtechnique0}\sup_ {1 \leq j \leq  \alpha B_n}\left|  \mu_0  \int_{- \e \sqrt{n}} ^{\e \sqrt{n}}
du \,  a_n(u) b_n(u,j) - c^ {1/ \theta}p_1\left(-\frac{j}{B_{n/\mu_0}}\right) \right| \quad 
 \mathop{\longrightarrow}_{n \rightarrow \infty} \quad 0 
 \end{equation}
 
Let us first show that:
\begin{equation}
\label{eq:lemtechnique1}
\sup_ {1 \leq j \leq  \alpha B_n}\left|  \mu_0 \int_{- \e \sqrt{n}} ^{\e \sqrt{n}} du \,
a_n(u) b_n(u,j) -
 \mu_0 \int_{- \infty} ^{+ \infty}
du \,  F_0( u) G_0(u, {j}/{B'_n}) \right| \quad 
 \mathop{\longrightarrow}_{n \rightarrow \infty} \quad 0.\end{equation} 
To this end, let us prove the following stronger convergence: 
\begin{equation}
\label{eq:lemtechnique2} \int_{- \e \sqrt{n}} ^{\e \sqrt{n}} du  \, \left(\sup _ {1 \leq j \leq  \alpha B_n} 
\left| a_n(u) b_n(u,j) - F_0( u) G_0(u, {j}/{B'_n}) \right| \right) \quad 
 \mathop{\longrightarrow}_{n \rightarrow \infty} \quad 0.
\end{equation}
It is clear that the function $G_0$ is uniformly bounded. Recall that the functions $b_n$ are uniformly bounded as well. Moreover, $F_0$ is an integrable function and we have the bound $a_n(u) \leq  C (1 \wedge u ^ {-2}) $. The convergence \eqref{eq:lemtechnique2} then follows from Lemma \ref {lem:techniqueutile} and the dominated convergence theorem. This proves \eqref{eq:lemtechnique1}.

To
conclude, we distinguish two cases.
First assume that $ \sigma^2< \infty$, so that $\theta=2$.  Then $W'_1$ has finite variance $ \sp$ as well. Recall that $\sp={(\sigma^2-\mu_0)}/({1-\mu_0})-({\mu_0}/({1-\mu_0}))^2$ and $B'_n=\sigma' \sqrt{n/2}$. A straightforward calculation based on the fact that, for $\a,\b,\gamma, \delta >0$,
$$\int_{- \infty} ^ { + \infty} du \, e^{- \a u^2} e^{- \b (\gamma+\delta u)^2}
= \frac{\sqrt{\pi}}{\sqrt{\a+\b \delta^2}} e^{- \frac{\a
\gamma^2}{\a/\b+\delta^2}}$$
gives:
\begin{eqnarray*}
\mu_0 \int_{- \infty} ^{+ \infty}
du \,  F_0( u) G_0(u, {j}/{B'_n})&=& \int_{- \infty} ^{+ \infty} du \,
\mu_0 F_0(u) c^ {1/2}p_{1}\left(c^{1/2}
\frac{j}{\sigma' \sqrt{n/2}}+\frac{ \sqrt {2} c^{3/2}u}{ \sigma'} \right) \\
&=& c^ {1/2}p_1\left(-\frac{j}{\sigma\sqrt{n/(2\mu_0)}}\right) = c^ {1/2}p_1\left(-\frac{j}{B_{n/\mu_0}}\right).
\end{eqnarray*}
By combining this with \eqref{eq:lemtechnique1}, we get \eqref{eq:lemtechnique0} as desired.

Now assume that $ \sigma^2 = \infty$. In this case:
\begin{equation}
\label{eq:lemtechnique3}
\mu_0 \int_{- \infty} ^{+ \infty}
du \,  F_0( u) G_0(u, {j}/{B'_n})=\int_{- \infty} ^{+ \infty} du \, \mu_0 F_0( u)
 c^ {1/ \theta}p_1\left(-c^{1/\theta} \frac{j}{B'_n}
 \right)= c^{1/\theta}
 p_1\left(-c^{1/\theta} \frac{j}{B'_n}
 \right).
 \end{equation}
 By Lemma \ref{lem:hhp}, $
 {B'_n}/{B_{n/\mu_0}} \rightarrow c^{1/\theta}$, which implies:
\begin{equation}
\label{eq:lemtechnique4}
 \sup_{1 \leq j \leq \a B_n}
 \left| 
 p_1\left(-c^{1/ \theta} \frac{j}{B'_n}
 \right) - 
 p_1\left(-\frac{j}{B_{n/\mu_0}}\right)\right|  \quad 
 \mathop{\longrightarrow}_{n \rightarrow \infty} \quad 0.\end{equation}
By combining \eqref{eq:lemtechnique3} and \eqref{eq:lemtechnique4} with \eqref{eq:lemtechnique1}, we get \eqref{eq:lemtechnique0} as desired.\end{proof}

\section{Convergence of rescaled Lukasiewicz paths when conditioning on having exactly $n$ leaves}

We have previously established that the rescaled Lukasiewicz path,
height function and contour process of a tree distributed according to $\P_\mu[\, \cdot \, | \,
\lt=n]$ converge in distribution on $[0,a]$ for every $a \in
(0,1)$. Recall that by means of a time-reversal argument, we were
able to extend the convergence of the scaled height and contour functions to the whole segment $[0,1]$.
However, since the Lukasiewicz path $ \W(\tau)$ of a tree distributed according to $\P_\mu[\, \cdot
\, | \, \lt=n]$ is not invariant under time-reversal, another
approach is needed to extend the
convergence of $ \W(\tau)$ (properly rescaled) to the whole segment $[0,1]$. To this end, we will use a Vervaat
transformation. Let us stress that the Lukasiewicz path
of a tree distributed according to $\Pmu[\, \cdot \, | \, \lt=n]$ does not have a deterministic
length, so that special care is necessary to prove the following
theorem.

Recall that $ \mu$ is a probability distribution on $ \N$ satisfying the hypothesis $ (H _ \theta)$ for some $ \theta \in (1,2]$. Recall also the definition of the sequence $(B_n)$, introduced just before Lemma
\ref{lem:locallimit}.

\begin{thm}\label{thm:cvW3}
For every $n\geq 1$ such that $\Pmu[ \lt = n]>0$, let $\t_n$ be a
random tree distributed according to $\Pmu[\, \cdot \, | \, \lt  =
n]$. Then:
\begin{equation}\label{eq:thmcvW3}\left(\frac{1}{B_{\z(\t_n)} } \W_{ \lfloor \z(\t_n) t \rfloor
}(\t_n);\, 0 \leq t \leq 1 \right) \quad
 \mathop{\longrightarrow}^{(d)}_{n \rightarrow \infty} \quad (X_t; 0 \leq t \leq 1) \textrm{ under } \Nn( \, \cdot \, |
\, \z=1).\end{equation}
\end{thm}

As previously, to avoid further technicalities, we prove Theorem \ref {thm:cvW3} in the
case where $\Pmu[ \lt = n]>0$ for all $n$ sufficiently large. Throughout this section, $(W_n)_{n \geq 0}$ will stand for the random walk of Proposition \ref{prop:RW}. 
Introduce the following notation for $n \geq  0$ and $u \geq 0$:
\begin{equation}\label{eq:notations}\Lambda(n)=\sum_{j=0}^{n-1}1_{\{W_{j+1}-W_j=-1\}}, \,\,  T_ {u}=
\inf\{ k \geq 0; \, \Lambda(k)=\fl{u}\}, \,\, \z= \inf\{ k \geq 0;
\, W_k=-1\}.\end{equation}
For technical reasons, we put $B_u = B_ {\fl{u}}$ for $u \geq 1$.

\begin{lem}\label{lem:CVf2}The following properties hold.
\begin{enumerate}
\item[(i)]We have $ \d\Pr{
   \left|\frac{T_n}{n}-\frac{1}{\mu_0}\right|> \frac{1}{n^{1/4}} } =
  oe_{1/2}(n)$.
\item[(ii)]For every $a>0$,  $ \d \left(\frac{1}{B_{ {n}/{\mu_0} }} W_{ \fl{ \frac{t}{a} T_ {an} }}; \, 0 \leq t
\leq a\right) \qquad \mathop{\longrightarrow}^{(d)}_{n\rightarrow
\infty} \qquad (X_t; \, 0 \leq t \leq a)$ under $ \P$.
\end{enumerate}
\end{lem}

\begin{proof} The first assertion is an easy consequence of Lemma \ref{lem:GdesDev} (i). For (ii), we use a
generalization of Donsker's invariance theorem to the stable case, which states that 
$\left(W_{\fl{nt}}/B_n; \, t \geq 0\right)$ converges in distribution towards $  (X_t; \, t \geq 0)$ as $n \rightarrow \infty$. See e.g. \cite[Chapter VII]{Shir}. By (i), $T_n/n$
converges almost surely towards $1/\mu_0$, and (ii) easily follows.
\end{proof}

\subsection{The Vervaat transformation}

We introduce the Vervaat transformation, which will allow us to deal
with random paths with no positivity constraint. Recall the notation $ \bx^ {(i)}$ introduced in Section 1.3.

\begin{defn}\label{def:vervaat}Let $k \geq 1$ be an integer and let $\bx=(x_1,\ldots,x_{k}) \in
\Z^k$. Set $w_j=x_1+\cdots + x_j$ for $1 \leq j \leq k$ and
let the integer $i_*(\bx)$ be defined by $i_*(\bx) = \inf \{ j \geq 1; w_j= \min_{1 \leq i \leq k} w_i \}$.
The Vervaat transform of $\bx$ is defined as $\V(\bx)=\bx^{(i_*(\bx))}$.
\end{defn}

Also introduce the following notation for positive integers $k\geq
n$:
$$\S^{k}(n)=\{ (x_1,\ldots,x_{k}) \in \{-1,0,1,\ldots\} ; \,  \sum_{i=1}^{k} x_i=-1 \textrm{ and } \Card \{ 1 \leq i \leq k; \, x_i=-1 \}=n\},$$
as well as:
$$\Sb^{k}(n)=\{ (x_1,\ldots,x_{k}) \in \S^{k}(n) ; \, \sum_{i=1}^m x_i
>-1
\textrm{ for every } m \in \{1,2,\ldots,k-1\}\}.$$
Finally set $\Sb(n)=\cup_{k \geq n}\Sb^{k}(n)$.

\begin{lem}\label{lem:cool}Let $k \geq n$ be positive integers. Set $Z^k=(W_1,W_2-W_1,\ldots,W_{k}-W_{k-1})$.
\begin{enumerate}
\item[(i)] Conditionally on the event $\{W_k=-1\}$, the random variable $i_*(Z^k)$ is uniformly distributed
on $\{1,2,\ldots,k\}$ and is independent of $\V(Z^k)$.
\item[(ii)] Let $\bx \in \Sb^{k}(n)$. Then:
\begin{equation}\label{eq:v1}\Pr{\V(Z^k)=\bx, 
Z^k_k=-1}=\frac{n}{k}\Pr{\V(Z^k)=\bx}.\end{equation}
\end{enumerate}
\end{lem}

\begin{proof}The first assertion is a well-known fact, but we give a proof for the sake
of completeness. Let $\bx \in \Sb^{k}(n)$ with $ k \geq n$. Then: \begin{equation}
\label{eq:V}
\Pr{Z^k=\bx}=\frac{1}{k} \sum_{i=1}^k \Pr{
\left(Z^k\right)^{(i)}=\bx}= \frac{1}{k} \Pr{\V(Z^k)=\bx}.
\end{equation}
For the first equality, we have used the fact that $Z^k$
and $\left(Z^k\right)^{(i)}$ have the same law. The
second equality follows from the fact that by the Cyclic Lemma, there exists a unique $1 \leq i_* \leq k$ such
that $\left(Z^k\right)^{(i_*)}\in \cup_{n \geq 1}
\overline{\S}^{k}(n)$, which entails $\V(Z^k)=\left(Z^k\right)^{(i_*)}$.
Then, for $1 \leq i \leq k$: \begin{eqnarray*}\Pr{i_*(Z^k)=i,
\V(Z^k)=\bx}=\Pr{(Z^k) ^ {(i)}=\bx}
= \Pr{ Z^k=\bx}=\frac{1}{k} \Pr{\V(Z^k)=\bx}.\end{eqnarray*} The
conclusion follows.

For (ii), write $\bx=(x_1,\ldots,x_{k})$ and observe that
$$ \Pr{\V(Z^k)=\bx, Z^k_k=-1}=
\Pr{\V(Z^k)=\bx,
x_{k-i_*(Z^k)}=-1}.$$ The conclusion follows from (i)
since $\Card\{ 1 \leq i \leq k ; \, x_i=-1\}=n$.\end{proof}

\begin{prop}\label{prop:vervaatdiscrete} For every integer $n \geq 1$, the law of the vector $\V(W_1,W_2-W_1,\ldots,W_{ T_n}-W_{T_n-1})$ under $\Pr{ \, \cdot \, | \,
W_{T_n}=-1}$ coincides with the law of the vector $( \W_1(\tau),
\W_2(\tau)- \W_1(\tau), \ldots, \\ \W_{\zt}(\tau)- \W_{\zt-1}(\tau))$ under
$\Prmu{\, \cdot \, | \, \lt=n}$.
\end{prop}

\begin{proof}To simplify notation set
$Z=(W_1,W_2-W_1,\ldots,W_{T_n}-W_{T_n-1})$. Fix an integer $k \geq
n$, and set $Z^k=(W_1,W_2-W_1,\ldots,W_{k}-W_{k-1})$. Let $\bx=(x_1,\ldots,x_{k}) \in
\Sb^{k}(n)$. We have:$$\Pr{ \V(Z) = \bx
\, | \, W_{T_n}=-1}=\Pr{ \V(Z^k) = \bx , \, T_n=k \, | \,
W_{T_n}=-1}$$
simply because $Z=Z^k$ on the event $  \{T_n=k\}$. Then write:
\begin{eqnarray*}
\Pr{ \V(Z) =  \bx \,  | \, W_{T_n}=-1} &=&
\frac{\Pr{
\V(Z^k) =  \bx, T_n=k}}{\Pr{W_{T_n}=-1}} =\frac{\Pr{ \V(Z^k) =
\bx, Z^k_k=-1}}{\Pr{W_{T_n}=-1}} \\
&=& \frac{n}{k} \frac{\Pr{
\V(Z^k) = \bx}}{\Pr{W_{T_n}=-1}}  = n\frac{ \Pr{
Z^k = \bx}}{ \Pr{W_{T_n}=-1}}=n\frac{ \Pr{
Z = \bx}}{ \Pr{W_{T_n}=-1}}\\
&=& \frac{n \Pr{Z \in \Sb(n)}}{ \Pr{W_{T_n}=-1}} \Pr{ Z = \bx \, |
\, Z \in \Sb(n)},
\end{eqnarray*}
where we have used (\ref{eq:v1}) for the third equality and (\ref{eq:V}) for the fourth equality.
Summing over all possible $\bx \in\Sb ^ k(n)$ and then over $k \geq n$, we get $\Pr{W_{T_n}=-1}=n \Pr{Z \in \Sb(n)}$. As a consequence, we have $\Pr{ \V(Z) =
\bx \,  | \, W_{T_n}=-1}=\Pr{ Z = \bx \, | \, Z \in \Sb(n)}$ for every $ \bx \in  \Sb(n)$. 

On
the other hand, by Proposition \ref{prop:RW}, for every $ \bx \in  \Sb(n)$,
$$
\Prmu{( \W_1(\tau), \ldots,
\W_{\zt}(\tau)- \W_{\zt-1}(\tau))=\bx \, | \, \lt=n} = \Pr {Z^ \zeta = \bx \, | \, \Lambda( \zeta)=n}$$
where we have used the notation $ \zeta$ introduced in \eqref {eq:notations}. The probability appearing in the right-hand side is equal to $ \Pr{Z=\bx \, |
\, \, Z \in \Sb(n)}$ because $  \{ \Lambda ( \zeta)=n\}=  \{Z\in \Sb(n)\}$, and moreover we have $ \zeta= T_n$  and $Z^ \zeta=Z$ on this event. We conclude that:
\begin{eqnarray*} \Prmu{( \W_1(\tau), \ldots,
\W_{\zt}(\tau)- \W_{\zt-1}(\tau))=\bx \, | \, \lt=n}&=& \Pr{Z=\bx \, |
\, \, Z \in \Sb(n)} \\
&=& \Pr { \V(Z)= \bx | W_ {T_n}=-1}.
\end{eqnarray*}This completes the proof.
\end{proof}

\begin{defn}Set $\D_0([0,1],\R)= \{ \omega \in \D([0,1],\R) ; \, \omega(0)=0\}$. The Vervaat transformation in continuous time, denoted by $\Vc: \D_0([0,1],\R) \rightarrow
\D([0,1],\R)$, is defined as follows. For $\omega \in
\D_0([0,1],\R)$, set $g_1(\omega)=\inf\{ t \in [0,1]; \omega(t-)
\wedge \omega(t)= \inf_{[0,1]} \omega\}$. Then define:
$$\Vc(\omega)(t)=\begin{cases} \omega(g_1(\omega)+t)-\inf_{[0,1]}
\omega, \qquad \qquad  \qquad \qquad \quad \, \, \, \, \, \textrm{if }
g_1(\omega)+t \leq 1,
\\
\omega(g_1(\omega)+t-1)+\omega(1)-\inf_{[0,1]} \omega \qquad \quad
\qquad \textrm{  if } g_1(\omega)+t \geq 1.
\end{cases}$$\end{defn}

\begin{cor}\label{cor:vervaatC}The law of $\left( \frac{1}{B_{\zt} } \W_{ \lfloor
\zt t \rfloor }( \tau);\, 0 \leq t \leq 1\right)$ under $\Pmuln$ coincides with the law of
$\Vc\left( \frac{1}{B_{T_n} } W_{ \lfloor T_n t \rfloor };\, 0
\leq t \leq 1\right)$ under $\Pr{ \, \cdot \, | \, W_{T_n}=-1}$.
\end{cor}

This immediately follows from  Proposition
\ref{prop:vervaatdiscrete}. In the next subsections, we first get a limiting result under $\Pr{ \, \cdot \, | \, W_{T_n}=-1}$ and then apply the Vervaat transformation using the preceding remark.
The advantage of dealing with $\Pr{ \, \cdot \, | \, W_{T_n}=-1}$
is to avoid any positivity constraint.

\subsection{Time Reversal}

The probability measure $\Pr{ \, \cdot
\, | \, W_{T_n}=-1}$ enjoys a time-reversal invariance property that will be useful in our applications.
Ultimately, as for the height and contour processes, this
time-reversal property will allow us to get the convergence of
rescaled Lukasiewicz paths over the whole segment $[0,1]$.

\begin{prop}\label{prop:timereversal} Fix two integers $m \geq n \geq 1$ such that $ \Pr {W_m=0, \Lambda(m)=n}>0$. For $ 0 \leq i \leq m$, set $\widehat{W}^{(m)}_i=W_{m}-W_{m-i}$. The law of the vector
$\left(W_0,\ldots,W_{m}\right)$ under $\Pr{ \, \cdot \, | \, W_{m}=0, \Lambda(m)=n}$ coincides with the law of the vector $\left(\widehat{W}^{(m)}_0,\ldots,\widehat{W}^{(m)}_{m}\right)$ under the same probability measure.
\end{prop}

\begin{proof}This is left as an exercise.\end{proof}

\subsection{The Lévy Bridge}

The Lévy bridge $X^\br$
can be seen informally as the path $(X_t; \, 0 \leq t \leq 1)$
conditioned to be at level zero at time one. See \cite[Chapter
VIII]{Bertoin} for definitions.

\begin{prop}\label{prop:bridge}The following two properties hold.\begin{enumerate}
\item[(i)] The continuous Vervaat transformation $\Vc$ is almost
surely continuous at $X^\br$ and $\Vc(X^\br)$ has the same distribution as $X$ under  $\Nn( \, \cdot \, | \, \z=1)$.
\item[(ii)] Fix $a \in (0,1)$. Let $F$ be a bounded continuous functional on
$\D([0,a],\R)$. We have:
$$\Es{F\left( X^{\br}_t; \, 0 \leq t \leq a \right)}=\Es{F\left( X_t; \, 0 \leq t \leq a \right)
\frac{p_{1-a}(-X_a)}{p_1(0)}}.$$ \end{enumerate}
\end{prop}

\begin{proof}The continuity of $\Vc$ at $X^\br$ follows from the
fact that the absolute minimum of $X^\br$ is almost surely attained
at a unique time. See \cite[Theorem 4]{Chaumont} for a proof of the
fact that $\Vc(X^\br)$ has the same distribution as $X$ under $\Nn(
\, \cdot \, | \z=1)$. For (ii), see \cite[Formula (8), chapter
VIII.3]{Bertoin}.
\end{proof}

\subsection{Absolute continuity and convergence of the Lukasiewicz path}

By means of a discrete absolute continuity argument similar to the
one used in Section 5, we shall show that for every $a \in (0,1)$ the law of $\left( \frac{1}{B_{T_n} }
W_{ \lfloor T_n t \rfloor };\, 0 \leq t \leq a\right)$  under
$\Pr{ \, \cdot \, | \, W_{T_n}=-1}$ converges to the law of $(X^\br_t, 0 \leq t \leq a)$.

\begin{lem}\label{lem:absolueC2} Fix $a \in (0,1)$ and let $n$ be a positive integer. To simplify notation, set
$W ^ {(u)}=(W_0, W_1,\ldots, W_ {T_ { \fl {u}}})$ for $u \geq 0$. For every function $f : \cup_ {i \geq 1} \Z^i \rightarrow \R_+$
 we have:
\begin{eqnarray*}
\Es{f(W^{(an)}) \,| \, W_{T_n} =- 1}=
\Es{f(W^{(an)}) \frac{\chi_{n- \fl{an}}(W_{T_  {an}})}{\chi_n(0)}},
\end{eqnarray*}
where $\chi_k(j)=\P_j[W_{T_k}=-1]$ for every $j \in \Z$ and $k \geq 1$, and $W$ starts from $j$ under the probability measure $ \P_j$.
\end{lem}
\begin{proof}This follows from the strong Markov property for the random walk $W$.\end{proof}

\begin{lem}\label{lem:technic3}For every $\a>0$, we have $ \d \lim_{n \rightarrow \infty} \sup_{ |j| \leq \a B_n}
\left|B_{{n}/{\mu_0}} \chi_n(j)-
p_1\left( -
\frac{j}{B_{{n}/{\mu_0}}}\right)
\right|=0$.
\end{lem}

 Note that $j$
is allowed to take negative values. Note also that Lemma \ref {lem:technic3} implies that
$\chi_n(0) \sim
{p_1(0)}/{B_{{n}/{\mu_0}}}$ as $n \rightarrow \infty$.

\begin{proof} Fix $\e \in (0,1)$. Using Lemma, \ref{lem:CVf2} (i), we have:
\begin{eqnarray*} \chi_n(j)=\P_ {j}[W_{T_n}=-1]&=& \Pr {W_ {T_n}=-j-1}\ \\
&=& \P\left[W_{T_n}=-j-1,
\left|T_n-\frac{n}{\mu_0}\right|\leq
 \e n\right]+oe_{1/2}(n)\\
&=& \sum_{\left| k - {n}/{\mu_0} \right| \leq \e n}
\P\left[W_{k}=-j-1, T_n=k\right] +oe_{1/2}(n)\\
&=& \mu_0 \sum_{\left| k - {n}/{\mu_0} \right| \leq \e n}
\P\left[W_{k-1}=-j, \Lambda(k-1)=n-1\right] +oe_{1/2}(n)
\end{eqnarray*}

Recall that $S_n$ stands for the sum of $n$ iid Bernoulli random
variables of parameter $\mu_0$ and that $W'$ is the random walk $W$
conditioned on having nonnegative jumps.  By \eqref{eq:LambdaUtile}:
\begin{eqnarray*}&&\P_ {j}[W_{T_n}=-1]\\
 &&\qquad  =  \mu_0 \sum_{\left| k - {n}/{\mu_0} \right|
\leq \e n}
 \P[S_{k-1}=n-1]\P\left[W'_{k-n}=n-j-1\right] +oe_{1/2}(n)
 \\
 && \d \qquad =  \mu_0     \d \int_ {n/ \mu_0 - \epsilon n + O(1)} ^{n/ \mu_0 + \epsilon n + O(1)} dx \,  \Pr {S_ { \fl {x-1}} =n-1} \Pr {W'_ { \fl {x}-n}=n-j-1} +oe_{1/2}(n). \\
 &&\qquad  = \mu_0 \int_{- \e \sqrt{n}+o(1)} ^{\e \sqrt{n}+o(1)} du \, \sqrt{n} \, \Pr{S_{\fl{\frac{n}{\mu_0}+u
 \sqrt{n}-1}}=n-1}  \P\left[W'_{\fl{\frac{n}{\mu_0}+u
 \sqrt{n}}-n}=n-j-1\right]+oe_{1/2}(n).
 \end{eqnarray*}
 For $ |u| \leq \epsilon \sqrt {n}$, set:
 $$ \widetilde{a}_n(u)=\sqrt{n} \, \Pr{S_{\fl{\frac{n}{\mu_0}+u
 \sqrt{n}-1}}=n-1}.$$
 Using the notation of Section \ref{sec:technik}, we have then:
 $$B'_n \P_ {j}[W_{T_n}=-1]= \mu_0 \int_{- \e \sqrt{n}+o(1)} ^{\e \sqrt{n}+o(1)} du \,  \widetilde{a}_n(u) b_n(u,j+1)+oe_{1/2}(n).$$
 The same argument that led us to \eqref{eq:lemtechnique0} gives that:
 $$\sup_{ |j| \leq \a B_n} \left| B'_n \P_ {j}[W_{T_n}=-1]-  c^ {1/ \theta}p_1\left(-\frac{j}{B_{n/\mu_0}}\right) \right| \quad 
 \mathop{\longrightarrow}_{n \rightarrow \infty} \quad 0.$$
The conclusion follows from that fact that $B'_n  / B_ {n/ \mu_0} \rightarrow c^ {1/ \theta}$.\end{proof}

All the necessary ingredients have been gathered and we can now turn to the
proof of the theorem.

\begin{proof}[Proof of Theorem \ref{thm:cvW3}]  Let $F: \D([0,a],\R) \rightarrow \R_+$ be a  bounded continuous function. Fix $a \in (0,1)$ and $\a>0$. To
simplify notation, set $A^\a_n= \left\{ \left|W_{ T_ {an}}\right| \leq
\a B_{n/ \mu_0} \right\}$ and $G^ {(n)}(W)=F\left( \frac{1}{B_{n/ \mu_0}} W_{ \fl{ \frac{t}{a} T_ {an} }}; \, 0 \leq t \leq
a\right)$. We apply Lemma \ref{lem:absolueC2} with $f(W_0,W_1, \ldots, W_ { T_  {an}})=  G^ {(n)}(W) 1_{ {A^\a_n}}$ and get:
$$\Es{G^ {(n)}(W) 1_{ {A^\a_n}} \, | \, W_{T_n}=-1}
= \Es {G^ {(n)}(W) 1_{ {A^\a_n}} \frac{\chi_{n-\fl{an}}(W_{T_ {an}})}{\chi_n(0)}}.$$
Lemma \ref{lem:technic3} then entails:
$$\lim_{n \rightarrow \infty}
\left|\Es{G^ {(n)}(W) 1_{ {A^\a_n}} \, | \, W_{T_n}=-1}
- \Es {G^ {(n)}(W) 1_{ {A^\a_n}}  \frac{(1-a)^ {-1/ \theta}}{p_1(0)} p_1 \left( - \frac{W_{T_ {an}}}{B_ {(n- \fl {an})/ \mu_0}}\right) } \right|=0$$
 From Lemma \ref{lem:CVf2} (ii), we deduce that:
$$\lim_{n \rightarrow \infty}
\Es{G^ {(n)}(W) 1_{ {A^\a_n}} \, | \, W_{T_n}=-1}
= \Es {F((X_t)_ {0 \leq t \leq a})1_{ \{ |X_a|< \a\}}  \frac{(1-a)^ {-1/ \theta}}{p_1(0)} p_1\left(-\frac{X_a}{(1-a)^{1/\theta}}\right) }.$$
By (\ref{eq:scaling}), we have ${(1-a)^{-1/\theta}}
p_1\left(-\frac{X_a}{(1-a)^{1/\theta}}\right)= p_{1-a}(-X_a)$.
Consequently, by Proposition \ref{prop:bridge} (ii), we conclude
that:
\begin{equation}
\label{eq:cvBr}\lim_{n \rightarrow \infty}
\Es{G^ {(n)}(W) 1_{ {A^\a_n}} \, | \, W_{T_n}=-1}
= \Es{F(X^\br_t; \ 0 \leq t \leq a) 1_{\{
|X^\br_a| \leq \a\}}}.\end{equation}
 By taking $F\equiv 1$, we obtain: \begin{equation}\label{eq:trucsuperutile}\lim_{\a \rightarrow
\infty }\lim_{n \rightarrow \infty} \Pr{A^\a_n \, | \,
W_{T_n}=-1}=1.\end{equation}
By choosing $ \a>0$ sufficiently large, we easily deduce from the convergence \eqref{eq:cvBr} that:
\begin{eqnarray*}\lim_{n \rightarrow \infty}
\left|\Es{F\left(\frac{1}{B_{n/ \mu_0}}
W_{\fl{\frac{t}{a} T_ {an} }}; \, 0 \leq t \leq a\right)  \, | \,
W_{T_n}=-1}-\Es{F(X^\br_t; \ 0 \leq t \leq a)} \right|&=&0.
\end{eqnarray*}
Next write: $$\frac{1}{B_{T_n} } W_{ \lfloor T_n t\rfloor }=K_n \cdot \frac{1}{B_{n/ \mu_0}}
W_{\fl{S_n  \cdot \frac{t}{a} T_{an}}},$$
where $K_n=B_{n/ \mu_0}/B_{T_n} $ and $S_n=a T_n / T_{an}$. Lemma \ref{lem:CVf2} (i) entails that $K_n$ and $S_n$ both converge in probability towards $1$. Lemma \ref{lem:CVproba} then implies that  the law of $
\left(\frac{1}{B_{T_n} } W_{ \lfloor T_n t \rfloor }; 0 \leq t \leq a \right)$ under $\Pr{ \, \cdot \, | \, W_{T_n}=-1}$ converges to the law of $(X^\br_t, 0 \leq t \leq a)$,  and this holds for every $a \in (0,1)$. 

We now show that the latter convergence holds also for $a=1$ by using a time-reversal argument based on Proposition \ref{prop:timereversal}. By the usual tightness criterion
(see e.g. \cite[Formula (13.8)]{Bill}), it is sufficient to show that, for every $ \eta>0$,
\begin{equation}\label{eq:montrertension}\lim_{\delta \rightarrow 0} \limsup_{n
\rightarrow \infty} \Pr{ \left. \sup_{s \in [1- \delta,1)}
\left|\frac{1}{B_{T_n} } W_{ \lfloor T_n s \rfloor }\right| > \eta
\, \right| \, W_{T_n}=-1}=0. \end {equation} Note that: $$\sup_{s \in [1- \delta,1)}
\left|\frac{1}{B_{T_n} } W_{ \lfloor T_n s \rfloor }\right|=\sup_{
\fl{(1-\delta)T_n} \leq k \leq T_n-1 } \left|\frac{1}{B_{T_n} }
W_{ k }\right|.$$ Using this remark, we  write:
\begin{eqnarray*}
&&\Pr{ \left. \sup_{s \in [1- \delta,1)} \left|\frac{1}{B_{T_n} } W_{
\lfloor T_n s \rfloor }\right| > \eta \, \right| \,
W_{T_n}=-1}\\
&& \qquad  = \frac{1}{\Pr{W_{T_n}=-1}} \sum_{k=n}^{\infty} \Pr{
T_n=k, W_k=-1,\, \sup_{s \in [1- \delta,1)} \left|\frac{1}{B_{k} } W_{
\lfloor k s \rfloor }\right| >
\eta}\\
&& \qquad  = \frac{\mu_0}{\Pr{W_{T_n}=-1}} \sum_{k=n}^{\infty} \Pr{
\Lambda(k-1)=n-1, \, W_{k-1}=0, \, \sup_{s \in [1- \delta,1)}
\left|\frac{1}{B_{k} } W_{ \lfloor k s \rfloor
}\right| > \eta}\\
&& \qquad  \leq \frac{\mu_0}{\Pr{W_{T_n}=-1}} \sum_{k=n}^{\infty}
\Pr{ \Lambda(k-1)=n-1, \, W_{k-1}=0, \, \sup_{s \in [0,\delta+1/k]}
\left|-\frac{1}{B_{k} } W_{ \lfloor k s \rfloor
}\right| > \eta}\\
&& \qquad =\Pr{ \left.\sup_{s \in [0,\delta+1/T_n]}
\left|\frac{1}{B_{T_n} } W_{ \lfloor T_n s \rfloor }\right| > \eta
\, \right| \, W_{T_n}=-1},\end{eqnarray*} using
Proposition \ref{prop:timereversal} in the upper bound of the last display.  \eqref{eq:montrertension} then follows from the fact that that  the law of $
\left(\frac{1}{B_{T_n} } W_{ \lfloor T_n t \rfloor }; 0 \leq t \leq a \right)$ under $\Pr{ \, \cdot \, | \, W_{T_n}=-1}$ converges to the law of $(X^\br_t, 0 \leq t \leq a)$ for every $ a \in (0,1)$. We conclude that this convergence also holds for $a=1$.

We then combine the continuous Vervaat transformation $\Vc$ with the latter
convergence. Since $\Vc$ is almost surely continuous at $X^\br$
(Proposition \ref{prop:bridge} (i)), we get that the law of
$$ \Vc\left(\frac{1}{B_{T_n} } W_{ \lfloor T_n
t \rfloor } ;\, 0 \leq t \leq 1 \right)$$
under $\Pr{ \, \cdot \,
| \, W_{T_n}=-1}$ converges to the law of $\Vc(X^\br)$. Corollary \ref{cor:vervaatC}
and Proposition \ref{prop:bridge} (i) entail:
$$\left(\frac{1}{B_{\z(\t_n)} } \W_{ \lfloor \z(\t_n) t \rfloor }(\t_n);\, 0 \leq t \leq 1
\right) \quad
 \mathop{\longrightarrow}^{(d)}_{n \rightarrow \infty} \quad (X_t; 0 \leq t \leq 1) \textrm{ under } \Nn( \, \cdot \, |
\, \z=1).$$ This completes the proof.
\end{proof}

\section{Application: maximum degree in a Galton-Watson tree
conditioned on having many leaves}

In this section, we study the asymptotic behavior of the
distribution of the maximum degree in a Galton-Watson tree
conditioned on having $n$ leaves. To
this end, we use tools introduced in Section $6$ such as the Vervaat
transformation and absolute continuity arguments.

As earlier, we fix $ \theta \in (1,2]$ and suppose that $ \mu$ is a probability distribution satisfying the hypothesis $ ( H_ \theta)$. For every $n\geq 1$ such that $\Pmu[ \lt = n]>0$, let also $\t_n$
be a random tree distributed according to $\Pmu[\, \cdot \, | \, \lt = n]$. If $ \tau \in \mathbb{T}$ is a tree, let $\Delta( \tau)=\max\{k_u; \, u \in
\tau\}$ be the maximum number of children of individuals of $ \tau$. We are interested in the asymptotic behavior of $ \Delta ( \t_n)$.

\bigskip

The case $ 1<\theta < 2$ easily follows from previous results. Indeed, let $(B_n) _ { n \geq 1}$ be defined as before Lemma \ref {lem:locallimit}. Then, by Theorem \ref{thm:cvW3} and Remark \ref{rem:changement}:
\begin{equation}\label{eq:thmcvW32}\left(\frac{ \mu_0 ^ {1/\theta}}{B_{n} } \W_{ \lfloor \z(\t_n) t \rfloor
}(\t_n);\, 0 \leq t \leq 1 \right) \quad
 \mathop{\longrightarrow}^{(d)}_{n \rightarrow \infty} \quad \X .\end{equation} 
 If $Z \in D([0,1], \R)$, let $ \overline{\Delta}(Z)$ be the largest jump of $Z$. Note that by construction, $\overline{ \Delta} ( \W( \t_n))) = \Delta( \t_n)-1$. Since $\overline{\Delta}$ is a continuous functional on $D([0,1], \R)$, \eqref{eq:thmcvW32} immediately gives that $ \mu_0 ^ {1/ \theta}  \overline{\Delta}(X)/B_n$ converges in distribution towards $\overline{\Delta}( \X)$, which is almost surely positive.

However, in the case $ \sigma^2< \infty$, $ \overline{\Delta}( \X)=0$ almost surely since $ \X$ is continuous. It is natural to ask whether the suitably rescaled sequence $ \Delta( \t_n)$ converges to a non-degenerate limit. A similar question has been previously studied by Meir \& Moon \cite{MM} when $\t_n$ is distributed according to $ \Pmuzn$. We shall make the same  assumptions on $ \mu$ as Meir \& Moon.

More precisely, let $ \nu$ be a critical aperiodic probability distribution on $ \N$ with finite variance. Let $R$ be the radius of convergence of $ \sum \nu(i) z^i$.  We say that $ \nu$ satisfies hypothesis $ \mathcal {H}$ if the following two conditions hold:  $R>1$ and if $R< \infty$, $ \nu(n)^ {1/n}$ converges towards $1/R$ as $n \to \infty$, if $ R= \infty$ there exists $N \geq 0$ such that the sequence $(\nu(k) ^ {1/k})_ {k \geq N}$ is decreasing.

\begin{thm}\label{thm:max} \textrm{ }
\begin{enumerate}
 \item[(i)]If $ 1 < \theta <2$, we have $ \d { \mu_0 ^ {1/ \theta} \Delta ( \t_n)}/{B_n}  \quad
\mathop{\longrightarrow}_{n \rightarrow \infty} ^ {(d)} \quad \overline{\Delta}( \X)$.
 \item[(ii)]
Set $D(n)= \max\{ k \geq 1; \,  \mu([k,\infty)) \geq
1/n\}$. If $ \sigma^2 < \infty$, under the additional assumption  that $ \mu$ satisfies hypothesis $ \mathcal {H}$, we have for every $ \e>0$:
$$\P[(1-\e) D(n) \leq \Delta(\t_n) \leq (1+\e) D(n)] \quad
\mathop{\longrightarrow}_{n \rightarrow \infty} \quad 1.$$
  \end{enumerate}
\end{thm}

Part (i) of the theorem follows from the preceding discussion. It remains to prove (ii). We suppose that $ \mu$ satisfies the assumptions in (ii). The first step is to control the asymptotic
behavior of $D(n)$. 

\begin{lem}[Meir \& Moon]\label{lem:MM}Let $ \epsilon>0$. For $n$ sufficiently large:
$$ \mu([(1-\e) D(n),\infty)) \geq n^{-\frac{1}{1+\e/3}}, \qquad \mu([(1+\e) D(n),\infty)) \leq n^{-1-\e/3}.$$\end{lem}

\begin{proof}See the proof of Theorem 1 in \cite{MM}, which uses the different assumptions made on $ \mu$. \end{proof}

\begin{proof}[Proof of Theorem \ref{thm:max} in the case $ \sigma^2 < \infty$]The idea of the proof consists in showing that if the
Lukasiewicz path of a non-conditioned Galton-Watson tree satisfies
asymptotically some property which is invariant under cyclic-shift (with some
additional monotonicity condition), then the Lukasiewicz path of a
conditioned Galton-Watson tree satisfies asymptotically the same
property.

 We first establish the lower bound. Recall the
notation introduced in (\ref{eq:notations}). If $\mathbf{u}=(u_1,\ldots,u_k) \in \Z^k$, set $\mathcal{M}(\mathbf{u})=\max_{1 \leq i \leq k} u_i$, 
so that
$\Delta(\t_n)=\mathcal{M}( \W_1(\t_n)- \W_0(\t_n),\ldots,\W_{\z(\t_n)}(\t_n)- \W_{\z(\t_n)-1}(\t_n))+1$.
Note that $\mathcal{M}$ is invariant under cyclic shift. Set
$p_n=(1-\e) D(n)$. To simplify notation, for $u_1,\ldots,u_k \in \Z$
set $F ^ {(n)}(u_1,\ldots,u_k)=1_{\{\mathcal{M}(u_1,\ldots,u_k) < p_n\}}$.
We have:
\begin{eqnarray} \Pr { \Delta ( \t_n) < p_n+1}&=&\Esmu{F^ {(n)}(
\W_1(\tau)- \W_0(\tau),\ldots,\ldots, \W_{\zt}(\tau)- \W_{\zt-1}(\tau))) \, | \, \lt=n}\notag\\
&=& \Es{F^ {(n)}(\V(W_1,W_2-W_1,\ldots,W_{T_ {n}}-W_{T_n-1})) \, |
\, W_{T_n}=-1}\notag\\
&=& \Es{F^ {(n)}(W_1,W_2-W_1,\ldots,W_{T_n}-W_{T_n-1}) \, | \, W_{T_n}=-1}\label{eq:Deltanon},\end{eqnarray} where
we have used Proposition \ref{prop:vervaatdiscrete} in the first
equality, and the fact that $F^ {(n)}(\V(\mathbf{u}))=F^ {(n)}(\mathbf{u})$ for
every $ \mathbf{u} \in \Z^k$ ($k \geq 1$) in the second one.  To simplify notation, we put $F_k^ {(n)}(W)=F_n(W_1,W_2-W_1, \ldots,W_k-W_ {k-1})$. Note that $\Es{F_ { T_n }^ {(n)}(W) \, | \, W_{T_n}=-1} \leq \Es{F_ { T_ {n/2} }^ {(n)}(W) \, | \, W_{T_n}=-1}$. In order to establish the lower bound in Theorem \ref {thm:max} (ii), it then suffices to prove that $ \Es{F_ { T_ {n/2} }^ {(n)}(W) \, | \, W_{T_n}=-1}$ tends to $0$ as $n \rightarrow \infty$. 	Let $ \alpha>0$, and let the
event $A_n^{\alpha}$ be defined by $$ A_n^{\alpha}= \left\{ |W_{T_{ n/2}}| < \alpha \sigma
\sqrt{n/(2\mu_0)} \right\},$$ where $\sigma^2$ is the variance of $\mu$. By Lemma \ref{lem:CVf2} (i), we have:
\begin{eqnarray*}&& \Es{F_ { T_ {n/2} }^ {(n)}(W) \, | \, W_{T_n}=-1}  \\
 && \qquad \qquad \leq  \Es{1_{\{A_n^{\alpha}\}^c} \, | \,
W_{T_n}=-1}+\Es{F_ { T_ {n/2} }^ {(n)}(W) 1_{\{A_n^{\alpha}, \frac{n}{4\mu_0} \leq T_{{n}/{2}} \leq
\frac{n}{\mu_0} \}} \, | \, W_{T_n}=-1}+oe_ {1/2}(n)
\end {eqnarray*}
By  Lemma \ref{lem:absolueC2}:
\begin{eqnarray*}&& \Es{F_ { T_{n/2} }^ {(n)}(W) 1_{\{A_n^{\alpha}, \frac{n}{4\mu_0} \leq T_{{n}/{2}} \leq
\frac{n}{\mu_0} \}} \, | \, W_{T_n}=-1} = \Es{ F_ { T_{n/2} }^ {(n)}(W)
1_{\{A_n^{\alpha}, \frac{n}{4\mu_0} \leq T_{{n}/{2}} \leq
\frac{n}{\mu_0} \}} \frac{\chi_{n- \fl{n/2}}(W_{T_{n/2}})}{\chi_n(0)}},
\end{eqnarray*}
where $\chi_n(j)=\P_j[W_{T_n}=-1]$. By
Lemma \ref{lem:technic3}, there exists $C>0$ such that for every $n$ large enough,
${\chi_{n- \fl{n/2}}(W_{T_{n/2}})}/{\chi_n(0)} \leq C$ on the event $A_n^{\alpha}$. By combining the previous observations, we get:
\begin{eqnarray}&& \Es{F_ { T_ {n/2} }^ {(n)}(W) \, | \, W_{T_n}=-1} \notag \\
&& \qquad  \qquad \leq \Es{1_{\{A_n^{\alpha}\}^c} \, | \,
W_{T_n}=-1} +   \ C\Es{F_ { T_{n/2} }^ {(n)}(W)1_{\{\frac{n}{4\mu_0} \leq T_{{n}/{2}} \leq
\frac{n}{\mu_0} \}}}+oe_{1/2}(n). \label {eq:i}\end{eqnarray}By (\ref{eq:trucsuperutile}), we have:
$$\lim_{\a \rightarrow \infty} \lim_{n \rightarrow \infty} \Es{1_{\{A_n^{\alpha}\}^c} \,
| \, W_{T_n}=-1}=0.$$ Let us finally show that the second term in the right-hand side of \eqref{eq:i}
tends to zero as well. We have:
\begin{eqnarray*}\Es{F_ { T_{n/2} }^ {(n)}(W) 1_{\{\frac{n}{4\mu_0} \leq T_ {n/2} \leq \frac{n}{\mu_0}\}}} &\leq& \Pr{
\mathcal{M}(W_1,W_2-W_1,\ldots,W_{ \fl{n/4\mu_0}}-W_{
\fl{n/4\mu_0}-1})
< p_n}\\
&=& \P[W_1 <
p_n]^{ \fl{n/4\mu_0}}=\left(1-P[W_1 \geq
p_n]\right)^{ \fl{n/4\mu_0}}.\end{eqnarray*} The first part of Lemma \ref{lem:MM},
implies that the last quantity tends to $0$ as $n \rightarrow \infty$. By combining the previous estimates, we conclude that $\P[(1-\e) D(n) \geq  \Delta(\t_n) ]
\rightarrow 0$ as $n \rightarrow \infty$.

Let us now establish the upper bound. Set $q_n=(1+\e) D(n)$. By
an argument similar to the one we used to establish (\ref{eq:Deltanon}), we get $\P[\Delta(\t_n) > q_n+1] =\P[\mathcal{M}(W_1,\ldots,W_{T_n}-W_{T_n-1}) > q_n \, | \,
W_{T_n}=-1]$. It follows that:
\begin{eqnarray*}\P[\Delta(\t_n) > q_n+1] &\leq&
\P[\mathcal{M}(W_1,W_2-W_1,\ldots,W_{T_{n/2}}-W_{T_{n/2}-1}) >
q_n \,
| \, W_{T_n}=-1]\\
&&     +\P[\mathcal{M}(W_{T_{n/2}}-W_{T_{n/2}-1},\ldots,W_{T_n}-W_{T_n-1}) > q_n \, | \,
W_{T_n}=-1]\end{eqnarray*} By a time-reversal argument based on
Proposition \ref{prop:timereversal}, it is sufficient to show that
the first term of the last expression tends to 0. To this end, we
use the same approach as for the proof of the lower bound,
taking this time $F^  {(n)}_k(W)=1_{\{\mathcal{M}(W_1,\ldots,W_k-W_ {k-1}) > q_n\}}$.
It is then sufficient to verify that:
$$\Es{F_ { T_{n/2} }^ {(n)}(W)1_{\{\frac{n}{4\mu_0} \leq T_{n/2} \leq \frac{n}{\mu_0} \}}}
\quad \mathop{\longrightarrow}_{n \rightarrow \infty} \quad 0.$$ To
this end, write:
\begin{eqnarray*}\Es{F_ { T_{n/2} }^ {(n)}(W)1_{\{\frac{n}{4\mu_0} \leq T_{n/2} \leq \frac{n}{\mu_0} \}}}
& \leq& \P[\mathcal{M}(W_1,W_2-W_1,\ldots,W_{ \fl{n/\mu_0}}-W_{ \fl{n/\mu_0}-1}))
> q_n]\\
&=& \quad 1-\left( 1- \P[W_1 >
q_n]\right)^{\fl{n/\mu_0}}\end{eqnarray*} which tends to $0$ as $n
\rightarrow \infty$ by Lemma \ref{lem:MM}. By combining the previous estimates, we conclude that $\P[(1+\e) D(n) \leq  \Delta(\t_n) ]
\rightarrow 0$ as $n \rightarrow \infty$. This completes the proof of the theorem.
\end{proof}

\begin{rem}\label{rem:uniformtree}In particular cases, it is possible to obtain better bounds in the previous theorem.
Let $\mu$ be defined by $\mu(0)=2-\sqrt{2}$,
$\mu(1)=0$ and $\mu(i)=\left( ({2-\sqrt{2}})/{2}\right)^{i-1}$ for
$i \geq 2$ (this probability distribution appears when we consider the tree associated with a uniform dissection of the $n$-gon, see \cite {CK}). One verifies that $\mu$ is a critical probability
measure. For $n \geq 1$, let $\t_n$ be a random tree distributed
according to $\Pmu[\, \cdot \, | \, \lt  = n]$. One easily checks that $ \mu$ is the unique critical
probability measure such that $\t_n$ is distributed uniformly over the set of all rooted plane trees with $n$ leaves such
that no vertex has exactly one child. In this particular case,
Theorem \ref{thm:max} (ii) can be strengthened as follows:
\begin{equation}
\label{eq:better}
\P[\log_b n - c \log_b \log_b n \leq \Delta(\t_n) \leq \log_b n + c \log_b \log_b n]
\quad \mathop{\longrightarrow}_{n \rightarrow \infty} \quad1,\end{equation}
for every $c>0$,
where $b={1}/{\mu(2)}=\sqrt{2}+2$. Indeed, the proof of Theorem \ref {thm:max} shows that it is sufficient to
verify that for every $ c>0$: $$\left(1-P[W_1 \geq \log_b n - c \log_b \log_b
n]\right)^{n/4\mu_0}\mathop{\longrightarrow}_{n \rightarrow
\infty}0, \qquad \left(1-P[W_1 \geq \log_b n + c \log_b \log_b
n]\right)^{n/4\mu_0}\mathop{\longrightarrow}_{n \rightarrow
\infty}1.$$ These asymptotics are easily obtained since the
probabilities appearing in these two expressions can be calculated
explicitly.

The convergence \eqref{eq:better} yields an interesting application to the maximum face degree in a uniform dissection (see \cite[Prop. 3.5]{CK}).
\end{rem}

\section{Extensions}

Recall that if $ \mathcal{A}$ is a non-empty subset of $\N$ and $ \tau$ a tree, $ \zeta_ \mathcal{A}( \tau)$ is the total number of vertices $u \in \tau$ such that $ k_u ( \tau) \in \mathcal{A}$.  For a forest $ \bf$,  $ \zeta_ \mathcal{A}( \bf)$ is defined in a similar way. In this section, we extend the results  (I) and (II) appearing in the Introduction to the case where  $ \mathcal{A} \neq  \{0\}$. By slightly adapting the previous techniques, it is possible to obtain the following more general result. 

 Recall that $ \mu$ is a probability distribution on $ \N$ satisfying the hypothesis $ (H _ \theta)$ for some $ \theta \in (1,2]$. We also consider the slowly varying function $h$ and the sequence $(B_n)_ { n \geq 1}$ introduced just before Lemma \ref{lem:locallimit}.

\begin {thm} \label {thm:general} Let $ \mathcal {A}$ be a non-empty subset of $\N$. If $\mu$ has infinite variance, suppose in addition that either $ \mathcal {A}$ is finite, or $ \N \backslash \cal {A}$ is finite.
\begin{enumerate}
\item[(I)] Let $d \geq 1$ be the largest integer such
that there exists $b \in \N$ such that supp$(\mu)\backslash  \cal {A}$ is
contained in $b+d \Z$, where supp$(\mu)$ is the support of $\mu$.
Then :
$$ \Prmu{\zeta_ \mathcal{A}( \tau)=n} \quad \mathop{\sim}_{n \rightarrow \infty} \quad \mu( \mathcal{A})^{1/\theta} p_1(0) \frac{\gcd(b-1,d)}{h(n) n^{ 1/\theta+1}}$$
for those values of $n$ such that $ \Prmu{\zeta_ \mathcal{A}( \tau)=n}>0$. 
\item[(II)] For every $n\geq 1$ such that $\Prmu{\zeta_ \mathcal{A}( \tau)=n} >0$, let $\t_n$ be a random tree distributed according to $\Pmu[\, \cdot \, | \,  \zeta_ \mathcal{A}( \tau)=n]$. Then
$$\left( \frac{1}{B_{\z(\t_n)} } \W_{ \lfloor \z(\t_n) t \rfloor
}(\t_n),\frac{B_{\z(\t_n)}}{\z(\t_n)}C_{2 \z(\t_n) t}(\t_n),
\frac{B_{\z(\t_n)}}{\z(\t_n)}H_{\z(\t_n) t}(\t_n)\right)_{0 \leq t
\leq 1}$$ converges in distribution to $ ( \X, \H, \H)$ as $n \rightarrow \infty$.
\end{enumerate}
\end {thm}

Theorem \ref {thm:general} can be established by the same arguments used to prove Theorems \ref{thm:ltn}, \ref{thm:cvG} and \ref{thm:cvW3}. The main difference comes from the proof of the needed extension of Lemma \ref {lem:technical} (ii), which is more technical. Let us explain the argument leading to the convergence
\begin{equation}
\label{eq:gen} \lim_{n \rightarrow \infty } \sup_{1 \leq j \leq \a B_n} \left| n \Prmuj{\z_ { \cal {A}}(\bf)=n}-
q_1\left(\frac{j}{B_{n/\mu( \mathcal {A})}}\right) \right|=0.
\end{equation}

The first step is to generalize Proposition \ref{prop:K} and find the joint law of $(\z(\bf), \z_ { \mathcal{A}}(\bf))$ under $ \Pmuj$ (which is the contents of the latter proposition in the case  $ \mathcal {A}=  \{ 0 \}$). To this end, let $ \rho$ and $ \mu'$ be the two probability measures on $ \N \cup  \{-1\}$ defined by:
$$  \rho(i)= \begin {cases} \frac{\mu(i+1)}{\mu ( \mathcal{A})} &\textrm {if } i+1  \in \mathcal{A} \\
0 &\textrm {otherwise} \end {cases}, \qquad
 \mu'(i)=\begin {cases} \frac{\mu(i+1)}{ 1-\mu ( \mathcal{A})}  &\textrm {if } i+1  \not \in \mathcal{A} \\
0 & \textrm {otherwise} .\end {cases}
$$
It is then straightforward to adapt Proposition \ref{prop:K} and get that:
$$\Pmuj[\z(\bf)=p, \z_ { \mathcal{A}}(\bf)=n]=\frac{j}{p} \, \P[S_p=n] \, \P[W'_{p-n}=-U_n-j].$$
where $S_p$ is the sum of $p$ independent Bernoulli random variables
of parameter $\mu ( \mathcal{A})$, $(W'_n)_ {n \geq 1}$ is the random walk started from $0$
with jump distribution $ \mu'$ and $(U_n) _ {n \geq 1}$ is an independent random walk started from $0$
with jump distribution $ \rho$. Note that $-U_n=n$ when $\mathcal{A}=  \{0\}$. 

First suppose that $ \mu$ has finite variance. Then both $W'_1$  and $U_1$ have finite variance. As in the proof of Lemma \ref {lem:technical}, we have, for $ 0 < \epsilon <1$:
\begin{equation}
\label{eq:gen2} n \Prmuj{\z_ { \cal {A}}(\bf)=n} = \d n \d \int_ {n/ \mu( \mathcal {A}) - \epsilon n + O(1)} ^{n/ \mu( \mathcal {A}) + \epsilon n + O(1)} dx \frac{j}{ \fl {x}} \Pr {S_ { \fl {x}} =n} \Pr {W'_ { \fl {x}-n}=-U_n-j} + oe_ {1/2}(n).
\end{equation}
By the law of large numbers, we can suppose that for $n$ sufficiently large, $ \Pr { \left| U_n-n \Es{U_1}\right|> \epsilon  n} < \epsilon$. Set $t_n(v)=\fl { n \Es {U_1} + v \sqrt {n}}$ for $n \geq 1$ and $ v \in \R$. It follows that:
$$ \left|\Pr {W'_ { \fl {x}-n}=-U_n-j}-  \int_ {- \epsilon  \sqrt {n}+ o(1)} ^ { \epsilon \sqrt {n}+o(1)}dv  \sqrt {n}\, \Pr {W'_ { \fl {x}-n}= -t_n(v) -j} \Pr {U_n=t_n(v)} \right| \leq  \epsilon.$$
The local limit theorems give bounds and estimates for the quantities $\Pr {W'_ { \fl {x}-n}= -t_n(v) -j}$ and $\Pr {U_n=t_n(v)}$. As previously, we can then use the dominated convergence theorem to obtain an estimate of $\Pr {W'_ { \fl {x}-n}=-U_n-j}$ as $n \rightarrow \infty$.  We substitute this estimate in \eqref{eq:gen2} and using once again the dominated convergence theorem we obtain \eqref{eq:gen}.

Now suppose that $ \mu$ has infinite variance and that $ \mathcal {A}$ is finite. Then $W'_1$ is in the domain of attraction of a stable law of index $ \theta$ and $U_1$ has bounded support hence finite variance. The proof of \eqref{eq:gen} then goes along the same lines as in the finite variance case.

When $ \mu$ has infinite variance and $ \N \backslash \mathcal {A}$ is finite, $W'_1$ has finite variance and $U_1$ is in the domain of attraction of a stable law of index $ \theta$. The proof of \eqref{eq:gen} goes along the same lines as when $ \mu$ has finite variance by interchanging the roles of $W'$ and of $U$ (see \cite {Kautre} for details in the case $ \mathcal {A}= \N$).

\begin {tabular}{l }
Laboratoire de mathématiques, UMR 8628 CNRS. \\
Université Paris-Sud\\
91405 ORSAY Cedex, France
\end {tabular}

\medbreak
\noindent \texttt{igor.kortchemski@normalesup.org}

\end{document}